\documentclass[a4paper,11pt]{article}

\usepackage[utf8]{inputenc} 
\usepackage[T1]{fontenc}
\usepackage[english]{babel}
\usepackage{amsmath, amsfonts, amsthm,amssymb,here,dsfont, hyperref}
\usepackage{mathtools}
\usepackage{graphicx,color,subfigure,multirow, here}
\usepackage{natbib}
\usepackage{enumitem}
\usepackage{comment}

\newtheorem{theo}{Theorem}[section]

\newtheorem{prop}[theo]{Proposition}

\newtheorem{coro}[theo]{Corollary}
\newtheorem{lemme}[theo]{Lemma}
\newtheorem{remark}[theo]{Remark}

\newtheorem{assumption}[theo]{Assumption}

\newcommand{\w}{\widehat}

\newcommand{\one}{\mathds{1}}

\newcommand{\E}{\mathbb{E}}

\renewcommand{\P}{\mathbb{P}}

\newcommand{\R}{\mathbb{R}}

\title{Nonparametric estimation of the diffusion coefficient from \textcolor{black}{i.i.d.} S.D.E. paths}
\author{Eddy Ella-Mintsa
\\
\small{LAMA, Université Gustave Eiffel.}}

\begin{document}

\maketitle

\begin{abstract} Consider a diffusion process $X=(X_t)_{t\in[0,1]}$ observed at discrete times and high frequency, solution of a stochastic differential equation whose drift and diffusion coefficients are assumed to be unknown. In this article, we focus on the nonparametric \textcolor{black}{estimation} of the diffusion coefficient. We propose ridge estimators of the square of the diffusion coefficient from discrete observations of $X$ that are obtained by minimization of the least squares contrast. We prove that the estimators are consistent and derive rates of convergence as the \textcolor{black}{number of} observations tends to infinity. Two observation schemes are considered in this paper. The first scheme consists in one diffusion path observed at discrete times, where the discretization step of the time interval $[0,1]$ tends to zero. The second scheme consists in repeated observations of the diffusion process $X$, where the number of the observed paths tends to infinity. The theoretical results are completed with a numerical study over synthetic data.
\end{abstract}



\section{Introduction}

Let $X=(X_t)_{t\in[0,1]}$ be a one dimensional diffusion process with finite horizon time, solution of the following stochastic differential equation:
\begin{equation}
    \label{eq:model}
    dX_t=b(X_t)dt+\sigma(X_t)dW_t, \ \ X_0=\textcolor{black}{x_0 \in \mathbb{R}}
\end{equation}
where $(W_t)_{t\geq 0}$ is a standard Brownian motion. \textcolor{black}{The drift function $b$ and the diffusion coefficient $\sigma$ are assumed to be unknown. }

\textcolor{black}{For a number of years now, there has been a great deal of interest in stochastic processes that are solutions of stochastic differential equations, also known as diffusion processes. This mathematical tool has been the subject of many studies in the modelling of complex random phenomena in a number of fields such as finance, with the modelling of the price of a financial asset (see \textit{e.g.} \cite{jaillet1990variational}, \cite{lamberton2011introduction}, \cite{el1997backward}), or the modelling of interest rates (see \textit{e.g.} \cite{kijima2009multi}, \cite{duffie1996yield}), in cell biology (see \textit{e.g.} \cite{sbalzarini2006analysis}, \cite{erban2009stochastic}), population genetics or population dynamics (see \textit{e.g.} \cite{etheridge2011some}, \cite{ethier1977error}), or in physics (see \textit{e.g.} \cite{domingo2020properties}). The estimation of coefficients of stochastic differential equations plays a key role in the study of these processes. It is the case for example for the construction of some classification procedures for diffusion paths (see \cite{gadat2020optimal}, \cite{denis2020consistent}), the automation of which is actively contributing to the development of several scientific and economic sectors. The analysis of the diffusion coefficient and its estimation plays a significant role in finance, with the volatility study of a financial asset.}

The goal of the article is to construct, from \textcolor{black}{$Nn$ observations $\left\{X^{j}_{k\Delta_n}, ~ k = 0, \ldots, n-1, ~ j = 1, \ldots, N\right\}$ coming from $N$ independent discrete time copies $\left\{\bar{X}^{1}, \ldots, \bar{X}^{N}\right\}$ of $X$} with time step $\Delta_n = 1/n$, a nonparametric estimator of the square of the diffusion coefficient $\sigma^{2}(.)$. We are in the framework of high frequency data since the time step $\Delta_{n}$ tends to zero as $n$ tends to infinity. Furthermore, we consider estimators of $\sigma^{2}(.)$ built from a single diffusion path \textcolor{black}{ and the $n$ observations $\left(X_{k\Delta_n}, ~ k = 0, \ldots, n\right)$} ($N = 1$), and those built when \textcolor{black}{$N$ tends to infinity}. In this paper, we first propose a ridge estimator of $\sigma^{2}(.)$ on a compact interval. Secondly, we focus on \textcolor{black}{the} nonparametric estimation of $\sigma^{2}(.)$ \textcolor{black}{on a non-compact interval, especially} the real line $\R$. We measure the risk of any estimator $\w{\sigma}^{2}$ of the square of the diffusion coefficient $\sigma^{2}$ by $\E\left[\|\w{\sigma}^{2} - \sigma^{2}\|^{2}_{n,N}\right]$, where $\|\w{\sigma}^{2} - \sigma^{2}\|^{2}_{n,N} := (Nn)^{-1}\sum_{j=1}^{N}{\sum_{k=0}^{n-1}{\left(\w{\sigma}^{2}(X^{j}_{k\Delta}) - \sigma^{2}(X^{j}_{k\Delta})\right)^2}}$ is an empirical norm defined from the \textcolor{black}{$Nn$ observations}.

\paragraph{Related works.}

There is a large literature on the estimation of coefficients of diffusion processes, and we focus on the papers studying the estimation of $\sigma^{2}$.

Estimation of the diffusion coefficient has been considered in the parametric case (see e.g.~\cite{genon1993estimation},~\cite{jacod1993random},~\cite{genon1994estimation},~\cite{clement1997},~\cite{genon1999parameter},~\cite{gloter2000discrete},~\cite{sorensen2002estimation}). In the nonparametric case, estimators of the diffusion coefficient from discrete observations are proposed under various frameworks. 

First, the diffusion coefficient is constructed from one discrete observation of the diffusion process ($N = 1$) in long time ($T \rightarrow \infty$) (see e.g. \cite{hoffmann1999adaptive},~\cite{comte2007penalized},~\cite{schmisser2012non},~\cite{schmisser2019non}), or in short time ($T = 1$) (see \textit{e.g.} \cite{genon1992non}, \cite{soulier1993estimation}, \cite{hoffmann1997minimax}, \cite{florens1998estimation}, \cite{hoffmann1999lp}). Note that in short time ($T<\infty$), only the diffusion coefficient can be estimated consistently from a single discrete path contrary to the drift function whose consistent estimation relies on repeated discrete observations of the diffusion process (see e.g.~\cite{comte2020nonparametric},~\cite{denis2020ridge}). For the case of short time diffusion processes (for instance $T = 1$), estimators of a time-dependent diffusion coefficients $t \mapsto \sigma^2(t)$ have been proposed. In this context, \cite{genon1992non} built a nonparametric estimator of $t\mapsto\sigma^{2}(t)$ and studied its $L_2$ risk using wavelets methods, \cite{soulier1993estimation} studies the $L_p$ risk of a kernel estimator of $\sigma^{2}(t)$, and \cite{hoffmann1997minimax} derived a minimax rate of convergence of order $n^{-s/(1+2s)}$ where $s>1$ is the smoothness parameter of the Besov space $\mathcal{B}^{s}_{p,\infty}([0,1])$ (see later in the paper). For the space-dependent diffusion coefficient $x \mapsto \sigma^{2}(x)$, a first estimator based on kernels and built from a single discrete observation of the diffusion process with $T = 1$ is proposed in \cite{florens1993estimating}. The estimator has been proved to be consistent under a condition on the bandwidth, but a rate of convergence of its risk of estimation has not been established.

Secondly, the diffusion coefficient is built in short time ($T < \infty$) from \textcolor{black}{$Nn$ observations of the diffusion process, with $N,n \rightarrow \infty$}. In \cite{denis2024nonparametric}, a nonparametric estimator of $\sigma^{2}$ \textcolor{black}{ on the real line $\R$} is proposed when the time horizon $T = 1$. The estimator has been proved to be consistent with a rate of order $N^{-1/5}$ over the space of Lipschitz functions. 

Two main methods are used to build consistent nonparametric estimators of $x \mapsto \sigma^{2}(x)$. The first method is the one using kernels (see e.g.~\cite{florens1998estimation},~\cite{bandi2003fully},~\cite{reno2006nonparametric},~\cite{gourieroux2017nonparametric},~\cite{schmisser2019non},~\cite{park2021nonparametric}), the other method consists in estimating $\sigma^{2}$ as solution of a nonparametric regression model using the least squares approach. Since the diffusion coefficient is assumed to belong to an infinite dimensional space, the method consists in projecting $\sigma^{2}$ into a finite dimensional subspace, estimating the projection and making a data-driven selection of the dimension by minimizing a penalized least squares contrast (see e.g.~\cite{hoffmann1999lp},~\cite{hoffmann1999adaptive},~\cite{schmisser2012non},~\cite{comte2017nonparametric},~\cite{schmisser2019non},~\cite{denis2024nonparametric}).

\paragraph{Main contribution.}

\textcolor{black}{In this paper, we assume to have at our disposal the following $Nn$ observations 
$$\left\{X^{j}_{k\Delta_n}, ~ j = 1, \ldots, N; k = 0, \ldots, n-1\right\},$$ 
gathered in the $N$ independent sample paths $\left\{\bar{X}^{j} = (X^{j}_{k\Delta_n})_{0 \leq k\leq n-1}, ~ j= 1, \ldots, N\right\}$ which are discrete-time copies of the diffusion process $X$ solution of Equation~\eqref{eq:model}.}
The main objectives of this paper are the following.

\begin{enumerate}
    \item \textcolor{black}{The construction of a consistent ridge estimator of $\sigma^{2}$ on a compact interval from a single diffusion path (N = 1), using the least squares approach. In the literature and in the same framework, \cite{hoffmann1999lp} proposed a projection estimator of $\sigma^{2}$ on the compact interval $[0,1]$, over an approximation space spanned by the wavelet basis. The author established an optimal rate of order $n^{-s/(2s+1)}$ on a Besov space of parameter $s$, defined from the space $L^{p}([0,1])$. To obtain this result, it is necessary to be sure that the data points are uniformly distributed over the compact interval $[0,1]$ using the local time of the diffusion process $X$, and working conditionally on the event "the local time is lower bounded on $[0,1]$". In this paper, we do not establish our results conditionally on the same event, we rather take advantage of the assumptions made on the coefficients $b$ and $\sigma$, the properties of the local time, and on the approximation of the transition density by Gaussian densities proposed in \cite{gobet2002lan}, to prove that the data points are uniformly scattered in the estimation interval. We establish, over a H\"older space, a rate of the same order as the one obtained in \cite{hoffmann1999lp}. Moreover, we propose an adaptive estimator of $\sigma^{2}$ based on a data-driven selection of the dimension through the minimization of a penalized least squares contrast.}
    \item \textcolor{black}{The construction of a nonparametric estimator of $\sigma^{2}$ on a non-compact interval and from $n+1$ observations $\left(X_{k\Delta_n}, ~ k = 0, \ldots, n\right)$ of a single diffusion path (N = 1), particularly on the real line $\mathbb{R}$. In the literature, only \cite{florens1993estimating} studied the nonparametric estimation of $\sigma^{2}$ on a non-compact interval from a single diffusion path. The author proposed a kernel estimator of $\sigma^{2}$ and proved its consistency. No rate of convergence has been established. In this article, we propose a projection estimator $\sigma^{2}$ using the least squares approach. We establish a rate of convergence, and propose an adaptive estimator through a data-driven selection of the dimension.}
    \item \textcolor{black}{For the case of repeated observations of the diffusion process ($N \rightarrow \infty$), only \cite{denis2024nonparametric} tackled the problem of estimation of $\sigma^{2}$ and established a rate of order $N^{-1/5}$ over a space of Lipschitz functions. In this paper, we propose projection estimators of $\sigma^{2}$ both on a compact interval and on a non-compact interval, particularly on the real line $\mathbb{R}$. We establish rates of convergence and propose adaptive estimators for both the compact and the non-compact cases.}
    \item Focusing on the support of the diffusion coefficient, we consider an intermediate case between a compact interval and $\R$ by proposing a ridge estimator of $\sigma^2$ restricted to the compact interval $[-A_N,A_N]$ where $A_N\rightarrow\infty$ as $N\rightarrow\infty$. The benefit of this approach is that the resulting projection estimator can reach a faster rate of convergence compared to the rate obtained on the real line $\R$.
\end{enumerate} 

We sum up below the rates of convergence of the ridge estimators of $\sigma^{2}_{|I}$ with $I\subseteq \R$ over a H\"older space defined in the next section with a smoothness parameter $\beta \geq 1$. 

\begin{table}[hbtp]
	\centering
\renewcommand{\arraystretch}{1.75}
\begin{tabular}{l|c|c} 
  Estimation interval & $N = 1$ and $n \rightarrow +\infty$ & $N \rightarrow +\infty$ and $n \rightarrow +\infty$ \\
\hline
 \multirow{1}{*}{$I = [-A,A], ~~ A>0$} & $n^{-\beta/(2\beta+1)}$ & $(Nn)^{-\beta/(2\beta+1)}$ \\ 
 \hline
 \multirow{1}{*}{$I = [-A_N,A_N], ~~ A_N\underset{N\rightarrow +\infty}{\longrightarrow}{+\infty}$} & xxxx & $\textcolor{black}{\log^{\beta/2}(N)}(Nn)^{-\beta/(2\beta+1)}, ~~ N\propto n$ \\ 
\hline
\multirow{1}{*}{$I = \R$} & \multirow{1}{*}{$\textcolor{black}{\log^{\beta}(n)}n^{-\beta/(4\beta+1)}$} & $\textcolor{black}{\log^{\beta}(N)}(Nn)^{-\beta/(4\beta+1)} + \textcolor{black}{n^{-1}}$ \\
\hline
\end{tabular}
	\caption{Rates of convergence of the square root of the risk of estimation $\mathbb{E}\left[\left\|\w{\sigma}^{2} - \sigma^{2}_{|I}\right\|^{2}_{n,N}\right]$ of the non-adaptive estimators $\widehat{\sigma}^2$ of the square of the diffusion coefficient $\sigma^{2}_{|I}$ built from one diffusion path ($N=1$) on the left column, and from repeated observations of the diffusion process ($N \rightarrow \infty$) on the right column. For the precise results, see Sections~\ref{sec:Estimation-OnePath} and \ref{sec:Estimation-N.paths}.}
	\label{tab:rate_spline}
\end{table}

\newpage

\paragraph{Outline of the paper.}

In Section~\ref{sec:framework and assumptions}, we define our framework with the key assumptions on the coefficients of the diffusion process ensuring for instance that Equation~\eqref{eq:model} admits a unique strong solution. Section~\ref{sec:Estimation-OnePath} is devoted to the non-adaptive estimation of the diffusion coefficient from \textcolor{black}{from one discrete observation of the diffusion process} both on a compact interval and on the real line $\R$. In Section~\ref{sec:Estimation-N.paths}, \textcolor{black}{we study the non-adaptive estimation of the diffusion coefficient from repeated discrete observations of the diffusion process}. We propose in Section~\ref{sec:AdaptiveEstimation-N.paths}, adaptive estimators of the diffusion coefficient, \textcolor{black}{considering the nature of the estimation interval together with the number of discrete observations of the process}, and Section~\ref{sec:NumericalStudy} complete the study with numerical evaluation of the performance of estimators. \textcolor{black}{A conclusion and some perspectives are given in Section~\ref{sec:Conclusion}, and we} prove our theoretical results in Section~\ref{sec:proof}.

\section{Framework and assumptions}
\label{sec:framework and assumptions}

Consider a diffusion process $X=(X_t)_{t\in[0,1]}$, solution of Equation~\eqref{eq:model} whose drift and diffusion coefficient satisfy the following assumption. 

\begin{assumption}
    \label{ass:Assumption 1}
    \begin{enumerate}
        \item There exists a constant $L_0>0$ such that $b$ and $\sigma$ are $L_0-$Lipschitz functions on $\mathbb{R}$.
        \item There exist constants $\sigma_{0},\sigma_{1}>0$ such that : $\sigma_{0}\leq\sigma(x)\leq\sigma_{1}, \ \ \forall x\in\mathbb{R}$.
        \item $\sigma\in\mathcal{C}^{2}\left(\mathbb{R}\right)$ and there exist $C >0$ and $\alpha\geq 0$ such that: 
        $$\left|\sigma^{\prime}(x)\right|+\left|\sigma^{\prime\prime}(x)\right|\leq C\left(1+|x|^{\alpha}\right), \ \ \forall x\in\mathbb{R}.$$
    \end{enumerate}
\end{assumption}
Under Assumption~\ref{ass:Assumption 1}, $X=(X_t)_{t\in[0,1]}$ is the unique strong solution of Equation~\eqref{eq:model} (see \cite{karatzas2014brownian}), and this unique solution admits a transition density $(t,\textcolor{black}{x_0},x)\mapsto p_X(t,\textcolor{black}{x_0},x)$ (see \textit{e.g.} \cite{gobet2002lan}). Besides, \textcolor{black}{we obtain from Assumption~\ref{ass:Assumption 1} the following result.
\begin{lemme}\label{lm:ConseqAssumption1}
    Under Assumption~\ref{ass:Assumption 1}, it holds:
 \begin{equation*}
    \forall q\geq 1, \ \ \mathbb{E}\left[\underset{t\in[0,1]}{\sup}{|X_t|^q}\right]<\infty.
\end{equation*}
\end{lemme}
The above result shows that under Assumption~\ref{ass:Assumption 1}, the diffusion process solution of Equation~\eqref{eq:model} admits moments of any order $q \in \mathbb{N}^{*}$. The proof of Lemma~\ref{lm:ConseqAssumption1} is provided in appendix.
}

\subsection{Definitions and notations}
\label{subsec:definitions and notations}

We suppose to have at our disposal \textcolor{black}{$Nn$ observations $D_{N,n}=\left\{X^{j}_{k\Delta}, ~ k = 0, \ldots, n, ~ j=1,\cdots,N\right\}$} constituted of $N$ independent copies of the discrete observation \textcolor{black}{$(X_{k\Delta_n}, ~ k = 0, \ldots, n)$} of the diffusion process $X$, where $\Delta_n = 1/n$ is the time-step. The objective is to construct, from the sample $D_{N,n}$, a nonparametric estimator of the square $\sigma^{2}$ of the diffusion coefficient on an interval $I \subseteq \R$. In the sequel, we consider two main cases, the first one being the estimation of $\sigma^{2}$ on the interval $I$ from \textcolor{black}{$n+1$ observations $D_{1,n} = \left\{X^{1}_{k\Delta_n}, ~ k = 0, \ldots, n\right\}$ ($N=1$), with $n\rightarrow \infty$}. For the second case, we assume that both $N$ and $n$ tend to infinity. 

For each measurable function $h$, such that $\mathbb{E}\left[h^{2}(X_t)\right]<\infty$ for all $t\in[0,1]$, we define the following empirical norms:
\begin{equation}
\label{eq:EmpNorms}
    \|h\|^{2}_{n}:=\mathbb{E}_{X}\left[\frac{1}{n}\sum_{k=0}^{n-1}{h^{2}\left(X_{k\Delta_n}\right)}\right], ~~~ \|h\|^{2}_{n,N}:=\frac{1}{Nn}\sum_{j=1}^{N}{\sum_{k=0}^{n-1}{h^{2}\left(X^{j}_{k\Delta_n}\right)}}, ~~~ \textcolor{black}{\|h\|^{2}_{X}:=\int_{0}^{1}{h^{2}(X_s)ds}}
\end{equation}
\textcolor{black}{where $\mathbb{E}_X$ is the corresponding expectation of the probability distribution of the diffusion process $X$ solution of Equation~\eqref{eq:model}. Besides, for} all $h \in \mathbb{L}^{2}(I)$, we have
\begin{align*}
    \|h\|^{2}_{n}=\int_{I}{h^{2}(x)\frac{1}{n}\sum_{k=0}^{n-1}{p_X(k\Delta_n,x_0,x)}dx}=\int_{I}{h^{2}(x)f_n(x)dx},
\end{align*}
where $f_n: x\mapsto\frac{1}{n}\sum_{k=0}^{n-1}{p_X(k\Delta_n,\textcolor{black}{x_0},x)}$ is a density function. For the case of non-adaptive estimators of $\sigma^2$, we also establish bounds of the risks of the estimators based on the empirical norm $\|.\|_n$ or the $\mathbb{L}^{2}-$norm $\|.\|$ when the estimation interval $I$ is compact.

For any integers $p,q \geq 2$ and any matrix $M \in \R^{p \times q}$, we denote by $^{t}M$, the transpose of $M$. \\

\textcolor{black}{In the sequel, without loss of generality and for the simplification of notations, we set $x_0 = 0$. The transition density is therefore defined by $p_X: (t,x) \mapsto p_X(t,x)$. Thus, for any solution $X$ of Equation~\eqref{eq:model} such that $X_0=x_0 \in \mathbb{R}\setminus\{0\}$, we consider the diffusion process $Y = (X_t - x_0)_{t \in [0,1]}$ solution of
$$dY_t = \widetilde{b}(Y_t)dt + \widetilde{\sigma}(Y_t)dW_t, ~~ Y_0 = 0,$$
with $\widetilde{b}: x\mapsto b(x+x_0)$ and $\widetilde{\sigma}: x\mapsto \sigma(x+x_0)$. Besides, setting the time horizon to $T = 1$ is without loss of generality.\\
For any real number $x \in \mathbb{R}_{+*}$, we denote by $\lfloor x \rfloor \in \mathbb{N}$ the largest integer strictly smaller than $x$, and for all integer $p,q \in \mathbb{N}\setminus\{0,1\}$ such that $p < q$, we denote by $[\![p,q]\!]$ the set of integers $\{p, p+1,\ldots,q\}$. 
} 

\subsection{Spaces of approximation}
\label{subsec:spaces of approximation}

We propose projection estimators of $\sigma^{2}$ on a finite-dimensional subspace. To this end, we consider for each $m \geq 1$, a $m$-dimensional subspace $\mathcal{S}_{m}$ given as follows:
\begin{equation}
\label{eq:approximation subspace}
     \mathcal{S}_m:=\mathrm{Span}\left(\phi_{\ell}, \ \ \ell=0,\cdots,m-1\right), ~~ m\geq 1
\end{equation}
where the functions $(\phi_{\ell}, ~ \ell \in \mathbb{N})$ are continuous, linearly independent and bounded on $I$. Furthermore, we need to control the $\ell^2-$norm of the coordinate vectors of elements of $\mathcal{S}_m$, which leads to the following constrained subspace,
\begin{equation}
 \label{eq:constrainded approximation subspace}
    \mathcal{S}_{m,L}:=\left\{h=\sum_{\ell=0}^{m-1}{a_{\ell}\phi_{\ell}}, \ \ \sum_{\ell=0}^{m-1}{a^{2}_{\ell}}=\|\mathbf{a}\|^{2}_{2}\leq mL, \ \ \mathbf{a}=\left(a_0,\cdots,a_{m-1}\right), \ \ L>0\right\}.
\end{equation}
Note that $\mathcal{S}_{m,L}\subset\mathcal{S}_m$ and $\mathcal{S}_{m,L}$ is no longer a vector space. The control of the coordinate vectors allows to establish an upper bound of the estimation error that tends to zero as $n\rightarrow\infty$ or $N,n\rightarrow \infty$. In fact, we prove in the next sections that the construction of consistent estimators of $\sigma^{2}$ requires the functions $h=\sum_{\ell=0}^{m-1}{a_{\ell}\phi_{\ell}}$ to be bounded, such that
$$ \|h\|_{\infty} \leq \underset{\ell=0,\ldots,m-1}{\max}{\|\phi_{\ell}\|_{\infty}}~\|\mathbf{a}\|_{2}. $$
This condition is satisfied for the functions of the constrained subspaces $\mathcal{S}_{m,L}$ with $m \geq 1$. In this article, we work with the following bases.

\paragraph{[B] The {\bf B}-spline basis.}

This is an \textcolor{black}{example} of a non-orthonormal basis defined on a compact interval. Let $A > 0$ be a real number, and suppose (without restriction) that $I = [-A,A]$. Let $K,M\in\mathbb{N}^{*}$, and consider $\mathbf{u}=\left(u_{-M},\cdots,u_{K+M}\right)$ a knots vector such that $u_{-M} = \cdots = u_{-1} = u_0 = -A$, $u_{K+1} = \cdots = u_{K+M} = A$, and for all $i=0,\cdots,K$,
\begin{align*}
   u_i = -A+i\frac{2A}{K}.
\end{align*}
One calls \textbf{B}-spline functions, the piecewise polynomial functions $\left(B_{\ell}\right)_{\ell=-M,\cdots,K-1}$ of degree $M$, associated with the knots vector $\mathbf{u}$ (see \cite{gyorfi2006distribution}, {\it Chapter 14}). The \textbf{B}-spline functions are linearly independent \textcolor{black}{smooth} functions returning zero for all $x\notin[-A,A]$, and satisfying some smoothness conditions established in \cite{gyorfi2006distribution}. Thus, we consider approximation subspaces $\mathcal{S}_{K+M}$ defined by
\begin{align*}
\mathcal{S}_{K+M}=\mathrm{Span}\left\{B_{\ell}, \ \ell=-M,\cdots,K-1\right\}
\end{align*}
of dimension $\dim(\mathcal{S}_{K+M})=K+M$, and in which, each function $h=\sum_{\ell=-M}^{K-1}{a_{\ell}B_{\ell}}$ is $M-1$ times continuously differentiable thanks to the properties of the spline functions (see~\cite{gyorfi2006distribution}). Besides, the spline basis is included in the definition of both the subspace $\mathcal{S}_{m}$ and the constrained subspace $\mathcal{S}_{m,L}$ (see Equations~\eqref{eq:approximation subspace}~and~\eqref{eq:constrainded approximation subspace}) with $m = K + M$ and for any coordinates vector $(a_{-M},\ldots,a_{K-1}) \in \R^{K+M}$,
$$ \sum_{\ell=-M}^{K-1}{a_{\ell}B_{\ell}} = \sum_{\ell=0}^{m-1}{a_{\ell-M}B_{\ell-M}}. $$
The integer $M \in \mathbb{N}^{*}$ is fixed, while $K$ varies in the set of integers $\mathbb{N}^{*}$. If we assume that $\sigma^{2}$ belongs to the H\"older space $\Sigma_{I}(\beta,R)$ given as follows:
\begin{equation*}
    \Sigma_{I}(\beta,R):=\left\{h\in\mathcal{C}^{\lfloor\beta\rfloor+1}(I), \ \left|h^{(\ell)}(x)-h^{(\ell)}(y)\right|\leq R|x-y|^{\beta-l}, \ \ x,y\in I\right\},
\end{equation*}
where $\beta \geq 1$, $\ell=\lfloor\beta\rfloor$ and $R>0$, then the unknown function $\sigma^{2}_{|I}$ restricted to the compact interval $I$ can be approximated in the constrained subspace $\mathcal{S}_{K+M,L}$ spanned by the spline basis. This approximation \textcolor{black}{results in} the following bias term:
\begin{equation*}
    \underset{h \in \mathcal{S}_{K+M,L}}{\inf}{\|h - \sigma^{2}_{|I}\|^{2}_{n}} \leq C|I|^{2\beta}K^{-2\beta}
\end{equation*}
where the constant $C > 0$ depends on $\beta, R$ and $M$, and $|I| = \sup I - \inf I$. The above result is a modification of \textit{Lemma D.2} in \cite{denis2020ridge}.

\paragraph{[F] The Fourier basis.}

The subspace $\mathcal{S}_{m}$ can be spanned by the Fourier basis
$$ \{f_{\ell}, ~~ \ell = 0, \ldots, m-1\} = \{1,\sqrt{2}\cos(2\pi jx), \sqrt{2}\sin(2\pi jx),~ j=1,...,d\} ~~ \mathrm{with} ~~ m=2d+1. $$
The above Fourier basis is defined on the compact interval $[0,1]$. The definition can be extended to any compact interval \textcolor{black}{$I \subset \mathbb{R}$}, replacing the bases functions $x \mapsto f_{\ell}(x)$ by $x \mapsto 1/(\max I - \min I)f_{\ell}(\frac{x-\min I}{\max I - \min I})$. We use this basis to build the estimators of $\sigma^{2}$ on a compact interval $I \subset \R$.  

Define for all $s \geq 1$ and for any compact interval $I \subset \R$, the Besov space $\mathcal{B}^{s}_{2,\infty}(I)$ which is a space of functions $f \in L^{2}(I)$ such that the $\lfloor s\rfloor^{th}$ derivative $f^{(\lfloor s \rfloor)}$ belongs to the space $\mathcal{B}^{s-\lfloor s \rfloor}_{2,\infty}(I)$ given by
$$ \mathcal{B}^{s - \lfloor s \rfloor}_{2,\infty}(I) = \left\{f \in L^{2}(I) ~ \mathrm{and} ~ \frac{w_{2,f}(t)}{t^{s - \lfloor s \rfloor}} \in L^{\infty}(I\cap\R^{+})\right\} $$
where for $s-\lfloor s\rfloor \in (0,1)$, $w_{2,f}(t)=\underset{|h|\leq t}{\sup}{\|\tau_{h}f - f\|_2}$ with $\tau_{h}f(x) = f(x-h)$, and for $s-\lfloor s\rfloor = 1$, $w_{2,f}(t)=\underset{|h|\leq t}{\sup}{\|\tau_{h}f + \tau_{-h}f - 2f\|_2}$. Thus, if we assume that the function $\sigma^{2}_{|I}$ belongs to the Besov space $\mathcal{B}^{s}_{2,\infty}$, then it can be approximated in a constrained subspace $\mathcal{S}_{m,L}$ spanned by the Fourier basis. Moreover, under Assumption~\ref{ass:Assumption 1} and from \textit{Lemma 12} in \cite{barron1999risk}, there exists a constant $C>0$ depending on the constant $\tau_1$ of Equation~\eqref{eq:equivalence between the L2 norm and the empirical norm}, the smoothness parameter $s$ of the Besov space such that
\begin{equation*}
    \underset{h\in\mathcal{S}_{m,L}}{\inf}{\left\|h-\sigma^{2}_{|I}\right\|^{2}_{n}} \leq \tau_1\underset{h\in\mathcal{S}_{m,L}}{\inf}{\left\|h-\sigma^{2}_{|I}\right\|^{2}} \leq C\left|\sigma^{2}_{|I}\right|^{2}_{\beta} m^{-2\beta}
\end{equation*}
where $|\sigma^{2}_{|I}|_{s}$ is the semi-norm of $\sigma^{2}_{|I}$ in the Besov space $\mathcal{B}^{s}_{2,\infty}(I)$. 

Note that for all $\beta \geq 1$, the H\"older space $\Sigma_{I}(\beta,R)$ and the Besov space $\mathcal{B}^{\beta}_{2,\infty}$ satisfy:
\begin{equation*}
    L^{\infty}(\R) \cap \Sigma_{I}(\beta,R) \subset \mathcal{B}^{\beta}_{\infty,\infty}(I) \subset \mathcal{B}^{\beta}_{2,\infty}(I)
\end{equation*}
(see \cite{devore1993constructive}, \textit{Chap. 2 page 16}). As a result, we rather consider in the sequel the H\"older space $\Sigma_{I}(\beta,R)$ which can also be approximated by the Fourier basis.

\paragraph{[H] The Hermite basis.}

The basis is defined from the Hermite functions $(h_j,j\geq 0)$ defined on $\mathbb{R}$ and given for all $j\geq 0$ and for all $x\in\mathbb{R}$ by:
\begin{align*}
    h_j(x)=c_jH_j(x), \ \ \mathrm{where} \ \ H_j(x)=(-1)^{j}\exp\left(\frac{x^2}{2}\right)\frac{d^{j}}{dx^{j}}\left(\mathrm{e}^{-x^2/2}\right) \ \ \mathrm{and} \ \ c_j=\left(2^{j}j!\sqrt{\pi}\right)^{-1/2}.
\end{align*}
The polynomials $H_j(x), \ j\geq 0$ are the Hermite polynomials, and $(h_j,j\geq 0)$ is an orthonormal basis of $L^{2}(\mathbb{R})$. Furthermore, \textcolor{black}{for each $j\geq 1$, the function $h_j$ satisfies $|h_j(x)|\leq c|x|\exp(-c_0x^2)$ for all $x \in \mathbb{R}$ such that $x^2\geq(3/2)(4j+3)$}, where $c,c_0>0$ are constants independent of $j$ (see ~\cite{comte2020regression}, Proof of Proposition 3.5). We use the Hermite basis in the sequel for the estimation of $\sigma^{2}$ on the real line $\R$.

If one assumes that $\sigma^{2}$ belongs to the Sobolev space $W^{s}_{f_n}(\R,R)$ given for all $s \geq 1$ by
\begin{equation*}
    W^{s}_{f_n}(\R,R) := \left\{g \in L^{2}(\R, f_n(x)dx),~ \forall~\ell \geq 1,~ \|g - g_{\ell}\|^{2}_{n} \leq R\ell^{-s}\right\}
\end{equation*}
where for each $\ell \geq 1$, $g_{\ell}$ is the $L^{2}(\R, f_n(x)dx)-$orthogonal projection of $g$ on the $\ell-$dimensional vector space $\mathcal{S}_{\ell}$ spanned by the Hermite basis. Consider a compact interval $I \subset \R$ and the following spaces:
\begin{align*}
    W^{s}(I,R) := &~ \left\{g \in L^{2}(I), ~ \sum_{j=0}^{\infty}{j^{s}\left<g,\phi_j\right>^{2} \leq R}\right\},\\
    W^{s}_{f_n}(I,R) := &~ \left\{g \in L^{2}(I, f_n(x)dx),~ \forall~\ell \geq 1,~ \|g - g_{\ell}\|^{2}_{n} \leq R\ell^{-s}\right\}
\end{align*}
where $(\phi_j)_{j\geq 0}$ is an orthonormal basis defined on $I$ and for all $\ell \geq 1$, $g_{\ell}$ is the orthogonal projection of $g$ onto $\mathcal{S}_{\ell} = \mathrm{Span}(h_j,~ j\leq \ell)$ of dimension $\ell \geq 1$ (see \textit{e.g.} \cite{comte2021drift}). Then, for all $g \in W^{s}(I,R)$, we have
$$ g=\sum_{j=0}^{\infty}{\left<g,\phi_j\right>\phi_j} ~~ \mathrm{and} ~~ \|g-g_{\ell}\|^2 = \sum_{j=\ell+1}^{\infty}{\left<g,\phi_j\right>^2} \leq \ell^{-s}\sum_{j=\ell+1}^{\infty}{j^{s}\left<g,\phi_j\right>^2} \leq R\ell^{-s}. $$
We have $W^{s}_{f_n}(I,R) = W^{s}(I,R)$ as the empirical norm $\|.\|_{n}$ and the $L^{2}-$norm $\|.\|$ are equivalent. The space $W^{s}_{f_n}(\R,R)$ is an extension of the space $W^{s}_{f_n}(I,R)$ where $I = \R$ and $(\phi_j)_{j \geq 0}$ is the Hermite basis.

\paragraph{\textcolor{black}{[L] The Laguerre basis.}}

\textcolor{black}{The Laguerre basis $(\ell_j)_{j \geq 1}$ is defined from the Laguerre functions defined on $\mathbb{R}^{+}$ and given for each $j \geq 1$ by
$$\ell_j(x) = \sqrt{2}L_j(2x)\mathrm{e}^{-x}\mathds{1}_{x \geq 0}, ~~ \mathrm{with} ~~ L_j(x) = \sum_{k=0}^{j}{(-1)^{k}\binom{j}{k}\frac{x^k}{k!}}.$$
The collection $[\mathbf{L}]$ is a complete orthonormal basis of $\mathbb{L}^{2}(\mathbb{R}^{+})$.
}

\paragraph{\textcolor{black}{[CS$-$OB] Other orthonormal bases.}}
\textcolor{black}{
We consider any orthonormal basis $\left(\phi_0, \ldots, \phi_{m-1}\right)$ of compactly supported functions that satisfies the following condition:
\begin{equation}\label{eq:Condition-OB}
    \exists ~ r \in \mathbb{N}: ~ \mathfrak{L}(m):= \sum_{j=0}^{m-1}{\left\|\phi_j\right\|^{2}_{\infty}} = O(m^{r}), ~~ \mathrm{and} ~~ \exists ~ r^{\prime} \in \mathbb{N}^{*}: ~ \mathfrak{R}(m):= \sum_{j=0}^{m-1}{\left\|\phi^{\prime}_{j}\right\|^{2}_{\infty}} = O(m^{r^{\prime}}).
\end{equation}
Thus, under Condition~\eqref{eq:Condition-OB}, any function $h \in \mathcal{S}_{m}$ satisfies $\|h\|_{\infty} = O\left(m^{r/2}\|\mathbf{a}\|_2\right)$ and $\|h^{\prime}\|_{\infty} = O\left(m^{r^{\prime}/2}\|\mathbf{a}\|_2\right)$, and, applying the finite increments theorem, there exists a constant $C>0$ such that for all $x,y \in \mathrm{Support}(h)$,
$$\left|h(x) - h(y)\right| \leq Cm^{r^{\prime}/2}\|\mathbf{a}\|_2|x-y|.$$
Among these bases, we have the Legendre basis with $r = r^{\prime} = 2$, the orthonormal basis of piecewise polynomials on $[0,1]$ with $r = r^{\prime} = 2$, the Fourier basis with $r = 1$ and $r^{\prime} = 3$.
}
\begin{remark}
    The \textbf{B}-spline basis is used for the estimation of $\sigma^{2}$ on a compact interval on one side ($N = 1$ and $N>1$), and on the real line on the other side restricting $\sigma^{2}$ on the compact interval $[-\log(n), \log(n)]$ for $N = 1$, or $[-\log(N), \log(N)]$ for $N > 1$, and bounding the exit probability of the process $X$ from the interval $[-\log(N), \log(N)]$ (or $[-\log(n), \log(n)]$) by a negligible term with respect to the estimation error. In a similar context, the Fourier basis is used as an orthonormal basis to built nonparametric estimators of $\sigma^{2}$ on a compact interval and on $\R$, both for $N = 1$ and for $N > 1$. The main goal is to show that, in addition to the spline basis which is not orthogonal, we can build projection estimators of $\sigma^{2}$ on orthonormal bases that are consistent. The advantage of the Hermite basis compared to the Fourier basis is its definition on the real line $\R$. As a result, we use the Hermite basis to propose for $N > 1$, a projection estimator  of $\sigma^{2}$ whose support is the real line $\R$.
\end{remark} 

\begin{remark}
    Denote, by $\mathcal{M}$, the set of possible values of the dimension $m \geq 1$ of the approximation subspace $\mathcal{S}_{m}$. If $\left(\phi_0,\cdots,\phi_{m-1}\right)$ is an orthonormal basis, then for all $m,m^{\prime}\in\mathcal{M}$ such that $m < m^{\prime}$, we have $\mathcal{S}_{m} \subset \mathcal{S}_{m^{\prime}}$. For the case of the \textbf{B}-spline basis, one can find a subset $\mathcal{K} \subset \mathcal{M}$ of the form 
    $$\mathcal{K}=\left\{2^{q}, \ q=0,\cdots,q_{\max}\right\}$$
    such that for all $K,K^{\prime}\in\mathcal{K}$, $K < K^{\prime}$ implies $\mathcal{S}_{K + M} \subset \mathcal{S}_{K^{\prime} + M}$ (see for example \cite{denis2020ridge}). The nesting of subspaces $\mathcal{S}_{m}, \ m\in\mathcal{M}$ is of great importance in the context of adaptive estimation of the diffusion coefficient and the establishment of upper-bounds for the risk of adaptive estimators.
\end{remark}

\textcolor{black}{In the sequel, we denote by $[\mathbf{CS} - \mathbf{OB}]$, $[\mathbf{NCS} - \mathbf{OB}]$ and $[\mathbf{B}]$, the respective collections of subspaces spanned by the orthonormal bases of compactly supported functions, the orthonormal bases of non-compactly supported functions, and the spline basis.  Note that on one side, the collection $[\mathbf{CS} - \mathbf{OB}]$ includes the Fourier basis related to the collection $[\mathbf{F}]$, and on the other side, the collection $[\mathbf{NCS} - \mathbf{OB}]$ includes the collection $[\mathbf{H}]$ for the Hermite basis, and the collection $[\mathbf{L}]$ for the Laguerre basis. Finally, we denote by $[\mathbf{OB}]$, the merger of the two collections $[\mathbf{CS} - \mathbf{OB}]$ and $[\mathbf{NCS} - \mathbf{OB}]$.
}

\subsection{Ridge estimators of the square of the diffusion coefficient}
\label{subsec:RidgeEstimatiors}

\textcolor{black}{We establish from Equation~\eqref{eq:model} and the sample paths $D_{N,n}$ the following regression model for the estimation of $\sigma^{2}$:
\begin{equation}\label{eq:RegModel}
  U^{j}_{k\Delta_n}=\sigma^{2}(X^{j}_{k\Delta_n})+\zeta^{j}_{k\Delta_n}+R^{j}_{k\Delta_n}, ~~ \forall (j,k)\in[\![1,N]\!]\times[\![0,n-1]\!]
\end{equation}
where for each pair $(j,k)\in[\![1,N]\!]\times[\![0,n-1]\!]$, the response variable $U^{j}_{k\Delta_n}$ is given by
$$ U^{j}_{k\Delta_n} := \frac{\left(X^{j}_{(k+1)\Delta_n} - X^{j}_{k\Delta_n}\right)^2}{\Delta_n},$$
and the error terms are respectively given by $\zeta^{j}_{k\Delta}=\zeta^{j,1}_{k\Delta}+\zeta^{j,2}_{k\Delta}+\zeta^{j,3}_{k\Delta}$, with:
\begin{equation*}
    \zeta^{j,1}_{k\Delta_n}=\frac{1}{\Delta_n}\left[\left(\int_{k\Delta_n}^{(k+1)\Delta_n}{\sigma(X^{j}_{s})dW^{j}_{s}}\right)^2-\int_{k\Delta_n}^{(k+1)\Delta_n}{\sigma^{2}(X^{j}_{s})ds}\right],
\end{equation*}
\begin{equation*}
    \zeta^{j,2}_{k\Delta_n}=\frac{2}{\Delta_n}\int_{k\Delta_n}^{(k+1)\Delta_n}{((k+1)\Delta_n-s)\sigma^{\prime}(X^{j}_{s})\sigma^{2}(X^{j}_{s})dW^{j}_{s}},
\end{equation*}
\begin{equation*}
\zeta^{j,3}_{k\Delta_n}=2b(X^{j}_{k\Delta_n})\int_{k\Delta_n}^{(k+1)\Delta_n}{\sigma\left(X^{j}_{s}\right)dW^{j}_{s}},
\end{equation*}
and $R^{j}_{k\Delta_n}
=R^{j,1}_{k\Delta_n}+R^{j,2}_{k\Delta_n}+R^{j,3}_{k\Delta_n}$, with:
\begin{equation*}
    R^{j,1}_{k\Delta_n}=\frac{1}{\Delta_n}\left(\int_{k\Delta_n}^{(k+1)\Delta_n}{b(X^{j}_{s})ds}\right)^2, ~~~~ R^{j,3}_{k\Delta_n} = \frac{1}{\Delta_n}\int_{k\Delta_n}^{(k+1)\Delta_n}{((k+1)\Delta_n-s)\Phi(X^{j}_{s})ds}
\end{equation*}
\begin{equation*}
R^{j,2}_{k\Delta_n}=\frac{2}{\Delta_n}\left(\int_{k\Delta_n}^{(k+1)\Delta_n}{\left(b(X^{j}_{s})-b(X^{j}_{k\Delta_n})\right)ds}\right)\left(\int_{k\Delta_n}^{(k+1)\Delta_n}{\sigma(X^{j}_{s})dW^{j}_{s}}\right),
\end{equation*}
where 
\begin{equation*} \Phi:=2b\sigma^{\prime}\sigma+\left[\sigma^{\prime\prime}\sigma+\left(\sigma^{\prime}\right)^2\right]\sigma^{2}.
\end{equation*}
The increments $U^{j}_{k\Delta_n}$ are approximations in discrete times of $\frac{d\left<X,X\right>_t}{dt}$ since, from Equation~\eqref{eq:model}, one has $d\left<X,X\right>_t = \sigma^{2}(X_t)dt$. The proof of the establishment of the regression model is given in appendix. 
}

We consider the least squares contrast $\gamma_{n,N}$ defined for all $m \in \mathcal{M}$ and for all \textcolor{black}{functions} $h\in\mathcal{S}_{m,L}$ by
\begin{equation*}
    \gamma_{n,N}(h):=\frac{1}{Nn}\sum_{j=1}^{N}{\sum_{k=0}^{n-1}{\left(U^{j}_{k\Delta_n}-h(X^{j}_{k\Delta_n})\right)^2}}.
\end{equation*}
For each dimension $m \in \mathcal{M}$, the projection estimator $\widehat{\sigma}^{2}_{m}$ of $\sigma^{2}$ over the subspace $\mathcal{S}_{m,L}$ \textcolor{black}{is defined as}:
\begin{equation}
\label{eq:non adaptive estimator}
    \widehat{\sigma}^{2}_{m}\in\underset{h\in\mathcal{S}_{m,L}}{\arg\min}{\ \gamma_{n,N}(h)}.
\end{equation}
Indeed, for each dimension $m \in \mathcal{M}$, the estimator $\w{\sigma}^{2}_{m}$ of $\sigma^{2}$ given in Equation~\eqref{eq:non adaptive estimator} satisfies $\w{\sigma}^{2}_{m}=\sum_{\ell=0}^{m-1}{\w{a}_{\ell}\phi_{\ell}}$, where
\begin{equation}\label{eq:ridge estimator-non adaptive}
\widehat{\mathbf{a}}=\left(\widehat{a}_{0},\cdots,\widehat{a}_{m-1}\right):=\underset{\|\mathbf{a}\|^{2}_{2}\leq mL}{\arg\min}{\left\|\mathbf{U}-\mathbf{F}_{m}\mathbf{a}\right\|^{2}_{2}}
\end{equation}
with ${\bf ^{t}U} = \left(U^{1}_{0},\ldots,U^{1}_{(n-1)\Delta_n}, \ldots, U^{N}_{0},\ldots,U^{N}_{(n-1)\Delta_n}\right)$ and the matrix $\mathbf{F}_{m}$ is defined as follows
\begin{equation*}
    {\bf F}_{m} := \left( ^{t}(\phi_{\ell}(X^{j}_{0}),\ldots,\phi_{\ell}(X^{j}_{(n-1)\Delta_n}))\right)_{\underset{1 \leq j \leq N}{0 \leq \ell \leq m-1}} \in \mathbb{R}^{Nn \times m}.
\end{equation*}
The vector of coefficients $\widehat{\mathbf{a}}$ is unique and called the ridge estimator of $\mathbf{a}$ because of the $\ell^{2}$ constraint on the coordinate vectors (see \cite{hastie2001springer} Chap. 3 page 61). 

\section{Estimation of the diffusion coefficient from a single diffusion path}
\label{sec:Estimation-OnePath}

This section focuses on the nonparametric estimation of the square of the diffusion coefficient $\sigma^{2}$ on an interval $I \subseteq \R$ when only a single diffusion path is observed at discrete times ($N=1$). It is proved in the literature that one can construct consistent estimators of the diffusion coefficient from one path when the time horizon $T$ is finite (see \textit{e.g.} \cite{hoffmann1999lp}). Two cases are considered. First, we propose a ridge estimator of $\sigma^{2}$ on a compact interval $I \subset \R$, say for example $I = [-1,1]$. Secondly, we extend the study to the estimation of $\sigma^2$ on the real line $I = \R$. 

\subsection{Non-adaptive estimation of the diffusion coefficient on a compact interval}
\label{subsec:Estimation-CompactInterval-OnePath} 

In this section, we consider the estimator $\w{\sigma}^{2}_{m}$ of the compactly supported square of the diffusion coefficient $\sigma^{2}_{|I}$ on the constrained subspaces $\mathcal{S}_{m,L}$ from the observation of a single diffusion path.

Since the interval $I\subset \R$ is compact, the immediate benefit is that the density function $f_n$ defined from the transition density of the diffusion process $\bar{X} = (X_{k\Delta_n})_{0 \leq k \leq n}$ is bounded from below. In fact, there exist constants $\tau_0,\tau_1\in(0,1]$ such that
\begin{equation*}
    \forall x\in I, \ \ \ \tau_0\leq f_n(x)\leq\tau_1,
\end{equation*}
(see \cite{denis2020ridge}). Thus, for each function $h\in\mathbb{L}^{2}(I)$,
\begin{equation}
\label{eq:equivalence between the L2 norm and the empirical norm}
    \tau_0\|h\|^{2}\leq\|h\|^{2}_{n}\leq\tau_1\|h\|^{2}
\end{equation}
where $\|.\|$ is the $\mathbb{L}^{2}-$norm. Equation~\eqref{eq:equivalence between the L2 norm and the empirical norm} allows establishing global rates of convergence of the risk of the ridge estimators $\w{\sigma}^{2}_{m}$ of $\sigma^{2}_{|I}$ with $m\in\mathcal{M}$ using the $L^{2}-$norm $\|.\|$ which is, in this case, equivalent with the empirical norm $\|.\|_n$.

To establish an upper-bound of the risk of estimation that tends to zero as $n$ tends to infinity, we need to establish equivalence relations between the pseudo-norms $\|.\|_{n,1} ~ (N=1)$ and $\|.\|_{X}$ (\textcolor{black}{see Equation~\eqref{eq:EmpNorms}}) on one side, and $\|.\|_{X}$ and the $L^{2}-$norm $\|.\|$ on the other side.
Define for $x\in\R$, the local time $\mathcal{L}^{x}$ of the diffusion process $X = (X_t)_{t\in[0,1]}$ by
\begin{equation}
\label{eq:LocalTime}
    \mathcal{L}^{x} = \underset{\varepsilon\rightarrow 0}{\lim}{\frac{1}{2\varepsilon}\int_{0}^{1}{\one_{(x-\varepsilon,x+\varepsilon)}(X_s)ds}}.
\end{equation}
In general, the local time of a continuous semimartingale is a.s. {\it càdlàg} (see {\it e.g.} \cite{revuz2013continuous}). But, for diffusion processes and under Assumption~\ref{ass:Assumption 1}, the local time $\mathcal{L}^{x}$ is \textcolor{black}{continuous} at any point $x\in\R$ (see Lemma~\ref{lm:LocalTimeBicontinuous} in Section~\ref{sec:proof}). Furthermore, we obtain the following result.
\begin{lemme}
    \label{lm:OccupationFormula}
    Under Assumption~\ref{ass:Assumption 1}, and for any continuous and integrable function $h$, \textcolor{black}{it holds},
    \begin{enumerate}
        \item $\int_{0}^{1}{h(X_s)ds} = \int_{\R}{h(x)\mathcal{L}^{x}dx}$.
        \item For all $x\in\R$, $\E(\mathcal{L}^{x}) = \int_{0}^{1}{p_{X}(s,x)ds}$.
    \end{enumerate}
\end{lemme}
In Lemma~\ref{lm:OccupationFormula}, we \textcolor{black}{show} that there is a link between the local time and the transition density of the diffusion process. \textcolor{black}{The proof of the Lemma is provided in appendix}. Thus, if we consider the pseudo-norm $\|.\|_{X}$ depending on the process $X = (X_t)_{t\in[0,1]}$ and given in Equation~\eqref{eq:EmpNorms}, and
using Lemma~\ref{lm:OccupationFormula}, we obtain that,
\begin{equation}
\label{eq:Equiv-NormX-NormL2}
    \E\left[\|h\|_X^2\right] = \int_{\R}{h^{2}(x)\E\left[\mathcal{L}^{x}\right]dx} = \int_{\R}{h^{2}(x)\int_{0}^{1}{p_X(s,x)ds}dx} \geq \tau_0\|h\|^{2}. 
\end{equation}
where $\int_{0}^{1}{p_X(s,x)ds} \geq \tau_0 >0$ (see \cite{denis2020ridge}, \textit{Lemma 4.3}), and $\|h\|^{2}$ is the $\mathbb{L}^{2}-$norm of $h$. 

\begin{theo}
    \label{thm:RiskBound-OnePath}
    Set $L = \log(n)$ \textcolor{black}{and $n \rightarrow \infty$}. Suppose that $\sigma^{2}$ is approximated in one of the collections \textcolor{black}{$[\mathbf{CS-OB}]$} and \textcolor{black}{$[\mathbf{B}]$}. Under Assumption~\ref{ass:Assumption 1}, \textcolor{black}{it holds} 
    \begin{align*}
        & \E\left[\left\|\w{\sigma}^{2}_{m} - \sigma^{2}_{|I}\right\|^{2}_{n,1}\right] \leq 3\underset{h\in\mathcal{S}_{m,L}}{\inf}{\left\|h-\sigma^{2}_{|I}\right\|^{2}_{n}}+C\left(\textcolor{black}{\frac{m^{\alpha^{\prime}}}{n} + \frac{m^{\alpha\gamma+1}\log(n)}{n^{\gamma/2}}} + \Delta^{2}_{n}\right) \\
        & \E\left[\left\|\w{\sigma}^{2}_{m} - \sigma^{2}_{|I}\right\|^{2}_{n}\right] \leq \frac{34\tau_1}{\tau_0}\underset{h\in\mathcal{S}_{m,L}}{\inf}{\left\|h-\sigma^{2}_{|I}\right\|^{2}_{n}}+C^{\prime}\left(\textcolor{black}{\frac{m^{\alpha^{\prime}}}{n} + \frac{m^{\alpha\gamma+1}\log(n)}{n^{\gamma/2}}} + \Delta^{2}_{n}\right)
    \end{align*}
    where \textcolor{black}{the numerical constant $\gamma$ can take any value of the interval $(1,+\infty)$}, \textcolor{black}{with $\alpha = 2$ and $\alpha^{\prime} = 1$ for the collection $[\mathbf{B}]$, $\alpha = (r+r^{\prime})/2$ and $\alpha^{\prime} = r$ for the collection $[\mathbf{CS} - \mathbf{OB}]$}. The constant $C>0$ depends on $\sigma_1$ and the constant $C^{\prime}>0$ depends on $\sigma_1, \tau_0$ and $\tau_1$.
\end{theo}

We observe that the upper-bound of the risk of estimation of $\w{\sigma}^{2}_{m}$ is composed of the bias term, which quantifies the cost of approximation of $\sigma^{2}_{|I}$ in the constrained space $\mathcal{S}_{m,L}$, the estimation error \textcolor{black}{$\mathrm{O}(m^{\alpha^{\prime}}/n)$} and the cost of the time discretization $\mathrm{O}(\Delta^{2}_{n})$ are established on a random event in which the pseudo-norms $\|.\|_{n,1}$ and $\|.\|_X$ are equivalent, and whose probability of the complementary times $\left\|\w{\sigma}^{2}_{m} - \sigma^{2}_{|I}\right\|^{2}_{\infty}$ is bounded by the term \textcolor{black}{$\mathrm{O}\left(\frac{m^{\alpha\gamma+1}\log(n)}{n^{\gamma/2}}\right)$} (see Lemma~\ref{lm:Proba-OmegaComp-OnePath} and proof of Theorem~\ref{thm:RiskBound-OnePath}). 

\textcolor{black}{Note that for the two collections $[\mathbf{B}]$ and $[\mathbf{F}]$, we have $\alpha = 2$, and $\alpha^{\prime} = 1$.} The next result proves that \textcolor{black}{for the collections $[\mathbf{B}]$ and $[\mathbf{F}]$,} the risk of estimation can reach a rate of convergence of the same order \textcolor{black}{as} the rate established in \cite{hoffmann1999lp} if the parameter $\gamma > 1$ is chosen such that the term $\mathrm{O}(m^{2\gamma+1}\log(n)/n^{\gamma/2})$ is \textcolor{black}{negligible with respect to} the estimation error of order $m/n$. \textcolor{black}{Let us also note} that the risk $\left\|\w{\sigma}^{2}_{m} - \sigma^{2}_{|I}\right\|^{2}_{n}$ is random since 
$$ \left\|\w{\sigma}^{2}_{m} - \sigma^{2}_{|I}\right\|^{2}_{n} = \E_{X}\left[\frac{1}{n}\sum_{k=0}^{n-1}{(\w{\sigma}^{2}_{m} - \sigma^{2}_{|I})(X_{k\Delta_n})}\right] $$
and the estimator $\w{\sigma}^{2}_{m}$ is built from an independent copy $\bar{X}^{1}$ of the discrete times process $\bar{X}$. Thus, the expectation $\E$ corresponds to the joint distribution of the random couple $(X^{1},X)$, and so relates to the estimator $\w{\sigma}^{2}_{m}$.

\begin{coro}
\label{coro:Rate-OnePath}
Suppose that $\sigma^{2} \in \Sigma_{I}(\beta,R)$ with $\beta > 3/2$, and \textcolor{black}{$\gamma = 4(2\beta+1)/(2\beta-3)$}. Assume that $K_{\mathrm{opt}} \propto n^{1/(2\beta+1)}$ for $[\mathbf{B}]$ ($m_{\mathrm{opt}} = K_{\mathrm{opt}} + M$), $m_{\mathrm{opt}} \propto n^{1/(2\beta+1)}$ for $[\mathbf{F}]$, \textcolor{black}{and $n \rightarrow \infty$}. Under Assumptions~\ref{ass:Assumption 1}, \textcolor{black}{it holds},
\begin{align*}
    & \E\left[\left\|\w{\sigma}^{2}_{m_{\mathrm{opt}}} - \sigma^{2}_{|I}\right\|^{2}_{n,1}\right] = \textcolor{black}{\mathrm{O}\left(n^{-2\beta/(2\beta+1)}\right)}\\
    & \E\left[\left\|\w{\sigma}^{2}_{m_{\mathrm{opt}}} - \sigma^{2}_{|I}\right\|^{2}_{n}\right] = \textcolor{black}{\mathrm{O}\left(n^{-2\beta/(2\beta+1)}\right)}.
\end{align*}
\end{coro}

Note that we obtain the exact same rates when considering the risk of $\w{\sigma}^{2}_{m_{\mathrm{opt}}}$ defined with the $\mathbb{L}^{2}-$norm equivalent to the empirical norm $\|.\|_n$. Moreover, these rates of convergence are of the same order \textcolor{black}{as} the optimal rate $n^{-s/(2s+1)}$ established in \cite{hoffmann1999lp} over a Besov ball.

\subsection{Non-adaptive estimation of the diffusion coefficient on \textcolor{black}{non-compact interval}}
\label{subsec:Estimation-R-OnePath}

In this section, we propose a ridge estimator of $\sigma^{2}$ on \textcolor{black}{a non-compact interval $I = \mathcal{I}$}, built from one diffusion path. In this context, the main drawback is that the density function $f_n:x\mapsto\frac{1}{n}\sum_{k=0}^{n-1}{p_X(k\Delta_n,x)}$ is no longer lower bounded. Consequently, the empirical norm $\|.\|_n$ is no longer equivalent to the $L_2-$norm $\|.\|$ and the consistency of the estimation error is no longer ensured under the \textcolor{black}{assumptions} made in the previous sections. Consider the truncated estimator $\w{\sigma}^{2}_{m,L}$ of $\sigma^{2}$ given by 
\begin{equation}
\label{eq:Truncated-Estimator}
    \w{\sigma}^{2}_{m,L}(x) = \w{\sigma}^{2}_{m}(x)\one_{\w{\sigma}^{2}_{m}(x) \leq \sqrt{L}} + \sqrt{L}\one_{\w{\sigma}^{2}_{m}(x) > \sqrt{L}}.
\end{equation}
\textcolor{black}{For the specific case of approximation spaces spanned by bases of compactly supported functions, we proceed to the dilation of the bases and consider the following decomposition of the risk of the ridge estimator $\w{\sigma}^{2}_{m,L}$:}
\begin{align*}
    \E\left[\left\|\w{\sigma}^{2}_{m,L} - \sigma_{\textcolor{black}{|\mathcal{I}}}^{2}\right\|^{2}_{n,1}\right] \leq &~ \E\left[\left\|(\w{\sigma}^{2}_{m,L} - \sigma_{\textcolor{black}{|\mathcal{I}}}^{2})\one_{[-\log(n),\log(n)]}\right\|^{2}_{n,1}\right] + \E\left[\left\|(\w{\sigma}^{2}_{m,L} - \sigma_{\textcolor{black}{|\mathcal{I}}}^{2})\one_{[-\log(n),\log(n)]^{c}}\right\|^{2}_{n,1}\right] \\
    \leq &~ \E\left[\left\|(\w{\sigma}^{2}_{m,L} - \sigma_{\textcolor{black}{|\mathcal{I}}}^{2})\one_{[-\log(n),\log(n)]}\right\|^{2}_{n,1}\right] + 4\textcolor{black}{L}\underset{t\in[0,1]}{\sup}{\P(|X_t|>\log(n))}.
\end{align*}
The first term on the {\it r.h.s.} is equivalent to the risk of a ridge estimator of $\sigma^{2}$ on the compact interval $[-\log(n),\log(n)]$. The second term on the {\it r.h.s.} is upper-bounded using Lemma~\ref{lem:controleSortiCompact-bis}. \textcolor{black}{The estimation interval $\mathcal{I}$ is set to $\mathcal{I} = \mathbb{R}^{+}$ for the collection $[\mathbf{L}]$, and $\mathcal{I} = \mathbb{R}$ for the collection $[\mathbf{H}]$ and any chosen basis of compactly supported functions like those related to the collections $[\mathbf{B}]$ and $[\mathbf{F}]$, and in this case, $\sigma_{|\mathcal{I}}^{2} = \sigma^{2}$}. We derive below, an upper-bound of the risk of estimation of $\w{\sigma}^{2}_{m}$.

\begin{theo}
    \label{thm:ConstFromOnePath}
    Suppose that $L = \log^{2}(n)$ \textcolor{black}{and $n \rightarrow \infty$}. Under Assumption~\ref{ass:Assumption 1}, \textcolor{black}{it holds},
    \begin{align*}
       \E\left[\left\|\w{\sigma}^{2}_{m,L} - \sigma_{\textcolor{black}{|\mathcal{I}}}^{2}\right\|^{2}_{n,1}\right] \leq \underset{h\in\mathcal{S}_{m,L}}{\inf}{\|h-\sigma^{2}\|^{2}_{n}}+C\sqrt{\frac{m^{q}\log^{2}(n)}{n}}
    \end{align*}
    where $C>0$ is a constant, $q = 1$ for the collection $[\mathbf{B}]$, and $q = 2$ for the collection \textcolor{black}{$[\mathbf{OB}]$}.
\end{theo}
We first remark that the upper-bound of the risk of the truncated estimator of $\sigma^{2}$ differs with respect to each of the chosen bases. This contrast comes from the fact that \textcolor{black}{any orthonormal basis $\{\phi_{\ell},~ \ell = 0, \ldots, m-1\}$} and the spline basis $\{B_{\ell-M}, ~ \ell = 0, \ldots, m-1\}$ \textcolor{black}{respectively} satisfy
$$ \textcolor{black}{\sum_{\ell = 0}^{m - 1}{\phi_{\ell}(x)} \leq C_{\phi}m}, ~ \mathrm{and} ~ \sum_{\ell=0}^{m-1}{B_{\ell-M}(x)} = 1. $$
Secondly, the estimation error is not as fine as the one established in Theorem~\ref{thm:RiskBound-OnePath} where $\sigma^{2}$ is estimated on a compact interval. In fact, on the real line $\R$, the pseudo-norm $\|.\|_X$ can no longer be equivalent to the $\mathbb{L}^{2}-$norm since the transition density is not bounded from below on $\R$. Consequently, we cannot take advantage of the exact method used to establish the risk bound obtained in Theorem~\ref{thm:RiskBound-OnePath} which uses the equivalence relation between the pseudo-norms $\|.\|_{n,1}$ and $\|.\|_X$ on one side, and $\|.\|_X$ and the $\mathbb{L}^2-$norm $\|.\|$ on the other side. Moreover, we can also notice that the term of order $1/n^2$ does not appear since it is dominated by the estimation error.

We obtain below rates of convergence of the ridge estimator of $\sigma^{2}$ for each of the collections $[\mathbf{B}]$ and $[\mathbf{F}]$.
\begin{coro}
    \label{coro:Rate-OnePath-NonCompact}
    Suppose that $\sigma^{2} \in \Sigma_{I}(\beta,R)$ with $\beta \geq 1$ 
    \paragraph{For [\textbf{B}].}
    Assume that $K \propto n^{1/(4\beta+1)}$ \textcolor{black}{and $n \rightarrow \infty$}. Under Assumptions~\ref{ass:Assumption 1}, there exists a constant $C>0$ depending on $\beta$ and $\sigma_1$ such that
    $$ \E\left[\left\|\w{\sigma}^{2}_{m,L} - \sigma^{2}\right\|^{2}_{n,1}\right] \leq C\log^{2\beta}(n)n^{-2\beta/(4\beta+1)}. $$
    \paragraph{For [\textbf{F}].}
    Assume that $m \propto n^{1/2(2\beta+1)}$ \textcolor{black}{and $n \rightarrow \infty$}. Under Assumptions~\ref{ass:Assumption 1}, \textcolor{black}{it holds},
    $$ \E\left[\left\|\w{\sigma}^{2}_{m,L} - \sigma^{2}\right\|^{2}_{n,1}\right] \leq C\log(n)n^{-\beta/(2\beta+1)} $$
    where the constant $C>0$ depends on $\beta$ and $\sigma_1$.
\end{coro}

As we can remark, the obtained rates are slower than the ones established in Section~\ref{subsec:Estimation-CompactInterval-OnePath} where $\sigma^{2}$ is estimated on a compact interval. This result is the immediate consequence of the result of Theorem~\ref{thm:ConstFromOnePath}. \textcolor{black}{At this stage, we cannot position ourselves on the optimality or not of the obtained rates of convergence, as a lower-bound on the risk of estimation is not provided in this paper. The establishment of a lower-bound could be obtained by adapting the method provided in \cite{hoffmann1999lp} to the case of a non-compactly supported diffusion coefficient. This study is beyond the scope of this paper, and will be subject to further investigations.} 

\section{Estimation of the diffusion coefficient from repeated diffusion paths}
\label{sec:Estimation-N.paths}

We now focus on the estimation of the (square) of the diffusion coefficient from i.i.d. discrete observations of the diffusion process ($N \rightarrow \infty$).

\subsection{Non-adaptive estimation of the diffusion coefficient on a compact interval}
\label{subsec:Rate-CompInterval-N.paths}

We study the rate of convergence of the ridge estimators $\w{\sigma}^{2}_{m}$ of $\sigma^{2}_{|I}$ from $D_{N,n}$ when $I$ is a compact interval. The next theorem gives an upper-bound of the risk of our estimators $\w{\sigma}^{2}_{m},~ m\in\mathcal{M}$.

\begin{theo}
\label{thm:RiskBound-CompactSupport}
Suppose that $L = \log(Nn)$, \textcolor{black}{$n,N \rightarrow \infty$} and $\mathcal{M} = \left\{1,\ldots,\sqrt{\min(n,N)}/\log(Nn)\right\}$. Under Assumption~\ref{ass:Assumption 1} and for all $m \in \mathcal{M}$, there exist constants $C>0$ and $C^{\prime}>0$ depending on $\sigma_1$ such that,
\begin{align*}
     \mathbb{E}\left[\left\|\widehat{\sigma}^{2}_{m}-\sigma^{2}_{|I}\right\|^{2}_{n,N}\right]\leq &~ 3\underset{h\in\mathcal{S}_{m,L}}{\inf}{\left\|h-\sigma^{2}_{|I}\right\|^{2}_{n}}+C\left(\frac{m}{Nn}+m\log(Nn)\exp\left(-C\sqrt{\min(n,N)}\right)+\Delta^{2}_{n}\right) \\
    \mathbb{E}\left[\left\|\widehat{\sigma}^{2}_{m}-\sigma^{2}_{|I}\right\|^{2}_{n}\right]\leq &~ 34\underset{h\in\mathcal{S}_{m,L}}{\inf}{\left\|h-\sigma^{2}_{|I}\right\|^{2}_{n}} + C^{\prime}\left(\frac{m}{Nn}+m\log(Nn)\exp\left(-C\sqrt{\min(n,N)}\right)+\Delta^{2}_{n}\right).
\end{align*}
\end{theo}

Note that the result of Theorem~\ref{thm:RiskBound-CompactSupport}~is independent of the choice of the basis that \textcolor{black}{generates} the approximation space $\mathcal{S}_{m}$. The first term on the right-hand side represents the approximation error of the initial space, the second term $\mathrm{O}\left(m/(Nn)\right)$ is the estimation error, and the last term characterizes the cost of the time discretization. \textcolor{black}{The third term on the right-hand side is negligible with respect to the estimation error, and is derived from the result of Lemma~\ref{lm:Proba-OmegaComp}. Note that Theorem~\ref{thm:RiskBound-CompactSupport} does not extend Theorem~\ref{thm:RiskBound-OnePath} since Theorem~\ref{thm:RiskBound-CompactSupport}, contrary to Theorem~\ref{thm:RiskBound-OnePath}, holds for any compactly supported basis, takes advantage of the independence of the sample paths $\bar{X}^{j}, ~ j = 1, \ldots, N$, and requires the condition $m \leq \sqrt{\min(n,N)}/\log(Nn)$ on the dimensions $m$ of the approximation spaces, which reduces the set of dimension $\mathcal{M}$ to the empty set $\emptyset$ for $N = 1$.}

\textcolor{black}{We derive below and from Theorem~\ref{thm:RiskBound-CompactSupport}, some rates of convergence related to the collections $[\mathbf{B}]$ and $[\mathbf{F}]$.} 


\begin{coro}
\label{coro:Rate-N.paths}
    Suppose that $\sigma^{2} \in \Sigma_{I}(\beta,R)$ with $\beta > 3/2$. Moreover, assume that \textcolor{black}{$n \propto N, ~ N \rightarrow \infty$}, $K_{\mathrm{opt}} \propto (Nn)^{1/(2\beta+1)}$ for $[\mathbf{B}]$ ($m_{\mathrm{opt}} = K_{\mathrm{opt}} + M$), and $m_{\mathrm{opt}} \propto (Nn)^{1/(2\beta+1)}$ for $[\mathbf{F}]$. Under Assumptions~\ref{ass:Assumption 1}, \textcolor{black}{it holds},
\begin{align*}
   \mathbb{E}\left[\left\|\widehat{\sigma}^{2}_{m_{\mathrm{opt}}}-\sigma^{2}_{|I}\right\|^{2}_{n,N}\right] =  &~\mathrm{O}\left((Nn)^{-2\beta/(2\beta+1)}\right)\\
    \mathbb{E}\left[\left\|\widehat{\sigma}^{2}_{m_{\mathrm{opt}}}-\sigma^{2}_{|I}\right\|^{2}_{n}\right] =  &~ \mathrm{O}\left((Nn)^{-2\beta/(2\beta+1)}\right).
 \end{align*}
\end{coro}

The obtained result shows that the nonparametric estimators of $\sigma^{2}_{|I}$ based on repeated observations of the diffusion process are more efficient when $N,n\rightarrow\infty$. Note that the same rate is obtained if the risk of $\w{\sigma}^{2}_{m_{\mathrm{opt}}}$ is defined with the $\mathbb{L}^{2}-$norm $\|.\|$ equivalent to the empirical norm $\|.\|_n$.

The rate obtained in Corollary~\ref{coro:Rate-N.paths} is established for $\beta > 3/2$. If we consider for example the collection ${\bf [B]}$ and assume that $\beta \in [1, 3/2]$, then $K_{\mathrm{opt}} \propto (Nn)^{1/(2\beta+1)}$ belongs to $\mathcal{M}$ for $n \propto \sqrt{N}/\log^{4}(N)$ and we have
       \begin{align*}
          \mathbb{E}\left[\left\|\widehat{\sigma}^{2}_{m_{\mathrm{opt}}}-\sigma^{2}_{|I}\right\|^{2}_{n,N}\right] = O\left((Nn)^{-2\beta/(2\beta+1)} + n^{-2}\right).
       \end{align*}
Under the condition $n \propto \sqrt{N}/\log^{4}(N)$, the obtained rate is of order $n^{-3\beta/(2\beta+1)}$ (up to a log-factor) which is equivalent to $N^{-3\beta/2(2\beta+1)}$ (up to a log-factor).

\textcolor{black}{It remains to prove the optimality of the obtained rate of convergence, especially when $N \propto n$. Unfortunately, the method proposed in \cite{tsybakov2008introduction} does not work for the establishment of a lower-bound of the risk of estimation of the diffusion coefficient. In fact, the third condition of Theorem 2.5 in \cite{tsybakov2008introduction} states that there exists a numerical constant $\alpha \in (0,1/8)$ such that 
$$\frac{1}{M}\sum_{j=1}^{M}{K(P^{\otimes N}_{0},P^{\otimes N}_{j})} \leq \alpha \log(M),$$
where $M$ is the number of hypotheses satisfying $M \geq 2^{m/8}$ with, in our framework, $m$ defined such that $m^{2\beta} \propto (Nn)^{2\beta/(2\beta+1)}$, and $\mathcal{K}(P_j^{\otimes N},P_0^{\otimes N})$ is the Kullback divergence between the two joint probability distributions $P_j^{\otimes N}$ and $P_0^{\otimes N}$ characterized by their respective diffusion coefficients $\sigma^{(j)}$ and $\sigma^{(0)}$. In our framework, the computation of the Kullback divergence is based on the ratio of their respective transition densities, and one obtains:
$$\frac{1}{M}\sum_{j=1}^{M}{K(P^{\otimes N}_{0},P^{\otimes N}_{j})} = \mathrm{O}(Nn).$$
Consequently, since $m$ is negligible with respect to $Nn$, it is not possible to satisfy this third condition using the inequality $M \geq 2^{m/8}$. A way to solve this problem may be to reconsider the construction of the set $\left\{\sigma^{(0)}, \ldots, \sigma^{(M)}\right\}$ of hypotheses, and use the technique developed in \cite{hoffmann1999lp}. This study is beyond the scope of this paper.
}

\subsection{Non-adaptive estimation of the diffusion coefficient on \textcolor{black}{a non-compact interval}}
\label{subsec:Rate-R-N.paths}

Consider a ridge estimator of $\sigma^{2}$ on \textcolor{black}{a non-compact interval $\mathcal{I}$} built from $N$ independent copies of the diffusion process $X$ observed in discrete times, where both $N$ and $n$ tend to infinity.
For each $m \in \mathcal{M}$, we still denote by $\w{\sigma}^{2}_{m}$ the ridge estimators of $\sigma^{2}$ and $\w{\sigma}^{2}_{m,L}$ the truncated estimators of $\sigma^{2}$ given in Equation~\eqref{eq:Truncated-Estimator}. We establish, through the following theorem, the first risk bound that highlights the main error terms.

\begin{theo}
\label{thm:RiskBound-AnySupport}
Suppose that $L=\log^{2}(N\textcolor{black}{n})$ \textcolor{black}{and $N,n \rightarrow \infty$}. Under Assumptions~\ref{ass:Assumption 1} and for any dimension $m\in\mathcal{M}$, the following holds:
 \begin{align*}
   \mathbb{E}\left[\left\|\widehat{\sigma}^{2}_{m,L}-\sigma_{\textcolor{black}{|\mathcal{I}}}^{2}\right\|^{2}_{n,N}\right] \leq &~ 2\underset{h\in\mathcal{S}_{m,L}}{\inf}{\|h-\sigma_{\textcolor{black}{|\mathcal{I}}}^{2}\|^{2}_{n}}+C\left(\sqrt{\frac{m^{q}\log^{2}(N\textcolor{black}{n})}{Nn}} + \Delta^{2}_{n}\right)
 \end{align*}
 where $C>0$ is a constant depending on the upper bound $\sigma_1$ of the diffusion coefficient. Moreover, $q = 1$ for the collection $[\mathbf{B}]$ and $q = 2$ for the collection \textcolor{black}{$[\mathbf{OB}]$}.
\end{theo}

\textcolor{black}{The result of Theorem~\ref{thm:RiskBound-AnySupport} is clearly an extension of Theorem~\ref{thm:ConstFromOnePath}, it suffices to set $N = 1$ in Theorem~\ref{thm:RiskBound-AnySupport} to obtain Theorem~\ref{thm:ConstFromOnePath} with the same hypotheses. Moreover, the respective proofs of the two theorems are similar, the only difference being the number of sample paths.}

If we consider the risk of $\w{\sigma}^{2}_{m,L}$ using the empirical norm $\|.\|_n$, then we obtain
\begin{equation}
\label{eq:mise-R-norm.n}
    \mathbb{E}\left[\left\|\widehat{\sigma}^{2}_{m,L}-\sigma_{\textcolor{black}{|\mathcal{I}}}^{2}\right\|^{2}_{n}\right] \leq 2\underset{h\in\mathcal{S}_{m,L}}{\inf}{\|h-\sigma_{\textcolor{black}{|\mathcal{I}}}^{2}\|^{2}_{n}}+C\left(\sqrt{\frac{m^{q}\log^{2}(N\textcolor{black}{n})}{Nn}} + \frac{m^2\log^{3}(N\textcolor{black}{n})}{N}+\Delta^{2}_{n}\right) 
\end{equation}
The risk bound given in Equation~\eqref{eq:mise-R-norm.n} is a sum of four error terms. The first term is the approximation error linked to the choice of the basis, the second term is the estimation error given in Theorem~\ref{thm:RiskBound-AnySupport}, the third term $m^2\log^{3}(N)/N$ comes from the relation linking the empirical norm $\|.\|_n$ to the pseudo-norm $\|.\|_{n,N}$ (see Lemma~\ref{lem:Relation-Nn-n}), and the last term is the cost of the time-discretization. \textcolor{black}{In this case, Theorem~\ref{thm:RiskBound-AnySupport} can no longer be an extension of Theorem~\ref{thm:ConstFromOnePath} since the establishment of Equation~\eqref{eq:mise-R-norm.n}, which is derived from Lemma~\ref{lem:Relation-Nn-n}, requires a large number $N > 1$ of independent sample paths (see \cite{denis2020ridge}, \textit{proof of Theorem 3.3}).} 

We derive, in the next result, rates of convergence of the risk bound of the truncated ridge estimators $\w{\sigma}^{2}_{m,L}$ based on the collections $[\mathbf{B}]$ and $[\mathbf{H}]$ respectively.

\begin{coro}
\label{cor:Rate-NonCompact}
Suppose that $\sigma^{2} \in \Sigma_{I}(\beta,R)$ with $\beta \geq 1$, $\textcolor{black}{N,n \rightarrow \infty}$, $I = [-\log(N),\log(N)]$, and $K \propto (Nn)^{1/(4\beta+1)}$ for $[\mathbf{B}]$, and $\sigma^{2} \in W^{s}_{f_n}(\R,R)$ with $s \geq 1$ and $m \propto (Nn)^{1/2(2s+1)}$ for $[\mathbf{H}]$. Under Assumption~\ref{ass:Assumption 1}, the following holds:
\begin{equation*}
  \mathrm{For} ~ [\mathbf{B}] ~~ \mathbb{E}\left[\left\|\widehat{\sigma}^{2}_{m,L}-\sigma^{2}\right\|^{2}_{n,N}\right] \leq C\left(\log^{2\beta}(N)(Nn)^{-2\beta/(4\beta+1)} + \frac{1}{n^2}\right),
\end{equation*}
\begin{equation*}
   \mathrm{For} ~ [\mathbf{H}] ~~ \mathbb{E}\left[\left\|\widehat{\sigma}^{2}_{m,L}-\sigma^{2}\right\|^{2}_{n,N}\right] \leq C\left(\log^{3}(N)(Nn)^{-s/(2s+1)} + \frac{1}{n^2}\right).
\end{equation*}
where $C>0$ is a constant depending on $\beta$ and $\sigma_1$ for $[\mathbf{B}]$, or $s$ and $\sigma_1$ for $[\mathbf{H}]$.
\end{coro}
The obtained rates are slower compared to the rates established in Section~\ref{subsec:Rate-CompInterval-N.paths} for the estimation of $\sigma^{2}_{|I}$ where the interval $I\subset\R$ is compact. In fact, the method used to establish the rates of Theorem~\ref{thm:RiskBound-AnySupport} from which the rates of Corollary~\ref{cor:Rate-NonCompact} are obtained, does not allow us to derive rates of order $(Nn)^{-\alpha/(2\alpha+1)}$ (up to a log-factor) with $\alpha\geq 1$ ({\it e.g.} $\alpha = \beta, s$). Finally, if we consider the risk defined with the empirical norm $\|.\|_n$, then from Equation~\eqref{eq:mise-R-norm.n} with $n \propto N$ and assuming that $m \propto N^{1/4(s+1)}$ for $[\mathbf{H}]$ or $K \propto N^{1/4(\beta+1)}$ for $[\mathbf{B}]$, we obtain
\begin{align*}
   [\mathbf{B}]: ~~ ~~ \mathbb{E}\left[\left\|\widehat{\sigma}^{2}_{m,L}-\sigma^{2}\right\|^{2}_{n}\right] \leq &~ C\log^{2\beta}(N)(Nn)^{-\beta/2(\beta+1)}, \\
   [\mathbf{H}]: ~~ ~~ \mathbb{E}\left[\left\|\widehat{\sigma}^{2}_{m,L}-\sigma^{2}\right\|^{2}_{n}\right] \leq &~ C\log^{3}(N)(Nn)^{-s/2(s+1)},
\end{align*}
where $C>0$ is a constant depending on $\sigma_1$ and on the smoothness parameter. We can see that the obtained rates are slower compared to the results of Corollary~\ref{cor:Rate-NonCompact} for $n \propto N$. The deterioration of the rates comes from the additional term of order $m^2\log^{3}(N)/N$ which is now regarded as the new estimation error since it dominates the other term in each case as $N\rightarrow \infty$.

\textcolor{black}{Finally, the study of an optimal rate of convergence, which is beyond the scope of this paper, has to take into account the new problem related to the non-compact support of the square of the diffusion coefficient.}

\section{Adaptive estimation of the diffusion coefficient from repeated observations}
\label{sec:AdaptiveEstimation-N.paths}

In this section, \textcolor{black}{we propose an adaptive ridge estimator of $\sigma^{2}$ by selecting an optimal dimension from the sample $D_{N,n}$}. In fact, consider the estimator \textcolor{black}{$\widehat{\sigma}^{2}_{\widehat{K}}$} where $\widehat{K}$ satisfies:
\begin{equation}
    \label{eq:selection of the dimension}
    \widehat{K}:=\underset{K\in\mathcal{K}}{\arg\min}{\left\{\gamma_{n,N}\left(\widehat{\sigma}^{2}_{K}\right)+\mathrm{pen}(K)\right\}}
\end{equation}
and the penalty function $\mathrm{pen} : K\mapsto \mathrm{pen}(K)$ is established using the chaining technique of \cite{baraud2001model}. We derive below the risk of the adaptive estimator of $\sigma^{2}_{|I}$ when the interval $I\subset\R$ is compact, \textcolor{black}{and the number $N$ of sample paths tends to infinity}.

\begin{theo}
    \label{thm:adaptive estimator - compact interval}
    Suppose that $N \propto n, ~ L=\log(N)$, \textcolor{black}{$N \rightarrow \infty$} and consider the collection ${\bf [B]}$ with 
    $$K \in \mathcal{K} = \left\{2^q,~ q=0,1,\ldots,q_{\max}\right\} \subset \textcolor{black}{\mathcal{M} = \left\{1,\ldots,\lfloor\sqrt{N}/\log(N)\rfloor\right\}}.$$
    Under Assumption~\ref{ass:Assumption 1}~, there exists a constant $C>0$ such that,
    \begin{align*}
        \mathbb{E}\left[\left\|\widehat{\sigma}^{2}_{\widehat{K}}-\sigma^{2}_{|I}\right\|^{2}_{n,N}\right]\leq 34\underset{K\in\mathcal{K}}{\inf}{\left\{\underset{h\in\mathcal{S}_{K+M,L}}{\inf}{\|h-\sigma^{2}_{|I}\|^{2}_{n}}+\textcolor{black}{\mathrm{pen}_1(K)}\right\}}+\frac{C}{Nn}
    \end{align*}
    \textcolor{black}{where the penalty function $\mathrm{pen}_1$ is given by} 
     \begin{equation}
        \label{eq:penalty-function-1}
        \textcolor{black}{\mathrm{pen}_1(K) := \kappa\dfrac{(K+M)\log(N)}{Nn}}
    \end{equation}
   \textcolor{black}{with $\kappa > 0$ a numerical constant}.
\end{theo}
We deduce from Corollary~\ref{coro:Rate-N.paths} and its assumptions that the adaptive estimator \textcolor{black}{$\widehat{\sigma}^{2}_{\widehat{K}}$} satisfies:
\begin{align*}
\mathbb{E}\left[\left\|\widehat{\sigma}^{2}_{\widehat{K}}-\sigma^{2}_{|I}\right\|^{2}_{n}\right] = \mathrm{O}\left((Nn)^{-2\beta/(2\beta+1)}\right).
\end{align*}
This result is justified since the penalty term is of te same order (up to a log-factor) \textcolor{black}{as} the estimation error established in Theorem~\ref{thm:RiskBound-CompactSupport}. \textcolor{black}{If we consider the estimation interval $[-A_N,A_N]$ where $A_N \propto \sqrt{\log(N)}$, then we obtain the following result:
$$\mathbb{E}\left[\left\|\widehat{\sigma}^{2}_{\widehat{K}}-\sigma^{2}_{|I}\right\|^{2}_{n,N}\right]\leq 34\underset{K\in\mathcal{K}}{\inf}{\left\{\underset{h\in\mathcal{S}_{K+M,L}}{\inf}{\|h-\sigma^{2}_{|I}\|^{2}_{n}}+\mathrm{pen}(K)\right\}}+\frac{C}{Nn}$$
where the penalty function $K \mapsto \mathrm{pen}(K)$ is given by
\begin{equation}\label{eq:penalty-function}
    \mathrm{pen}(K) := \kappa\dfrac{(K+M)\log^{2}(N)}{Nn}.
\end{equation}
}
Considering the adaptive estimator of $\sigma^{2}$ on the real line $I=\R$ when \textcolor{black}{$N \geq 1$ and $n \rightarrow \infty$}, we obtain the following result.

\begin{theo}
    \label{thm:adaptive estimator - non compact interval}
    Suppose that \textcolor{black}{$N \geq 1$}, \textcolor{black}{$n \rightarrow \infty$} and \textcolor{black}{$L = \log(Nn)$}, and consider the collection $[\mathbf{B}]$ with
    $$K \in \mathcal{K} = \left\{2^q,~ q=0,1,\ldots,q_{\max}\right\} \subset \textcolor{black}{\mathcal{M} = \left\{1,\ldots, Nn\right\}}.$$
    Under Assumption~\ref{ass:Assumption 1}~and for $N$ large enough, the exists a constant $C>0$ such that,
    \begin{equation*}
      \E\left[\left\|\w{\sigma}^{2}_{\w{K},L} - \sigma^{2}\right\|^{2}_{n,N}\right] \leq 3\underset{K\in\mathcal{K}}{\inf}\left\{\underset{h\in\mathcal{S}_{K+M,L}}{\inf}{\left\|h-\sigma^{2}\right\|^{2}_{n}} + \textcolor{black}{\mathrm{pen}_2(K)}\right\} + \textcolor{black}{\frac{C}{\sqrt{Nn}}}
    \end{equation*}
     \textcolor{black}{where the penalty function $\mathrm{pen}_2$ given by} 
     \begin{equation}
        \label{eq:penalty-function-2}
        \textcolor{black}{\mathrm{pen}_2(K) := \kappa\sqrt{\dfrac{(K+M)\log^{2}(Nn)}{Nn}}}
    \end{equation}
    \textcolor{black}{with $\kappa > 0$ a numerical constant}.
\end{theo}

We have a penalty term of the same order \textcolor{black}{as} the one obtained in Theorem~\ref{thm:adaptive estimator - compact interval} where $\sigma^{2}$ is estimated on a compact interval. One can deduce that the adaptive estimator reaches a rate of the same order \textcolor{black}{as} the rate of the non-adaptive estimator given in Corollary~\ref{cor:Rate-NonCompact} for the collection $[\mathbf{B}]$. \textcolor{black}{For the numerical study of the nonparametric estimation of the square of the diffusion coefficient, we rather consider the compact interval $[-\log(Nn), \log(Nn)]$ on which is built the spline basis, and which tends to the real line as $N,n \rightarrow \infty$ or $N = 1$ and $n \rightarrow \infty$. Then, we use the penalty function $K \mapsto \mathrm{pen}(K)$ given in Equation~\eqref{eq:penalty-function}}. 

If we consider the adaptive estimator of the diffusion coefficient, built \textcolor{black}{on a compact interval} from a single diffusion path, we obtain below an upper-bound of its risk of estimation.

\begin{theo}
    \label{thm:AdaptiveEstimator-OnePath}
    Suppose that $N = 1, ~ \textcolor{black}{L = \log(n)}$, \textcolor{black}{$n \rightarrow \infty$} and consider the collection $[\mathbf{B}]$ with 
    $$K \in \mathcal{K} = \left\{2^q,~ q=0,\ldots,q_{\max}\right\} \subset \textcolor{black}{\mathcal{M} = \left\{1,\ldots,\lfloor\sqrt[5]{n}\rfloor\right\}}.$$ 
    Under Assumption~\ref{ass:Assumption 1}, \textcolor{black}{it holds}
    \begin{align*}
       \mathbb{E}\left[\left\|\widehat{\sigma}^{2}_{\widehat{K}}-\sigma^{2}_{|I}\right\|^{2}_{n,1}\right] \leq &~ 3\underset{K\in\mathcal{K}}{\inf}{\left\{\underset{h\in\mathcal{S}_{K+M,L}}{\inf}{\left\|h-\sigma^{2}_{|I}\right\|^{2}_{n}}+\textcolor{black}{\mathrm{pen}_3(K)}\right\}} + \frac{C}{n}.
   \end{align*}
    where $C>0$ is a constant depending on $\tau_0$, and 
    $$\textcolor{black}{\mathrm{pen}_3(K)} = \kappa\frac{(K+M)\log(n)}{n}$$ 
    with $\kappa>0$ a numerical constant.
\end{theo}

We deduce from Theorem~\ref{thm:AdaptiveEstimator-OnePath} that if we assume that $\sigma^2 \in \Sigma_{I}(\beta,R)$, then the adaptive estimator $\widehat{\sigma}^{2}_{\widehat{K}}$ reaches a rate of order $n^{-\beta/(2\beta+1)}$ (up to a log-factor). The result of this theorem is almost a deduction of the result of Theorem~\ref{thm:adaptive estimator - compact interval}, the slight difference being the use, in the proofs, of the local time of the process and the equivalence relation between the pseudo-norm $\|.\|_{n,1}$ with the pseudo-norm $\|.\|_X$ instead of the empirical norm $\|.\|_n$ considered in the proof of Theorem~\ref{thm:adaptive estimator - compact interval}. \textcolor{black}{Since the condition imposed on the dimension of the approximation subspace strongly restricts the model selection, we extend the set of possible values of the dimension to $\mathcal{M} = \left\{1, \ldots, \sqrt{n}/\log(n)\right\}$ for the numerical study of the estimation of the square $\sigma^{2}$ of the diffusion coefficient.}

\section{Numerical study}
\label{sec:NumericalStudy}

This section is devoted to the numerical study on a simulation scheme. Section~\ref{subsec:Models-Simulations} focuses on the presentation of the chosen diffusion models. In Section~\ref{subsec:Implementation}, we describe the scheme for the implementation of the ridge estimators. We mainly focus on the \textbf{B}-spline basis for the numerical study, and in Section~\ref{subsec:NumResults}, we add a numerical study on the performance of the Hermite-based ridge estimator of $\sigma^{2}$ on $\R$. Finally, we compare the efficiency of our estimator built on the real line $\R$ from a single path with that of the Nadaraya-Watson estimator proposed in \cite{florens1993estimating}.

\subsection{Models and simulations}
\label{subsec:Models-Simulations}

Recall that the time horizon is $T=1$ and $X_0 = 0$. Consider the following diffusion models:

\begin{itemize}
    \item[] Model $1$ Ornstein-Uhlenbeck: $b(x) = 1-x, ~~ \sigma(x)= 1$ 
    \item[] Model $2$: $b(x) = 1-x, ~~ \sigma(x) = 1-x^2$
    \item[] Model $3$: $b(x) = 1-x, ~~ \sigma(x) = \frac{1}{3+\sin(2\pi x)}+\cos^2\left(\frac{\pi}{2}x\right)$
\end{itemize}

Model $1$ is the commonly used \textcolor{black}{Ornstein$-$Uhlenbeck} model, known to be a simple diffusion model satisfying Assumption~\ref{ass:Assumption 1}. Model $2$ does not satisfy Assumption~\ref{ass:Assumption 1}. Model $3$ satisfies Assumption~\ref{ass:Assumption 1}~ with a multimodal diffusion coefficient.

The \textcolor{black}{number} $N$ of \textcolor{black}{sample paths} takes values in the set $\{1,10,100,1000\}$ where the \textcolor{black}{number} $n$ of \textcolor{black}{of observations of each diffusion path} varies in the set $\{100,250,500,1000\}$. As we work with the spline basis, the dimension $m=K+M$ of the approximation space is chosen such that $M=3$ and $K$ takes values in $\mathcal{K}=\{2^p, \ p=0,\cdots,5\}$ so that the subspaces are nested inside each other. We are using \texttt{R} for the simulation of diffusion paths via the function \texttt{sde.sim} of \texttt{sde} package, (see \cite{iacus2009simulation} for more details on the simulation of SDEs).

\subsection{Implementation of the ridge estimators}
\label{subsec:Implementation}

In this section, we assess the quality of estimation of the adaptive estimator $\widehat{\sigma}^{2}_{\widehat{m}}$ in each of the $3$ models through the computation of its risk of estimation. We compare the performance of the adaptive estimator with that of the oracle estimator $\widehat{\sigma}^{2}_{m^{*}}$ where $m^{*}$ is given by:
\begin{equation}
\label{eq:the oracle dimension}
    m^{*}:=\underset{m\in\mathcal{M}}{\arg\min}{\ \left\|\widehat{\sigma}^{2}_{m}-\sigma^{2}\right\|^{2}_{n,N}}.
\end{equation}
For the spline basis, we have $m^{*} = K^{*} + M$ with $M=3$. Finally, we complete the numerical study with a representation of a set of $10$ estimators of $\sigma^{2}$ for each of the $3$ models.  

We evaluate the MISE of the spline-based adaptive estimators $\widehat{\sigma}^{2}_{\widehat{K}}$ by repeating $100$ times the following steps:

\begin{enumerate}
    \item Simulate \textcolor{black}{observations} $D_{N,n}$ and $D_{N^{\prime},n}$ with $N\in\{1,10,100,1000\}$, $N^{\prime}=100$ and $n \in \{100, 250,1000\}$.
    \item For each $K\in\mathcal{K}$, and from $D_{N,n}$, compute estimators $\widehat{\sigma}^{2}_{K}$ given in Equations~\eqref{eq:non adaptive estimator}~and~\eqref{eq:ridge estimator-non adaptive}.
    \item Select the optimal dimension $\widehat{K}\in\mathcal{K}$ using Equation~\eqref{eq:selection of the dimension}~and compute $K^{*}$ from Equation~\eqref{eq:the oracle dimension}
    \item Using $D_{N^{\prime},n}$, evaluate $\left\|\widehat{\sigma}^{2}_{\widehat{K}}-\sigma^{2}\right\|^{2}_{n,N^{\prime}}$ and $\left\|\widehat{\sigma}^{2}_{K^{*}}-\sigma^{2}\right\|^{2}_{n,N^{\prime}}$.
\end{enumerate}

We deduce the risks of estimation considering the average values of $\left\|\widehat{\sigma}^{2}_{\widehat{m}}-\sigma^{2}\right\|^{2}_{n,N^{\prime}}$ and $\left\|\widehat{\sigma}^{2}_{m^{*}}-\sigma^{2}\right\|^{2}_{n,N^{\prime}}$ over the $100$ repetitions. Note that we consider in this section, the estimation of $\sigma^2$ on the compact interval $I = [-1,1]$ and on the real line $\R$. \textcolor{black}{The unknown parameter $\kappa$ of the corresponding penalty functions given in Equations~\eqref{eq:penalty-function-1} and \eqref{eq:penalty-function-2}, are numerically calibrated (details are given in appendix). We choose $\kappa = 4$ for the compact interval $I = [-1,1]$,  and $\kappa = 5$ for the estimation of $\sigma^{2}$ on the real line $\mathbb{R}$, more precisely, on the interval $[-\log(N), \log(N)]$ which tends to $\mathbb{R}$ as the number $N$ of sample paths tends to infinity.} 

\subsection{Numerical results}
\label{subsec:NumResults}

We present in this section the numerical results of the performance of the spline-based adaptive estimators of $\sigma^{2}_{|I}$ with $I \subseteq \R$ together with the performance of the oracle estimators. We consider the case $I=[-1,1]$ for the compactly supported diffusion coefficient, and the case $I=\R$.

Tables~\ref{tab:MSEs_adaptive_oracle-n100}~and~\ref{tab:MSEs_adaptive_oracle-n250}~present the numerical results of estimation of $\sigma^{2}_{|I}$ from simulated data following the steps given in Section~\ref{subsec:Implementation}.

\begin{table}[hbtp]
	\centering
\renewcommand{\arraystretch}{1.5}
\begin{tabular}{l|c|c|c|c|c} 
  Models & Intervals & Estimators & $N=10$ & $N=100$ & $N=1000$ \\
\hline
 \multirow{4}{*}{Model 1} & \multirow{2}{*}{$[-1,1]$} & $\widehat{\sigma}^{2}_{\widehat{K},L}$ & $0.0102 \ \ (0.0083)$ & $0.0009 \ \ (0.0009)$  & $0.0002\ \ (0.0001)$ \\ 
& & $\widehat{\sigma}^{2}_{K^{*},L}$ & $0.0094 \ \ (0.0065)$ & $0.0009 \ \ (0.0009)$ & $0.0002 \ \ (0.0001)$ \\
 & \multirow{2}{*}{$\R$} & $\widehat{\sigma}^{2}_{\widehat{K},L}$ & $0.0096 \ (0.0062)$ & $0.0009 \ \ (0.0008)$  & $0.0003\ \ (0.0002)$ \\ 
& & $\widehat{\sigma}^{2}_{K^{*},L}$ & $0.0093 \ \ (0.0057)$ & $0.0009 \ \ (0.0008)$ & $0.0003 \ \ (0.0002)$ \\
 \hline
 \multirow{4}{*}{Model 2}  & \multirow{2}{*}{$[-1,1]$} & $\widehat{\sigma}^{2}_{\widehat{K},L}$ & $0.0048 \ \ (0.0052)$ & $0.0019 \ \ (0.0008)$  & $0.0005\ \ (0.0002)$ \\ 
& & $\widehat{\sigma}^{2}_{K^{*},L}$ & $0.0039 \ \ (0.0043)$ & $0.0009 \ \ (0.0005)$ & $0.0005 \ \ (0.0002)$ \\
 & \multirow{2}{*}{$\R$} & $\widehat{\sigma}^{2}_{\widehat{K},L}$ & $0.0195 \ (0.0140)$ & $0.0057 \ \ (0.0006)$  & $0.0012\ \ (0.0002)$ \\ 
& & $\widehat{\sigma}^{2}_{K^{*},L}$ & $0.0048 \ \ (0.0064)$ & $0.0025 \ \ (0.0021)$ & $0.0010 \ \ (0.0003)$  \\ 
\hline
\multirow{4}{*}{Model 3}  & \multirow{2}{*}{$[-1,1]$} & $\widehat{\sigma}^{2}_{\widehat{K},L}$ & $0.0521 \ \ (0.0191)$ & $0.0176 \ \ (0.0070)$  & $0.0073\ \ (0.0021)$ \\ 
& & $\widehat{\sigma}^{2}_{K^{*},L}$ & $0.0260 \ \ (0.0081)$ & $0.0073 \ \ (0.0030)$ & $0.0048 \ \ (0.0009)$ \\
 & \multirow{2}{*}{$\R$} & $\widehat{\sigma}^{2}_{\widehat{K},L}$ & $0.1132 \ (0.0595)$ & $0.0319 \ \ (0.0031)$  & $0.0179 \ \ (0.0054)$ \\ 
& & $\widehat{\sigma}^{2}_{K^{*},L}$ & $0.0351 \ \ (0.0169)$ & $0.0319 \ \ (0.0031)$ & $0.0116 \ \ (0.0051)$ \\
\hline
\end{tabular}
	\caption{Assessment of $\mathrm{MISEs}$ (mean and standard deviation between brackets) of both the adaptive estimator $\widehat{\sigma}^{2}_{\widehat{K},L}$ and the oracle estimator $\widehat{\sigma}^{2}_{K^{*},L}$ from diffusion paths of size $n=100$.}
	\label{tab:MSEs_adaptive_oracle-n100}
\end{table}
\begin{table}[hbtp]
	\centering
\renewcommand{\arraystretch}{1.5}
\begin{tabular}{l|c|c|c|c|c} 
  Models & Intervals & Estimators & $N=10$ & $N=100$ & $N=1000$ \\
\hline
  \multirow{4}{*}{Model 1} & \multirow{2}{*}{$[-1,1]$} & $\widehat{\sigma}^{2}_{\widehat{K},L}$ & $0.0047 \ \ (0.0037)$ & $0.0003 \ \ (0.0002)$  & $0.0001\ \ (0.00003)$ \\ 
& & $\widehat{\sigma}^{2}_{K^{*},L}$ & $0.0042 \ \ (0.0030)$ & $0.0003 \ \ (0.0002)$ & $0.0001 \ \ (0.00003)$ \\
 & \multirow{2}{*}{$\R$} & $\widehat{\sigma}^{2}_{\widehat{K},L}$ & $0.0053 \ (0.0037)$ & $0.0003 \ \ (0.0002)$  & $0.0001\ \ (0.00004)$ \\ 
& & $\widehat{\sigma}^{2}_{K^{*},L}$ & $0.0050 \ \ (0.0031)$ & $0.0003 \ \ (0.0002)$ & $0.0001 \ \ (0.00004)$ \\
 \hline
 \multirow{4}{*}{Model 2}  & \multirow{2}{*}{$[-1,1]$} &$\widehat{\sigma}^{2}_{\widehat{K},L}$ & $0.0027 \ \ (0.0019)$ & $0.0003 \ \ (0.0002)$  & $0.0002\ \ (0.00004)$ \\ 
& & $\widehat{\sigma}^{2}_{K^{*},L}$ & $0.0018 \ \ (0.0019)$ & $0.0002 \ \ (0.0001)$ & $0.0001 \ \ (0.00004)$ \\
 & \multirow{2}{*}{$\R$} & $\widehat{\sigma}^{2}_{\widehat{K},L}$ & $0.0091 \ (0.0077)$ & $0.0028 \ \ (0.0025)$ & $0.0008 \ \ (0.0002)$ \\ 
& & $\widehat{\sigma}^{2}_{K^{*},L}$ & $0.0020 \ \ (0.0023)$ & $0.0021 \ \ (0.0023)$ & $0.0002 \ \ (0.00004)$ \\
\hline
\multirow{4}{*}{Model 3}  & \multirow{2}{*}{$[-1,1]$} & $\widehat{\sigma}^{2}_{\widehat{K},L}$ & $0.0306 \ \ (0.0150)$ & $0.0058 \ \ (0.0012)$  & $0.0010 \ \ (0.0003)$ \\ 
& & $\widehat{\sigma}^{2}_{K^{*},L}$ & $0.0216 \ \ (0.0067)$ & $0.0023 \ \ (0.0020)$ & $0.0010 \ \ (0.0003)$ \\
 & \multirow{2}{*}{$\R$} & $\widehat{\sigma}^{2}_{\widehat{K},L}$ & $0.0560 \ (0.0313)$ & $0.0275 \ \ (0.0049)$  & $0.0069 \ \ (0.0049)$ \\ 
& & $\widehat{\sigma}^{2}_{K^{*},L}$ & $0.0261 \ \ (0.0127)$ & $0.0096 \ \ (0.0051)$ & $0.0065 \ \ (0.0041)$ \\
\hline
\end{tabular}
	\caption{Assessment of $\mathrm{MISEs}$ of both the adaptive estimator $\widehat{\sigma}^{2}_{\widehat{K},L}$ and the oracle estimator $\widehat{\sigma}^{2}_{K^{*},L}$ from diffusion paths of size $n=250$.}
	\label{tab:MSEs_adaptive_oracle-n250}
\end{table}

The results of Table~\ref{tab:MSEs_adaptive_oracle-n100}~and Table~\ref{tab:MSEs_adaptive_oracle-n250}~show that the adapted estimator $\widehat{\sigma}^{2}_{\widehat{K}}$ is consistent, since its MISE tends to zero as \textcolor{black}{the number $Nn$ of observations is larger}. Moreover, note that in most cases, the ridge estimators of the compactly supported diffusion coefficients perform better than those of the non-compactly supported diffusion functions. As expected, we observe that the oracle estimator has generally a better performance compared to the adaptive estimator. Nonetheless, we can remark that the performances are very close in several cases, highlighting the efficiency of the data-driven selection of the dimension.

An additional important remark is the significant influence of the \textcolor{black}{number} $n$ of \textcolor{black}{each diffusion path} on the performance of $\widehat{\sigma}^{2}_{\widehat{K}}$ and $\w{\sigma}^{2}_{K^{*},L}$ (by comparison of Table \ref{tab:MSEs_adaptive_oracle-n100} with Table~\ref{tab:MSEs_adaptive_oracle-n250}), which means that estimators built from higher frequency data are more efficient. A similar remark is made for theoretical results obtained in Sections~\ref{subsec:Rate-R-N.paths}~and~\ref{subsec:Rate-CompInterval-N.paths}. 

\paragraph{Performance of the Hermite-based estimator of the diffusion coefficient}

We focus on the estimation of $\sigma^{2}$ on $\R$ and assess the performance of its Hermite-based estimator (see Section~\ref{subsec:Rate-R-N.paths}). We present in Table~\ref{tab:MSEs_hermite}, the performance of the oracle estimator $\w{\sigma}^{2}_{m^{*},L}$. 

\begin{table}[hbtp]
	\centering
\renewcommand{\arraystretch}{1.5}
\begin{tabular}{l|c|c|c|c|c} 
  Models & Intervals & Estimators & $N=10, \ n=100$ & $N=100, \ n=100$ & $N=100, \ n=250$ \\
\hline
 \multirow{1}{*}{Model 1} & \multirow{1}{*}{$\R$} & $\widehat{\sigma}^{2}_{K^{*},L}$ & $0.0082 \ (0.0059)$ & $0.0015 \ \ (0.0008)$  & $0.0006\ \ (0.0004)$ \\ 
 \hline
 \multirow{1}{*}{Model 2} & \multirow{1}{*}{$\R$} & $\widehat{\sigma}^{2}_{K^{*},L}$ & $0.0058 \ (0.0111)$ & $0.0007 \ \ (0.0004)$  & $0.0003\ \ (0.0002)$ \\ 
\hline
\multirow{1}{*}{Model 3} & \multirow{1}{*}{$\R$} & $\widehat{\sigma}^{2}_{K^{*},L}$ & $0.0188 \ (0.0151)$ & $0.0077 \ \ (0.0037)$  & $0.0040 \ \ (0.0036)$ \\ 
\hline
\end{tabular}
	\caption{Assessment of $\mathrm{MISEs}$ of the Hermite-based oracle estimator $\widehat{\sigma}^{2}_{K^{*},L}$ of the square of the diffusion coefficient.}
	\label{tab:MSEs_hermite}
\end{table}

From the numerical results of Table~\ref{tab:MSEs_hermite}, we observe that the Hermite-based estimator of $\sigma^{2}$ is consistent as the sample size $N$ and the length $n$ paths take larger values. 

\paragraph{Estimation of the diffusion coefficient from one path}

Consider ridge estimators of $\sigma^{2}_{|I}$ with $I=[-1,1]$. For the case of the adaptive estimators of $\sigma^{2}_{|I}$, the dimension $\w{K}$ is selected such that 
\begin{equation}
    \label{eq:Selection-N1}
    \w{K} = \underset{K\in\mathcal{K}}{\arg\min}{\gamma_{n}(\w{\sigma}^{2}_{\w{K}}) + \mathrm{pen}(K)}
\end{equation}
where $\mathrm{pen}(K) = \kappa(K+M)\log(n)/n$ with $\kappa >0$. We choose the numerical constant $\kappa = 4$ and we derive the numerical performance of the adaptive estimator of $\sigma^{2}_{|I}$.

\begin{table}[hbtp]
	\centering
\renewcommand{\arraystretch}{1.5}
\begin{tabular}{l|c|c|c|c} 
  Models & Intervals & Estimators & $n=100$ & $n=1000$ \\
\hline
  \multirow{2}{*}{Model 1} & \multirow{2}{*}{$[-1,1]$} & $\widehat{\sigma}^{2}_{\widehat{K},L}$ & $0.1751 \ \ (0.1921)$ & $0.0915 \ \ (0.1925)$ \\
& & $\widehat{\sigma}^{2}_{K^{*},L}$ & $0.1563 \ \ (0.1776)$ & $0.0783 \ \ (0.1699)$ \\
 \hline
 \multirow{2}{*}{Model 2}  & \multirow{2}{*}{$[-1,1]$} & $\widehat{\sigma}^{2}_{\widehat{K},L}$ & $0.1721 \ \ (0.3483)$ & $0.1365 \ \ (0.5905)$ \\
& & $\widehat{\sigma}^{2}_{K^{*},L}$ & $0.0987 \ \ (0.1644)$ & $0.0552 \ \ (0.2409)$ \\
\hline
\multirow{2}{*}{Model 3}  & \multirow{2}{*}{$[-1,1]$} & $\widehat{\sigma}^{2}_{\widehat{K},L}$ & $0.2184 \ \ (0.2780)$ & $0.2106 \ \ (0.5790)$ \\
& & $\widehat{\sigma}^{2}_{K^{*},L}$ & $0.1263 \ \ (0.1486)$ & $0.0751 \ \ (0.1469)$ \\
\hline
\end{tabular}
	\caption{Evaluation of MISEs of adaptive estimators $\widehat{\sigma}^{2}_{\widehat{K},L}$ built from a single diffusion path ($N=1$) for each of the three models.}
	\label{tab:MSEs_adaptive_oracle-N1}
\end{table}

Table~\ref{tab:MSEs_adaptive_oracle-N1} gives the numerical performances of both the adaptive estimator and the oracle estimator of $\sigma^{2}_{|I}$ on the compact interval $I=[-1,1]$ and from a single diffusion path. From the obtained results, we see that the estimators are numerically consistent. However, we note that the convergence is slow (increasing $n$ from $100$ to $1000$), which highlights the significant impact of the number $N$ of paths on the efficiency of the ridge estimator.

\paragraph{Comparison of the efficiency of the ridge estimator of the diffusion coefficient with its \textcolor{black}{Nadaraya$-$Watson} estimator.}

Consider the adaptive estimator $\w{\sigma}^{2}_{\w{K}}$ of the square of the diffusion coefficient buit on the real line $\R$ from a single diffusion path ($N=1$), where the dimension $\w{K}$ is selected using Equation~\eqref{eq:Selection-N1}. For the numerical assessment, we use the interval $I = [-10^6, 10^6]$ to approximate the real line $\R$, and then, use Equation~\eqref{eq:Selection-N1} for the data-driven selection of the dimension.

We want to compare the efficiency of $\w{\sigma}^{2}_{\w{K}}$ with that of the Nadaraya-Watson estimator of $\sigma^{2}$ given from a diffusion path $\bar{X} = (X_{k/n})_{1\leq k\leq n}$ and for all $x \in \R$ by
$$ S_n(x) = \frac{\sum_{k=1}^{n-1}{K\left(\frac{X_{k/n} - x}{h_n}\right)[X_{(k+1)/n} - X_{k/n}]^2/n}}{\sum_{k=1}^{n}{K\left(\frac{X_{k/n} - x}{h_n}\right)}} $$
where $K$ is a positive kernel function, and $h_n$ is the bandwidth. Thus, the estimator $S_n(x)$ is consistent under the condition $nh^{4}_{n} \rightarrow 0$ as $n$ tends to infinity (see \cite{florens1993estimating}). We use the function \texttt{ksdiff()} of the R-package \texttt{sde} to compute the \textcolor{black}{Nadaraya$-$Watson} estimator $S_n$.

\begin{table}[hbtp]
	\centering
\renewcommand{\arraystretch}{1.5}
\begin{tabular}{l|c|c} 
  Models & Ridge estimator & Nadaraya-Watson estimator \\
\hline
  \multirow{1}{*}{Model 1} & $0.0020 \ \ (0.0023)$ & $0.9377 \ \ (0.0017)$ \\ 
 \hline
 \multirow{1}{*}{Model 2}  & $0.1323 \ \ (0.0794)$ & $0.5086 \ \ (0.0885)$ \\ 
\hline
\multirow{1}{*}{Model 3}  & $0.4077 \ \ (0.1178)$ & $1.3175 \ \ (0.3039)$ \\ 
\hline
\end{tabular}
	\caption{This table shows the loss errors of the ridge estimator $\w{\sigma}^{2}_{\w{K},L}$ on $\R$ and the Nadaraya-Watson estimator $S_n$ of $\sigma^{2}$ built from a diffusion path ($N=1)$ of length $n=1000$.}
	\label{tab:MSEs_Ridge_NW}
\end{table}

We remark from the results of Table~\ref{tab:MSEs_Ridge_NW} that our ridge estimator is more efficient. Note that for the kernel estimator $S_n$, the bandwidth is computed using the rule of thumb of Scott (see \cite{odell1992multivariate}). The bandwidth is proportional to $n^{-1/(d+4)}$ where $n$ is the number of points, and $d$ is the number of spatial dimensions.

\subsection{Concluding remarks}\label{subsec:discussion}

The results of our numerical study show that our ridge estimators built both on a compact interval and on the real line are consistent as $N$ and $n$ take larger values, or as only $n$ takes larger values when the estimators are built from a single path. These results are in accordance with the theoretical results established in the previous sections. Moreover, as expected, we obtained the consistency of the Hermite-based estimators of $\sigma^{2}$ on the real line $\R$. Nonetheless, we only focus on the Hermite-based oracle estimator since we did not establish a risk bound of the corresponding adaptive estimator. Finally, we remark that the ridge estimator of $\sigma^{2}$ built from a single path performs better than its Nadaraya-Watson kernel estimator proposed in \cite{florens1993estimating} and implemented in the R-package \texttt{sde}.

\begin{figure}[hbtp]
    \centering
    \includegraphics[width=0.8\linewidth, height=0.5\textheight]{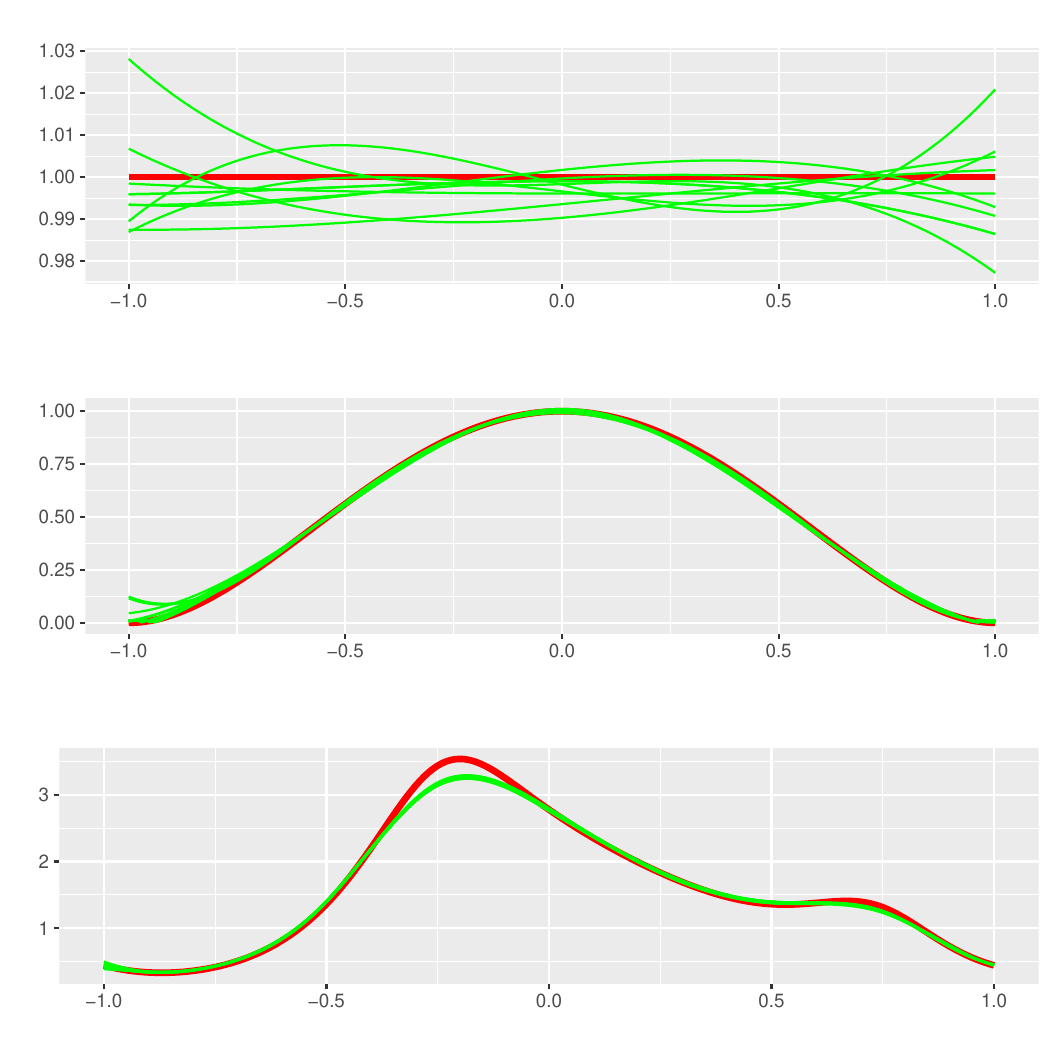}
    \caption{Bundles of $10$ estimators $\w{\sigma}^{2}_{\w{K},L}$ (in green) of the true diffusion coefficient $\sigma^{2}_{|I}$ restricted on the compact interval $I = [-1,1]$ (in red) of each of Models $1, 2, 3$ (from top to bottom) using samples of size $N=1000$ with diffusion paths of length $n=500$.}
    \label{fig:plot}
\end{figure}

\section{Conclusion}
\label{sec:Conclusion}

In this article, we have proposed ridge-type estimators of the diffusion coefficient on a compact interval from a single diffusion path. We took advantage of the local time of the diffusion process to prove the consistency of non-adaptive estimators of $\sigma^{2}$ and derive a rate of convergence of the same order than the optimal rate established in \cite{hoffmann1999lp}. We also propose an estimator of $\sigma^{2}$ on the real line from a single path. We proved its consistency using the method described in Section~\ref{subsec:Estimation-R-OnePath}, and derive a rate of convergence order $n^{-\beta/(4\beta+1)}$ over a H\"older space for the collection $[\mathbf{B}]$. Then, we extended the study to the estimation of $\sigma^{2}$ from repeated discrete observations of the diffusion process. We establish rates of convergence of the ridge estimators both on a compact interval and on $\R$. We complete the study proposing adaptive estimators of $\sigma^{2}$ on a compact interval for $N=1$ and $N\rightarrow\infty$, and on the real line $\R$ for $N\rightarrow \infty$.

A perspective on the estimation of the diffusion coefficient could be the establishment of a minimax rate of convergence of the compactly supported (square of the) diffusion coefficient from repeated discrete observations of the diffusion process. The case of the non-compactly supported diffusion coefficient may be a lot more challenging, since the transition density of the diffusion process is no longer lower-bounded. This new fact can lead to different rates of convergence depending on the considered method (see Section~\ref{sec:Estimation-N.paths}).


\section*{Acknowledgements}

I would like to thank my supervisors, Christophe Denis,  Charlotte Dion-Blanc, and Viet-Chi Tran, for their sound advice, guidance and support throughout this research project. 

\section*{Statements and Declarations}

I have no conflicts of interests to declare that are relevant to the content of this article. No funding was received to assist with the preparation of this paper. 


\section{Proofs}
\label{sec:proof}

In this section, we prove our main results of Sections~\ref{sec:Estimation-OnePath},~\ref{sec:Estimation-N.paths} and \ref{sec:AdaptiveEstimation-N.paths}. We denote by $(\mathcal{F}_t)_{t\in [0,1]}$ the natural filtration of the diffusion process $X$. To simplify our notations, we set $\Delta_n = \Delta(=1/n)$ and constants are generally denoted by $C>0$ or $c>0$ whose values can change from a line to another. Moreover, we use the notation $C_{\alpha}$ in case we need to specify the dependency of the constant $C$ on a parameter $\alpha$.

\subsection{Technical results}

Recall first some useful results on the local time and estimates of the transition density of diffusion processes.

\begin{lemme}
\label{lem:discrete-bis}
For all integer $q\geq 1$, there exists $C^{*}>0$ depending on $q$ such that for all $0\leq s<t\leq 1$,
\begin{equation*}
\E\left[\left|X_t-X_s\right|^{2q}\right]\leq C^{*}(t-s)^{q}.
\end{equation*}
\end{lemme}

The proof of Lemma~\ref{lem:discrete-bis} is provided in \cite{denis2024nonparametric}.

\begin{prop}
\label{prop:densityTransition-bis}
Under Assumptions~\ref{ass:Assumption 1}, there exist constants $c_{\sigma} >1$, $C > 1$ such that for all $t \in (0,1]$, $x \in \mathbb{R}$,
\begin{equation*}
\frac{1}{C\sqrt{t}} \exp\left(-c_{\sigma}\frac{x^2}{t}\right) \leq  p_{X}(t,x) \leq \dfrac{C}{\sqrt{t}} \exp\left(-\frac{x^2}{c_{\sigma}t}\right).
\end{equation*}
\end{prop}

The proof of Proposition~\ref{prop:densityTransition-bis}~is provided in \cite{gobet2002lan}, \textit{Proposition 1.2}.

\begin{prop}
\label{prop:approx-bis}
Let $h$ be a $L_0$-\textcolor{black}{Lipschitz} function. Then there exists $\tilde{h} \in \mathcal{S}_{K_N,M}$, such that
\begin{equation*}
|\tilde{h}(x)-h(x)| \leq C \frac{\log(N)}{K_N}, \;\; \forall x \in (-\log(N),\log(N)),
\end{equation*}
where $C >0$ depends on $L_0$, and $M$.
\end{prop}

The proof of Proposition~\ref{prop:approx-bis} is provided in \cite{denis2024nonparametric}. The finite-dimensional vector space $\mathcal{S}_{K_N,M} = \mathcal{S}_{K_N+M}$ is introduced in Section~\ref{sec:framework and assumptions}.

\begin{lemme}
\label{lem:controleSortiCompact-bis}
Under Assumption~\ref{ass:Assumption 1}, there exist $C_1,C_2 >0$ such that for all $A >0$,
\begin{equation*}
\sup_{t \in [0,1]}\P\left(\left|X_t\right|\geq A\right) 
\leq \frac{C_1}{A} \exp(-C_2A^2).
\end{equation*}
\end{lemme}

The proof of Lemma~\ref{lem:controleSortiCompact-bis}~is provided in \cite{denis2024nonparametric}, \textit{Lemma 7.3}.

\begin{lemme}
    \label{lm:LocalTimeBicontinuous}
    Under Assumption~\ref{ass:Assumption 1}, the following holds:
    \begin{equation*}
        \forall~x\in\R, ~~ \mathcal{L}^{x} = \mathcal{L}^{x_{-}} ~~ a.s.
    \end{equation*}

    where $\mathcal{L}^{x_{-}} = \underset{\varepsilon\rightarrow 0}{\lim}{\mathcal{L}^{x-\varepsilon}}$.
\end{lemme}

The result of Lemma~\ref{lm:LocalTimeBicontinuous} justifies the definition of the local time $\mathcal{L}^{x}$, for $x\in\R$, given in Equation~\eqref{eq:LocalTime}. \textcolor{black}{The proof of Lemma~\ref{lm:LocalTimeBicontinuous} is provided in appendix.}

\subsection{Proofs of Section~\ref{sec:Estimation-OnePath}}

\subsubsection{Proof of Theorem~\ref{thm:RiskBound-OnePath}}

Let $\Omega_{n,m}$ \textcolor{black}{be a random event given by 
\begin{equation*}
    \Omega_{n,m} := \underset{g\in\mathcal{S}_{m}\setminus \{0\}}{\bigcap}\left\{\left|\frac{\|g\|^{2}_{n,1}}{\|g\|^{2}_{X}}-1\right| \leq \frac{1}{2}\right\}.
\end{equation*}
On the event $\Omega_{n,m}$, the two pseudo-norms $\|.\|_{n,1}$ and $\|.\|_{X}$ are equivalent with $(1/2)\|g\|^{2}_{X} \leq \|g\|^{2}_{n,1} \leq (3/2)\|g\|^{2}_{X}$ for each function $g \in \mathcal{S}_{m}\setminus\{0\}$.
}
The proof of Theorem~\ref{thm:RiskBound-OnePath}~relies on the following lemma.

\begin{lemme}
\label{lm:Proba-OmegaComp-OnePath}
 Let $\gamma > 1$ be a real number. Under Assumption~\ref{ass:Assumption 1}, the following holds
    \begin{equation*}
      \P\left(\Omega^{c}_{n,m}\right) \leq \textcolor{black}{C\frac{m^{\alpha\gamma}}{n^{\gamma/2}}},
    \end{equation*}
 where $C>0$ is a constant depending on $\gamma$, \textcolor{black}{with $\alpha = 2$ for the collection [\textbf{B}], and $\alpha = (r+r^{\prime})/2$ for the collection [\textbf{CS$-$OB}]}.
\end{lemme}

The parameter $\gamma > 1$ has to be chosen appropriately so that we obtain a variance term of the risk of the estimator $\w{\sigma}^{2}_{m}$ \textcolor{black}{that is negligible with respect to $m^{\alpha^{\prime}}/n$ as $n$ tends to infinity} (see Theorem~\ref{thm:RiskBound-OnePath} and Corollary~\ref{coro:Rate-OnePath}). \textcolor{black}{The proof of Lemma~\ref{lm:Proba-OmegaComp-OnePath} is provided in appendix.}

\begin{proof}[{\bf Proof of Theorem~\ref{thm:RiskBound-OnePath}~}]
Recall that since $N = 1$,
$\zeta^{1}_{k\Delta}=\zeta^{1,1}_{k\Delta}+\zeta^{1,2}_{k\Delta}+\zeta^{1,3}_{k\Delta}$ is the error term of the regression model, with:
\begin{equation}
\label{eq:zeta 1}
    \zeta^{1,1}_{k\Delta}=\frac{1}{\Delta}\left[\left(\int_{k\Delta}^{(k+1)\Delta}{\sigma(X^{1}_{s})dW^{1}_{s}}\right)^2-\int_{k\Delta}^{(k+1)\Delta}{\sigma^{2}(X^{1}_{s})ds}\right],
\end{equation}
\begin{equation}
\label{eq:zeta 2}
    \zeta^{1,2}_{k\Delta}=\frac{2}{\Delta}\int_{k\Delta}^{(k+1)\Delta}{((k+1)\Delta-s)\sigma^{\prime}(X^{1}_{s})\sigma^{2}(X^{1}_{s})dW^{1}_{s}},
\end{equation}
\begin{equation}
\label{eq:zeta 3}
\zeta^{1,3}_{k\Delta}=2b(X^{1}_{k\Delta})\int_{k\Delta}^{(k+1)\Delta}{\sigma\left(X^{1}_{s}\right)dW^{1}_{s}}.
\end{equation}
Besides,
 \textcolor{black}{$R^{1}_{k\Delta}
=R^{1,1}_{k\Delta}+R^{1,2}_{k\Delta} + R^{1,3}_{k\Delta}$, with:
\begin{equation}
\label{eq: R 1}
    R^{1,1}_{k\Delta}=\frac{1}{\Delta}\left(\int_{k\Delta}^{(k+1)\Delta}{b(X^{1}_{s})ds}\right)^2, ~~ R^{1,2}_{k\Delta} = \frac{1}{\Delta}\int_{k\Delta}^{(k+1)\Delta}{((k+1)\Delta-s)\Phi(X^{1}_{s})ds}
\end{equation}
\begin{equation}
\label{eq: R 2}
R^{1,3}_{k\Delta}=\frac{2}{\Delta}\left(\int_{k\Delta}^{(k+1)\Delta}{\left(b(X^{1}_{s})-b(X^{1}_{k\Delta})\right)ds}\right)\left(\int_{k\Delta}^{(k+1)\Delta}{\sigma(X^{1}_{s})dW^{1}_{s}}\right)
\end{equation}
}
where 
\begin{equation}
  \label{eq:SigmaDerivatives}  \Phi:=2b\sigma^{\prime}\sigma+\left[\sigma^{\prime\prime}\sigma+\left(\sigma^{\prime}\right)^2\right]\sigma^{2}.
\end{equation}
By definition of the projection estimator $\widehat{\sigma}^{2}_{m}$ for each $m\in\mathcal{M}$ (see Equation~\eqref{eq:non adaptive estimator}), for all $h\in\mathcal{S}_{m,L}$, we have:
\begin{equation}
\label{eq:Ineq-Gamma}
    \gamma_{n,1}\left(\widehat{\sigma}^{2}_{m}\right)-\gamma_{n,1}(\sigma^{2}_{|I})\leq\gamma_{n,1}(h)-\gamma_{n,1}(\sigma^{2}_{|I}).
\end{equation}
Furthermore, for all $h\in\mathcal{S}_{m,L}$, 
$$\gamma_{n,1}(h)-\gamma_{n,1}(\sigma^{2}_{|I})=\left\|\sigma^{2}_{|I}-h\right\|^{2}_{n,1}+2\nu_{1}(\sigma^{2}_{|I}-h)+2\nu_{2}(\sigma^{2}_{|I}-h)+2\nu_{3}(\sigma^{2}_{|I}-h)+2\mu(\sigma^{2}_{|I}-h),$$
where, 
\begin{equation}
    \label{eq: nu and mu}
    \nu_{i}\left(h\right) = \frac{1}{n}\sum_{k=0}^{n-1}{h(X^{1}_{k\Delta})\zeta^{1,i}_{k\Delta}}, \ \ i\in\{1,2,3\}, \ \ \ \ \mu(h)=\frac{1}{n}\sum_{k=0}^{n-1}{h(X^{1}_{k\Delta})R^{1}_{k\Delta}}, 
\end{equation}
and $\zeta^{1,1}_{k\Delta}, \ \zeta^{1,2}_{k\Delta}, \ \zeta^{1,3}_{k\Delta}$ are given in Equations \eqref{eq:zeta 1}, \eqref{eq:zeta 2}, \eqref{eq:zeta 3}, and finally, $R^{1}_{k\Delta} = R^{1,1}_{k\Delta}+R^{1,2}_{k\Delta}$ given in Equations \eqref{eq: R 1} and \eqref{eq: R 2}. Then, for all $m \in \mathcal{M}$, and for all $h \in \mathcal{S}_{m,L}$, we obtain from Equation~\eqref{eq:Ineq-Gamma} that
\begin{equation*}
\left\|\widehat{\sigma}^{2}_{m}-\sigma^{2}_{|I}\right\|^{2}_{n,1}\leq \left\|h-\sigma^{2}_{|I}\right\|^{2}_{n,1}+2\nu\left(\widehat{\sigma}^{2}_{m}-h\right)+2\mu\left(\widehat{\sigma}^{2}_{m}-h\right), \ \ \mathrm{with} \ \ \nu=\nu_1+\nu_2+\nu_3.
\end{equation*}
Then, it comes,
\begin{equation}
    \label{eq:equation2-proof1}
    \mathbb{E}\left[\left\|\widehat{\sigma}^{2}_{m}-\sigma^{2}_{|I}\right\|^{2}_{n,1}\right] \leq \underset{h\in\mathcal{S}_{m,L}}{\inf}{\left\|h-\sigma^{2}_{|I}\right\|^{2}_{n}}+2\mathbb{E}\left[\nu\left(\widehat{\sigma}^{2}_{m}-h\right)\right]+2\mathbb{E}\left[\mu\left(\widehat{\sigma}^{2}_{m}-h\right)\right].
\end{equation}

Besides, for any $a,d>0$, using the inequality $xy \leq \eta x^2 + y^2/\eta$ with $\eta = a, d$, we have,
\begin{equation}
\label{eq:Devlopp-Nu-Mu}
\begin{cases}
2\nu\left(\widehat{\sigma}^{2}_{m}-h\right) \leq \frac{2}{a}\left\|\widehat{\sigma}^{2}_{m}-\sigma^{2}_{|I}\right\|^{2}_{X}+\frac{2}{a}\left\|h-\sigma^{2}_{|I}\right\|^{2}_{X}+a\underset{h\in\mathcal{S}_{m}, \ \|h\|_X=1}{\sup}{\nu^{2}(h)},\\ \\
     2\mu\left(\widehat{\sigma}^{2}_{m}-h\right) \leq \frac{2}{d}\left\|\widehat{\sigma}^{2}_{m}-\sigma^{2}_{|I}\right\|^{2}_{n,1}+\frac{2}{d}\left\|h-\sigma^{2}_{|I}\right\|^{2}_{n,1}+\frac{d}{n}\sum_{k=1}^{n}{(R^{1}_{k\Delta})^2}.
\end{cases}
\end{equation}

\subsubsection*{Upper bound of $\frac{1}{n}\sum_{k=1}^{n}{(R^{1}_{k\Delta})^2}$}

We have: 
\begin{align*}
    &~ \forall k\in[\![1,n]\!], \ R^{1}_{k\Delta}=R^{1,1}_{k\Delta}+R^{1,2}_{k\Delta}+R^{1,3}_{k\Delta} \ \mathrm{with}, \\
    &~ R^{1,1}_{k\Delta}=\frac{1}{\Delta}\left(\int_{k\Delta}^{(k+1)\Delta}{b(X^{1}_s)ds}\right)^2, \ \
	R^{1,2}_{k\Delta}=\frac{1}{\Delta}\int_{k\Delta}^{(k+1)\Delta}{((k+1)\Delta-s)\Phi(X^{1}_s)ds} \\
    &~ R^{1,3}_{k\Delta}=\frac{2}{\Delta}\left(\int_{k\Delta}^{(k+1)\Delta}{\left(b(X^{1}_s)-b(X^{1}_{k\Delta})\right)ds}\right)\left(\int_{k\Delta}^{(k+1)\Delta}{\sigma(X^{1}_s)dW^{1}_s}\right).
\end{align*} 
For all $k\in[\![1,n]\!]$, using the Cauchy-Schwarz inequality and Lemma~\eqref{lm:ConseqAssumption1}, 
\begin{align*} 
\mathbb{E}\left[\left|R^{1,1}_{k\Delta}\right|^2\right] & \leq \mathbb{E}\left[\left(\int_{k\Delta}^{(k+1)\Delta}{b^{2}(X^{1}_{k\Delta})ds}\right)^2\right]\leq\Delta\mathbb{E}\left[\int_{k\Delta}^{(k+1)\Delta}{b^{4}(X^{1}_{k\Delta})ds}\right]\leq C\Delta^2.
\end{align*}
Consider now the term $R^{1,2}_{k\Delta}$. From Equation~\eqref{eq:SigmaDerivatives}, we have $\Phi=2b\sigma^{\prime}\sigma+\left[\sigma^{\prime\prime}\sigma+\left(\sigma^{\prime}\right)^2\right]\sigma^{2}$ and according to Assumption~\ref{ass:Assumption 1}, there exists a constant $C>0$ depending on $\sigma_1$ and $\alpha$ such that
$$ \left|\Phi(X^{1}_{s})\right| \leq C\left[(2+|X^{1}_{s}|)(1+|X^{1}_{s}|^{\alpha}) + (1+|X^{1}_{s}|^{\alpha})^2\right]. $$
Then, from Lemma~\eqref{lm:ConseqAssumption1} and for all $s\in(0,1]$,
\begin{equation*}
    \E\left[\Phi^{2}(X^{1}_{s})\right] \leq C\underset{s\in(0,1]}{\sup}{\E\left[(2+|X^{1}_{s}|)^2(1+|X^{1}_{s}|^{\alpha})^2 + (1+|X^{1}_{s}|^{\alpha})^4\right]} < \infty
\end{equation*}
and
\begin{align*}
\mathbb{E}\left[\left|R^{1,2}_{k\Delta}\right|^2\right] & \leq \frac{1}{\Delta^2}\int_{k\Delta}^{(k+1)\Delta}{((k+1)\Delta-s)^2ds}\int_{k\Delta}^{(k+1)\Delta}{\mathbb{E}\left[\Phi^{2}\left(X^{1}_s\right)\right]ds}\leq C\Delta^2
\end{align*}
Finally, under Assumption~\ref{ass:Assumption 1}, from Lemma~\eqref{lm:ConseqAssumption1} and using the Cauchy-Schwarz inequality, we have
\begin{align*}
\mathbb{E}\left[\left|R^{1,3}_{k\Delta}\right|^2\right] & \leq \frac{4}{\Delta^2}\mathbb{E}\left[\Delta\int_{k\Delta}^{(k+1)\Delta}{L^{2}_{0}\left|X^{1}_s-X^{1}_{k\Delta}\right|^2ds}\left(\int_{k\Delta}^{(k+1)\Delta}{\sigma(X^{1}_s)dW_s}\right)^2\right]\\
	& \leq \frac{4}{\Delta}\sqrt{\mathbb{E}\left[L^{4}_{0}\Delta\int_{k\Delta}^{(k+1)\Delta}{\left|X^{1}_s-X^{1}_{k\Delta}\right|^4ds}\right]\mathbb{E}\left[\left(\int_{k\Delta}^{(k+1)\Delta}{\sigma(X^{1}_s)dW_s}\right)^4\right]}\\
	& \leq C\Delta^2.
\end{align*}
As a result, there exists a constant $C>0$ such that,
\begin{equation}
\label{eq:UpperBound-TimeStep}
\mathbb{E}\left[\frac{1}{n}\sum_{k=1}^{n}{(R^{1}_{k\Delta})^2}\right]\leq C\Delta^2.
\end{equation}
We set $a = d = 8$ and considering the event $\Omega_{n,m}$ on which the empirical norms $\|.\|_X$ and $\|.\|_{n,1}$ are equivalent,  we deduce from Equations~\eqref{eq:equation2-proof1},~\eqref{eq:Devlopp-Nu-Mu}~and~\eqref{eq:UpperBound-TimeStep} that,
\begin{equation}
\label{eq:equation3-proof1}
	\mathbb{E}\left[\left\|\widehat{\sigma}^{2}_{m}-\sigma^{2}_{|I}\right\|^{2}_{n,1}\one_{\Omega_{n,m}}\right]\leq 3\underset{h\in\mathcal{S}_{m}}{\inf}{\left\|h-\sigma^{2}_{|I}\right\|^{2}_{n}}+C\mathbb{E}\left(\underset{h\in\mathcal{S}_{m}, \|h\|_{X}=1}{\sup}{\nu^{2}(h)}\right)+C\Delta^2
\end{equation}
where $C>0$ is a constant depending on $\sigma_1$.

\subsection*{Upper bound of $\mathbb{E}\left(\underset{h\in\mathcal{S}_{m}, \ \|h\|_{X}=1}{\sup}{\nu^{2}(h)}\right)$}

For all $h=\sum_{\ell=0}^{m-1}{a_{\ell}\phi_{\ell}}\in\mathcal{S}_{m}$ such that $\|h\|^{2}_{X}=1$, we have $\|h\|^{2}\leq\frac{1}{\tau_0}$ (see Equation~\eqref{eq:Equiv-NormX-NormL2}) and the coordinate vector $\mathbf{a} = \left(a_{-M},\cdots,a_{K-1}\right)$ satisfies:
\begin{itemize}
    \item $\|\mathbf{a}\|^{2}_{2}\leq Cm ~~ (m = K+M)$ for \textcolor{black}{the collection [\textbf{B}]} (see \cite{denis2020ridge}, Lemma 2.6)
    \item $\|\mathbf{a}\|^{2}_{2} \leq 1/\tau_0$ for \textcolor{black}{the collection [\textbf{CS$-$OB}]} since $\|h\|^2 = \|\mathbf{a}\|^{2}_{2}$.
\end{itemize}
Furthermore, using the Cauchy$-$Schwarz inequality, we have:
\begin{equation*}
    \nu^{2}(h)=\left(\sum_{\ell=0}^{m-1}{a_{\ell}\nu\left(\phi_{\ell}\right)}\right)^2\leq\|\mathbf{a}\|^{2}_{2}\sum_{\ell=0}^{m-1}{\nu^{2}\left(\phi_{\ell}\right)}.
\end{equation*}
Thus, since $\nu=\nu_1+\nu_2+\nu_3$, for all $\ell\in[\![-M,K-1]\!]$ and for all $i\in\{1,2,3\}$,
\begin{align*}
    \mathbb{E}\left[\nu^{2}_{i}\left(\phi_{\ell}\right)\right]=&~\frac{1}{n^2}\mathbb{E}\left[\left(\sum_{k=0}^{n-1}{\phi_{\ell}\left(X^{1}_{k\Delta}\right)\zeta^{1,i}_{k\Delta}}\right)^2\right].
\end{align*}
\begin{enumerate}
    \item Case $i=1$

    Recall that $\zeta^{1,1}_{k\Delta}=\frac{1}{\Delta}\left[\left(\int_{k\Delta}^{(k+1)\Delta}{\sigma(X^{1}_{s})}dW_s\right)^2-\int_{k\Delta}^{(k+1)\Delta}{\sigma^{2}(X^{1}_{s})ds}\right]$ where $W=W^{1}$.
     We fix a initial time $s\in[0,1)$ and set $M^{s}_t=\int_{s}^{t}{\sigma(X^{1}_u)dW_u}, \ \forall t\geq s$. $(M^{s}_t)_{t\geq s}$ is a martingale and for all $t\in[s,1]$, we have:  
    \begin{align*}
      \left<M^{s},M^{s}\right>_t=\int_{s}^{t}{\sigma^{2}\left(X^{1}_{u}\right)du}.
    \end{align*}
     Then, $\zeta^{1,1}_{k\Delta}=\frac{1}{\Delta}\left(M^{k\Delta}_{(k+1)\Delta}\right)^2-\left<M^{k\Delta},M^{k\Delta}\right>_{(k+1)\Delta}$ is also a $\mathcal{F}_{k\Delta}$-martingale, and, using the Burkholder-Davis-Gundy inequality, we obtain for all $k\in[\![0,n-1]\!]$,
   \begin{equation}
     \label{eq:martingale-burkholder-davis-gundy}
      \mathbb{E}\left[\zeta^{1,1}_{k\Delta}|\mathcal{F}_{k\Delta}\right]=0, \ \ \ \mathbb{E}\left[\left(\zeta^{1,1}_{k\Delta}\right)^2|\mathcal{F}_{k\Delta}\right]\leq \frac{C}{\Delta^2}\mathbb{E}\left[\left(\int_{k\Delta}^{(k+1)\Delta}{\sigma^{2}(X^{1}_u)du}\right)^2\right]\leq C\sigma^{4}_{1}.
   \end{equation}
Then, using Equation~\eqref{eq:martingale-burkholder-davis-gundy} we have:
\begin{align*}
    \mathbb{E}\left[\nu^{2}_{1}\left(\phi_{\ell}\right)\right] = &~\frac{1}{n^2}\mathbb{E}\left[\sum_{k=0}^{n-1}{\phi^{2}_{\ell}\left(X^{1}_{k\Delta}\right)\left(\zeta^{1,1}_{k\Delta}\right)^2}\right]=\frac{1}{n^2}\mathbb{E}\left[\sum_{k=0}^{n-1}{\phi^{2}_{\ell}\left(X^{1}_{k\Delta}\right)\mathbb{E}\left[\left(\zeta^{1,1}_{k\Delta}\right)^2|\mathcal{F}_{k\Delta}\right]}\right]\\
    \leq &~\frac{C\sigma^{4}_{1}}{n^2}\mathbb{E}\left[\sum_{k=0}^{n-1}{\phi^{2}_{\ell}\left(X^{1}_{k\Delta}\right)}\right]
\end{align*}
and,
\begin{align*}
    \sum_{\ell=0}^{m-1}{\mathbb{E}\left[\nu^{2}_{1}\left(\phi_{\ell}\right)\right]}\leq \frac{C\sigma^{4}_{1}}{n^2}\mathbb{E}\left[\sum_{k=0}^{n-1}{\sum_{\ell=0}^{m-1}{\phi^{2}_{\ell}\left(X^{1}_{k\Delta}\right)}}\right].
\end{align*}
One has: 
\begin{equation*}
\begin{cases}
    \sum_{\ell=-M}^{K-1}{B^{2}_{\ell}\left(X^{1}_{\eta(s)}\right)} \leq 1 ~~ \mathrm{for} ~ \textcolor{black}{\mathrm{the} ~ \mathrm{Collection} ~ [\mathbf{B}]} ~~ (m = K + M), \\ \\
    \sum_{\ell = 0}^{m-1}{\phi^{2}_{\ell}(X^{1}_{\eta(s)})} \leq \textcolor{black}{Cm^{r} ~~ \mathrm{for} ~ \mathrm{the} ~ \mathrm{collection} ~ [\mathbf{CS-OB}]}.
\end{cases}
\end{equation*}
Thus, it \textcolor{black}{follows} that 
\begin{itemize}
    \item $\sum_{\ell=-M}^{K-1}{\mathbb{E}\left[\nu^{2}_{1}\left(B_{\ell}\right)\right]}\leq C/n ~~ \mathrm{for} ~ \mathrm{the} ~ \mathrm{Spline} ~ \mathrm{basis,} $
    \item $\sum_{\ell=0}^{m-1}{\mathbb{E}\left[\nu^{2}_{1}\left(\phi_{\ell}\right)\right]}\leq \textcolor{black}{Cm^{r}/n ~~ \mathrm{for} ~ \mathrm{the} ~ \mathrm{collection} ~ [\textbf{CS-OB}]},$
\end{itemize}
and,
\begin{equation}
    \label{eq:upper bound of nu - case i=1}
    \mathbb{E}\left(\underset{h\in\mathcal{S}_{m}, \ \|h\|^{2}_{X}=1}{\sup}{\nu^{2}_{1}(h)}\right)\leq \textcolor{black}{C\frac{m^{\alpha^{\prime}}}{n}}
\end{equation}
where $C>0$ is a constant depending on $\sigma_{1}$ and the basis, \textcolor{black}{with $\alpha^{\prime} = 1$ for the collection [\textbf{B}], and $\alpha^{\prime} = r$ for the collection [\textbf{CS$-$OB}]}.
\item Case $i=2$

\textcolor{black}{We} have $\zeta^{1,2}_{k\Delta}=\frac{2}{\Delta}\int_{k\Delta}^{(k+1)\Delta}{\left((k+1)\Delta-s\right)\sigma^{\prime}\left(X^{1}_{s}\right)\sigma^{2}\left(X^{1}_{s}\right)dW_s}$ and,
\begin{align*}
    \mathbb{E}\left[\nu^{2}_{2}\left(\phi_{\ell}\right)\right] = &~4\mathbb{E}\left[\left(\sum_{k=0}^{n-1}{\phi_{\ell}\left(X^{1}_{k\Delta}\right)\int_{k\Delta}^{(k+1)\Delta}{(k+1)\Delta-s)}\sigma^{\prime}\left(X^{1}_{s}\right)\sigma^{2}\left(X^{1}_{s}\right)dW_s}\right)^2\right]\\
    =&~4\mathbb{E}\left[\left(\int_{0}^{1}{\phi_{\ell}\left(X^{1}_{\eta(s)}\right)(\eta(s)+\Delta-s)\sigma^{\prime}\left(X^{1}_{s}\right)\sigma^{2}\left(X^{1}_{s}\right)dW_s}\right)^2\right]\\
    \leq &~C\sigma^{4}_{1}\Delta^{2}\mathbb{E}\left[\int_{0}^{1}{\phi^{2}_{\ell}\left(X^{1}_{\eta(s)}\right)ds}\right]
\end{align*}
where $C>0$ is a constant. We deduce for both the spline basis and any orthonormal basis that there exists a constant $C>0$ depending on $\sigma_{1}$ such that:
\begin{equation}
    \label{eq:upper bound of nu - case i=2}
    \mathbb{E}\left(\underset{h\in\mathcal{S}_{m}, \ \|h\|^{2}_{X}=1}{\sup}{\nu^{2}_{2}(h)}\right)\leq \textcolor{black}{C\frac{m^{\alpha^{\prime}}}{n^2}}.
\end{equation}

\item Case $i=3$

We have $\zeta^{1,3}_{k\Delta}=2b\left(X^{1}_{k\Delta}\right)\int_{k\Delta}^{(k+1)\Delta}{\sigma\left(X^{1}_{s}\right)dW_s}$ and,
\begin{align*}
    \mathbb{E}\left[\nu^{2}_{3}\left(\phi_{\ell}\right)\right] = &~\frac{4}{n^{2}}\mathbb{E}\left[\left(\int_{0}^{1}{\phi_{\ell}\left(X^{1}_{\eta(s)}\right)b\left(X^{1}_{\eta(s)}\right)\sigma\left(X^{1}_{s}\right)dW_s}\right)^2\right]\\
    \leq &~ \frac{4\sigma^{2}_{1}}{n^2}\mathbb{E}\left[\int_{0}^{1}{\phi^{2}_{\ell}\left(X^{1}_{\eta(s)}\right)b^{2}\left(X^{1}_{\eta(s)}\right)ds}\right]
\end{align*}
Since for all $x\in\mathbb{R}, \ b^{2}(x)\leq C_{0}(1+x^2)$ and $\underset{t\in[0,1]}{\sup}{\mathbb{E}\left(|X_t|^{2}\right)}<\infty$, there exists a constant $C>0$ depending on $\sigma_{1}$ such that:
\begin{equation}
    \label{eq:upper bound of nu - case i=3}
    \mathbb{E}\left(\underset{h\in\mathcal{S}_{m}, \ \|h\|^{2}_{X}=1}{\sup}{\nu^{2}_{3}(h)}\right)\leq \textcolor{black}{C\frac{m^{\alpha^{\prime}}}{n^2}}.
\end{equation}
\end{enumerate}
We finally obtain from Equations~\eqref{eq:upper bound of nu - case i=1}, \eqref{eq:upper bound of nu - case i=2} and \eqref{eq:upper bound of nu - case i=3} that there exists a constant $C>0$ depending on $\sigma_{1}$ such that:
\begin{equation}
    \label{eq:upper bound of nu - all cases}
    \mathbb{E}\left(\underset{h\in\mathcal{S}_{m}, \ \|h\|^{2}_{X}=1}{\sup}{\nu^{2}(h)}\right)\leq \textcolor{black}{C\frac{m^{\alpha^{\prime}}}{n}}.
\end{equation}
We deduce from Equations~\eqref{eq:equation3-proof1}~and~\eqref{eq:upper bound of nu - all cases} that there exists a constant $C>0$ depending on $\sigma_1$ such that,
   \begin{equation*}
       \E\left[\|\w{\sigma}^{2}_{m} - \sigma^{2}_{|I}\|^{2}_{n,1}\one_{\Omega_{n,m}}\right] \leq 3\underset{h\in\mathcal{S}_{m,L}}{\inf}{\|\sigma^{2}_{|I} - h\|^{2}_{n}} + C\left(\textcolor{black}{\frac{m^{\alpha^{\prime}}}{n}} + \Delta^{2}\right).
   \end{equation*}
 For $n$ large enough, we have $\|\w{\sigma}^{2}_{m} - \sigma^{2}_{|I}\|^{2}_{\infty} \leq 2mL$ since $\|\w{\sigma}^{2}_{m}\|_{\infty}\leq\sqrt{mL}$. Then, from Lemma~\ref{lm:Proba-OmegaComp-OnePath}~and for all $m\in\mathcal{M}$, there exists a constant $C>0$ depending on $\sigma_1$ such that
 \begin{align*}
     \mathbb{E}\left[\left\|\widehat{\sigma}^{2}_{m}-\sigma^{2}_{|I}\right\|^{2}_{n,1}\right]&=\mathbb{E}\left[\left\|\widehat{\sigma}^{2}_{m}-\sigma^{2}_{|I}\right\|^{2}_{n,1}\one_{\Omega_{n,m}}\right]+\mathbb{E}\left[\left\|\widehat{\sigma}^{2}_{m}-\sigma^{2}_{|I}\right\|^{2}_{n,1}\one_{\Omega^{c}_{n,m}}\right]\\
	&\leq\mathbb{E}\left[\left\|\widehat{\sigma}^{2}_{m}-\sigma^{2}_{|I}\right\|^{2}_{n,1}\one_{\Omega_{n,m}}\right]+2mL\mathbb{P}\left(\Omega^{c}_{n,m}\right)\\
      &\leq 3\underset{h\in\mathcal{S}_{m,L}}{\inf}{\|\sigma^{2}_{|I} - h\|^{2}_{n}} + C\left(\textcolor{black}{\frac{m^{\alpha^{\prime}}}{n} + \frac{m^{\alpha\gamma+1}L}{n^{\gamma/2}}} + \Delta^{2}\right).  
\end{align*}
Since the pseudo-norms $\|.\|_{n,1}$ and $\|.\|_{X}$ are equivalent on the event $\Omega_{n,m}$, then, using Lemma~\ref{lm:Proba-OmegaComp-OnePath}, there exists a constant $C>0$ depending on $\sigma_1$ such that
\begin{align*}
    \mathbb{E}\left[\left\|\widehat{\sigma}^{2}_{m}-\sigma^{2}_{|I}\right\|^{2}_{X}\right] = & \mathbb{E}\left[\left\|\widehat{\sigma}^{2}_{m}-\sigma^{2}_{|I}\right\|^{2}_{X}\one_{\Omega_{n,m}}\right] + \mathbb{E}\left[\left\|\widehat{\sigma}^{2}_{m}-\sigma^{2}_{|I}\right\|^{2}_{X}\one_{\Omega^{c}_{n,m}}\right]\\
    \leq & 8\mathbb{E}\left[\left\|\widehat{\sigma}^{2}_{m}-\sigma^{2}_{|I}\right\|^{2}_{n,1}\right] + 10\underset{h\in\mathcal{S}_{m}}{\inf}{\left\|\sigma^{2}_{|I}-h\right\|^{2}_{n}} + 2mL\P\left(\Omega^{c}_{n,m}\right) \\
    \leq & 34\underset{h\in\mathcal{S}_{m,L}}{\inf}{\left\|h-\sigma^{2}_{|I}\right\|^{2}_{n}}+C\left(\textcolor{black}{\frac{m^{\alpha^{\prime}}}{n} + \frac{m^{\alpha\gamma+1}L}{n^{\gamma/2}}} +\Delta^{2}\right).
\end{align*}
Finally, since the estimator $\w{\sigma}^{2}_{m}$ is built from a diffusion path $\bar{X}^{1}$ independent of the diffusion process $X$, and from Equations~\eqref{eq:Equiv-NormX-NormL2}~and~\eqref{eq:equivalence between the L2 norm and the empirical norm}, the pseudo-norm $\|.\|_X$ depending on the process $X$ and the empirical norm $\|.\|_n$ are equivalent ($\forall~h\in \mathbb{L}^{2}(I), ~ \|h\|^{2}_{n}\leq (\tau_1/\tau_0)\E\left[\|h\|^{2}_{X}\right]$), there exists a constant $C>0$ depending on $\sigma_1, \tau_0$ and $\tau_1$ such that
\begin{align*}
    \mathbb{E}\left[\left\|\widehat{\sigma}^{2}_{m}-\sigma^{2}_{|I}\right\|^{2}_{n}\right] \leq & \frac{34\tau_1}{\tau_0}\underset{h\in\mathcal{S}_{m,L}}{\inf}{\left\|h-\sigma^{2}_{|I}\right\|^{2}_{n}}+C\left(\textcolor{black}{\frac{m^{\alpha^{\prime}}}{n} + \frac{m^{\alpha\gamma+1}L}{n^{\gamma/2}}} +\Delta^{2}\right).
\end{align*}
\end{proof}

\subsubsection{Proof of Theorem~\ref{thm:ConstFromOnePath}}

\begin{proof}
Since $L=\log^{2}(n)$. 
From Equation~\eqref{eq:equation2-proof1} (Proof of Theorem~\ref{thm:RiskBound-OnePath}), for all $h\in\mathcal{S}_{m,L}$,
\textcolor{black}{
\begin{equation}
\label{eq:FirstRistBound-OnePath}
    \mathbb{E}\left[\left\|(\widehat{\sigma}^{2}_{m,L}-\sigma_{|\mathcal{I}}^{2})\right\|^{2}_{n,1}\right] \leq\underset{h\in\mathcal{S}_{m,L}}{\inf}{\left\|h-\sigma_{|\mathcal{I}}^{2}\right\|^{2}_{n}} + 2\sum_{i=1}^{3}{\mathbb{E}\left[\nu_i\left(\widehat{\sigma}^{2}_{m}-h\right)\right]} + 2\mathbb{E}\left[\mu\left(\widehat{\sigma}^{2}_{m}-h\right)\right]
\end{equation}
}
where $\nu_i, ~ i=1,2,3$ and $\mu$ are given in Equation~\eqref{eq: nu and mu}. For all $i\in\{1,2,3\}$ and for all $h\in\mathcal{S}_{m,L}$, one has
\begin{equation}
\label{eq:upper-bound nu.i}
    \mathbb{E}\left[\nu_i\left(\widehat{\sigma}^{2}_{m,L}-h\right)\right]\leq\sqrt{2m\log^{2}(n)}\sqrt{\sum_{\ell=0}^{m-1}{\mathbb{E}\left[\nu^{2}_{i}(\phi_{\ell})\right]}}.
\end{equation}
\begin{enumerate}
    \item Upper bound of $\sum_{\ell=0}^{m-1}{\mathbb{E}\left[\nu^{2}_{1}(\phi_{\ell})\right]}$
    
    According to Equation~\eqref{eq: nu and mu}, we have
    \begin{equation*}
        \forall \ell\in[\![0,m-1]\!], \ \nu_1(\phi_{\ell})=\frac{1}{n}\sum_{k=0}^{n-1}{\phi_{\ell}(X^{1}_{k\Delta})\zeta^{1,1}_{k\Delta}}
    \end{equation*}
    where $\zeta^{1,1}_{k\Delta}=\frac{1}{\Delta}\left[\left(\int_{k\Delta}^{(k+1)\Delta}{\sigma(X^{1}_{s})dW_s}\right)^2-\int_{k\Delta}^{(k+1)\Delta}{\sigma^{2}(X^{1}_{s})ds}\right]$ is a martingale satisfying
    \begin{equation*}
        \mathbb{E}\left[\zeta^{1,1}_{k\Delta}|\mathcal{F}_{k\Delta}\right]=0 \ \ \mathrm{and} \ \ \mathbb{E}\left[\left(\zeta^{1,1}_{k\Delta}\right)^2|\mathcal{F}_{k\Delta}\right]\leq\frac{1}{\Delta^2}\mathbb{E}\left[\left(\int_{k\Delta}^{(k+1)\Delta}{\sigma^{2}(X^{1}_{s})ds}\right)^2\right]\leq C\sigma^{4}_{1}
    \end{equation*}
    with $C>0$ a constant, $W=W^{1}$ and $(\mathcal{F}_t)_{t\geq 0}$ the natural filtration of the martingale $(M_t)_{t\in[0,1]}$ given for all $t\in[0,1]$ by $M_{t}=\int_{0}^{t}{\sigma(X^{1}_s)dW_s}$. We derive that
    \begin{align*}
        \sum_{\ell=0}^{m-1}{\mathbb{E}\left[\nu^{2}_{1}(\phi_{\ell})\right]}=&\frac{1}{n^2}\sum_{\ell=0}^{m-1}{\mathbb{E}\left[\left(\sum_{k=0}^{n-1}{\phi_{\ell}(X^{1}_{k\Delta})\zeta^{1,1}_{k\Delta}}\right)^2\right]}=\frac{1}{n^2}\mathbb{E}\left[\sum_{k=0}^{n-1}{\sum_{\ell=0}^{m-1}{\phi^{2}_{\ell}(X^{1}_{k\Delta})\left(\zeta^{1,1}_{k\Delta}\right)^2}}\right]
    \end{align*}
    since for all integers $k, k^{\prime}$ such that $k > k^{\prime} \geq 0$, we have 
    \begin{align*}
        \E\left[\phi_{\ell}(X^{1}_{k\Delta})\zeta^{1,1}_{k\Delta}\phi_{\ell}(X^{1}_{k^{\prime}\Delta})\zeta^{1,1}_{k^{\prime}\Delta}|\mathcal{F}_{k\Delta}\right] = \phi_{\ell}(X^{1}_{k\Delta})\zeta^{1,1}_{k^{\prime}\Delta}\phi_{\ell}(X^{1}_{k^{\prime}\Delta})\E\left[\zeta^{1,1}_{k\Delta}|\mathcal{F}_{k\Delta}\right] = 0.
    \end{align*}
    For each $k\in[\![0,n-1]\!]$, we have 
    $$\begin{cases} 
    \sum_{\ell=0}^{m-1}{\phi_{\ell}(X^{1}_{k\Delta})} = \sum_{\ell=-M}^{K-1}{B_{\ell}(X^{1}_{k\Delta})} =1 ~~ \mathrm{for ~ the ~ spline ~ basis} \\ \\
    \sum_{\ell=0}^{m-1}{\phi_{\ell}(X^{1}_{k\Delta})} \leq Cm ~~ \mathrm{For ~ an ~ orthonormal ~ basis ~ with} ~ C=\underset{0 \leq\ell\leq m-1}{\max}{\|\phi_{\ell}\|_{\infty}}.
    \end{cases}
    $$ 
    Finally, there exists a constant $C>0$ such that
    \begin{align*}
        \sum_{\ell=0}^{m-1}{\mathbb{E}\left[\nu^{2}_{1}(\phi_{\ell})\right]} \leq \begin{cases}
        \frac{C}{n} ~~ \mathrm{for ~ the ~ spline ~ basis} \\ \\
        C\frac{m}{n} ~~ \mathrm{for ~ an ~ orthonormal ~ basis}.
        \end{cases}
    \end{align*}
    
    \item Upper bound of $\sum_{\ell=0}^{m-1}{\mathbb{E}\left[\nu^{2}_{2}(\phi_{\ell})\right]}$
    
    For all $k\in[\![0,n-1]\!]$ and for all $s\in[0,1]$, set $\eta(s)=k\Delta$ if $s\in[k\Delta,(k+1)\Delta)$. We have:
    \begin{align*}
        \sum_{\ell=0}^{m-1}{\mathbb{E}\left[\nu^{2}_{2}(\phi_{\ell})\right]}&=4\sum_{\ell=0}^{m-1}{\mathbb{E}\left[\left(\sum_{k=0}^{n-1}{\int_{k\Delta}^{(k+1)\Delta}{\phi_{\ell}(X^{1}_{k\Delta})((k+1)\Delta-s)\sigma^{\prime}(X^{1}_{s})\sigma^{2}(X^{1}_{s})dW_s}}\right)^2\right]}\\
        &=4\sum_{\ell=0}^{m-1}{\mathbb{E}\left[\left(\int_{0}^{1}{\phi_{\ell}(X^{1}_{\eta(s)})(\eta(s)+\Delta-s)\sigma^{\prime}(X^{1}_{s})\sigma^{2}(X^{1}_{s})dW_s}\right)^2\right]}.
    \end{align*}
    We conclude that
    \begin{align*}
        \sum_{\ell=0}^{m-1}{\mathbb{E}\left[\nu^{2}_{2}(\phi_{\ell})\right]} \leq \begin{cases}
        \frac{C}{n^2} ~~ \mathrm{for ~ the ~ spline ~ basis} \\ \\
        C\frac{m}{n^2} ~~ \mathrm{for ~ an ~ orthonormal ~ basis}.
        \end{cases}
    \end{align*}
     where the constant $C>0$ depends on the diffusion coefficient and the upper bound of the basis functions.
     
    \item Upper bound of $\sum_{\ell=0}^{m-1}{\mathbb{E}\left[\nu^{3}_{2}(\phi_{\ell})\right]}$
    
    We have:
    \begin{align*}
        \sum_{\ell=0}^{m-1}{\mathbb{E}\left[\nu^{2}_{3}(\phi_{\ell})\right]}&=\frac{4}{n^2}\sum_{\ell=0}^{m-1}{\mathbb{E}\left[\left(\sum_{k=0}^{n-1}{\int_{k\Delta}^{(k+1)\Delta}{\phi_{\ell}(X^{1}_{k\Delta})b(X^{1}_{k\Delta})\sigma(X^{1}_{s})dW_s}}\right)^2\right]}\\
        &=\frac{4}{n^2}\sum_{\ell=0}^{m-1}{\mathbb{E}\left[\left(\int_{0}^{1}{\phi_{\ell}(X^{1}_{\eta(s)})b(X^{1}_{\eta(s)})\sigma(X^{1}_{s})dW_s}\right)^2\right]}\\
        &\leq\frac{4}{n^2}\mathbb{E}\left[\int_{0}^{1}{\sum_{\ell=0}^{m-1}{\phi^{2}_{\ell}(X^{1}_{\eta(s)})b^{2}(X^{1}_{\eta(s)})\sigma^{2}(X^{1}_{s})ds}}\right].
    \end{align*}
    Since for all $x\in\mathbb{R}, \ b(x)\leq C_0(1+x^2)$ and $\underset{t\in[0,1]}{\sup}{\mathbb{E}\left(|X_t|^4\right)}<\infty$, there exists a constant $C>0$ depending on the diffusion coefficient such that
    \begin{align*}
        \sum_{\ell=0}^{m-1}{\mathbb{E}\left[\nu^{2}_{3}(\phi_{\ell})\right]} \leq \begin{cases}
        \frac{C}{n^2} ~~ \mathrm{for ~ the ~ spline ~ basis} \\ \\
        C\frac{m}{n^2} ~~ \mathrm{for ~ an ~ orthonormal ~ basis}.
        \end{cases}
    \end{align*}
\end{enumerate}
We finally deduce that from Equations~\eqref{eq:FirstRistBound-OnePath}~and~\eqref{eq:upper-bound nu.i}~ that for all $h\in\mathcal{S}_{m,L}$,
\textcolor{black}{
\begin{equation}
\label{eq:First-UpperBounds-OnePath}
\begin{cases}
    \mathbb{E}\left[\|(\widehat{\sigma}^{2}_{m,L}-\sigma_{|\mathcal{I}}^{2})\|^{2}_{n,1}\right]\leq \underset{h\in\mathcal{S}_{m}}{\inf}{\|h-\sigma_{|\mathcal{I}}^{2}\|^{2}_{n}}+C\sqrt{\frac{m\log^{2}(n)}{n}}+2\mathbb{E}\left[\mu(\widehat{\sigma}^{2}_{m,L}-h)\right] ~~ {\bf [B]} \\\\
    \mathbb{E}\left[\|(\widehat{\sigma}^{2}_{m,L}-\sigma_{|\mathcal{I}}^{2})\|^{2}_{n,1}\right]\leq \underset{h\in\mathcal{S}_{m}}{\inf}{\|h-\sigma_{|\mathcal{I}}^{2}\|^{2}_{n}}+C\sqrt{\frac{m^{2}\log^{2}(n)}{n}}+2\mathbb{E}\left[\mu(\widehat{\sigma}^{2}_{m,L}-h)\right] ~~ {\bf [OB]}
\end{cases}
\end{equation}
}
where $C>0$ is a constant. It remains to obtain an upper bound of the term $\mu(\widehat{\sigma}^{2}_{m,L}-h)$. For all $a>0$ and for all $h\in\mathcal{S}_{m,L}$,
\begin{align*}
    2\mu\left(\widehat{\sigma}^{2}_{m,L}-h\right) \leq & \frac{2}{a}\left\|\widehat{\sigma}^{2}_{m,L}-\sigma_{\textcolor{black}{|\mathcal{I}}}^{2}\right\|^{2}_{n,1}+\frac{2}{a}\left\|h-\sigma_{\textcolor{black}{|\mathcal{I}}}^{2}\right\|^{2}_{n,1}+\frac{a}{n}\sum_{k=0}^{n-1}{\left(R^{1}_{k\Delta}\right)^2}\\
    2\mathbb{E}\left[\mu\left(\widehat{\sigma}^{2}_{m,L}-h\right)\right] \leq & \frac{2}{a}\mathbb{E}\left\|\widehat{\sigma}^{2}_{m,L}-\sigma_{\textcolor{black}{|\mathcal{I}}}^{2}\right\|^{2}_{n,1}+\frac{2}{a}\underset{h\in\mathcal{S}_{m}}{\inf}{\|h-\sigma_{\textcolor{black}{|\mathcal{I}}}^{2}\|^{2}_{n}} +\frac{a}{n}\sum_{k=0}^{n-1}{\mathbb{E}\left[\left(R^{1}_{k\Delta}\right)^2\right]}.
\end{align*}
Using Equations~\eqref{eq:UpperBound-TimeStep},~\eqref{eq:First-UpperBounds-OnePath}~and setting $a=4$, we deduce that there exists constant $C>0$ depending on $\sigma_{1}$ such that,
\textcolor{black}{
\begin{equation}
\label{eq:pseudo-norm-loss-error-OnePath}
   \begin{cases}
    \mathbb{E}\left[\|\widehat{\sigma}^{2}_{m,L}-\sigma_{|\mathcal{I}}^{2}\|^{2}_{n,1}\right]\leq \underset{h\in\mathcal{S}_{m}}{\inf}{\|h-\sigma_{|\mathcal{I}}^{2}\|^{2}_{n}}+C\sqrt{\frac{m\log^{2}(n)}{n}} ~~ {\bf [B]} \\ \\
    \mathbb{E}\left[\|\widehat{\sigma}^{2}_{m,L}-\sigma_{|\mathcal{I}}^{2}\|^{2}_{n,1}\right]\leq \underset{h\in\mathcal{S}_{m}}{\inf}{\|h-\sigma_{|\mathcal{I}}^{2}\|^{2}_{n}}+C\sqrt{\frac{m^2\log^{2}(n)}{n}} ~~ {\bf [OB]}.
   \end{cases}
\end{equation}
}

\textcolor{black}{For the case of bases of compactly supported functions, we proceed to the dilation of the bases so that $\sigma^{2}$ can be estimated on the compact interval $[-\log(n), \log(n)]$. Then, we use the folowing decomposition:
\begin{align*}
    \E\left[\left\|\w{\sigma}^{2}_{m,L} - \sigma^{2}\right\|^{2}_{n,1}\right] \leq &~ \E\left[\left\|(\w{\sigma}^{2}_{m,L} - \sigma^{2})\one_{[-\log(n),\log(n)]}\right\|^{2}_{n,1}\right] + \E\left[\left\|(\w{\sigma}^{2}_{m,L} - \sigma^{2})\one_{[-\log(n),\log(n)]^{c}}\right\|^{2}_{n,1}\right] \\
    \leq &~ \E\left[\left\|(\w{\sigma}^{2}_{m,L} - \sigma^{2})\one_{[-\log(n),\log(n)]}\right\|^{2}_{n,1}\right] + 4\textcolor{black}{L}\underset{t\in[0,1]}{\sup}{\P(|X_t|>\log(n))},
\end{align*}
and we deduce from Equation~\eqref{eq:pseudo-norm-loss-error-OnePath} that
\begin{equation*}
   \begin{cases}
    \mathbb{E}\left[\|(\widehat{\sigma}^{2}_{m,L}-\sigma_{|\mathcal{I}}^{2})\one_{[-\log(n),\log(n)]}\|^{2}_{n,1}\right]\leq \underset{h\in\mathcal{S}_{m}}{\inf}{\|h-\sigma_{|\mathcal{I}}^{2}\|^{2}_{n}}+C\sqrt{\frac{m\log^{2}(n)}{n}} ~~ {\bf [B]} \\ \\
    \mathbb{E}\left[\|(\widehat{\sigma}^{2}_{m,L}-\sigma_{|\mathcal{I}}^{2})\one_{[-\log(n),\log(n)]}\|^{2}_{n,1}\right]\leq \underset{h\in\mathcal{S}_{m}}{\inf}{\|h-\sigma_{|\mathcal{I}}^{2}\|^{2}_{n}}+C\sqrt{\frac{m^2\log^{2}(n)}{n}} ~~ {\bf [OB]},
   \end{cases}
\end{equation*}
and
} 
\begin{equation}
\label{eq:pseudo-norm-loss-error-OnePath-bis}
   \begin{cases}
    \mathbb{E}\left[\|\widehat{\sigma}^{2}_{m,L}-\sigma_{\textcolor{black}{|\mathcal{I}}}^{2}\|^{2}_{n,1}\right]\leq \underset{h\in\mathcal{S}_{m}}{\inf}{\|h-\sigma_{\textcolor{black}{|\mathcal{I}}}^{2}\|^{2}_{n}}+C\sqrt{\frac{m\log^{2}(n)}{n}} + 2\log^{2}(n)\underset{t\in(0,1]}{\sup}{\P(|X_t|>A_n)} ~~ {\bf [B]} \\ \\
    \mathbb{E}\left[\|\widehat{\sigma}^{2}_{m,L}-\sigma_{\textcolor{black}{|\mathcal{I}}}^{2}\|^{2}_{n,1}\right]\leq \underset{h\in\mathcal{S}_{m}}{\inf}{\|h-\sigma_{\textcolor{black}{|\mathcal{I}}}^{2}\|^{2}_{n}}+C\sqrt{\frac{m^2\log^{2}(n)}{n}} + 2\log^{2}(n)\underset{t\in(0,1]}{\sup}{\P(|X_t|>A_n)} ~~ {\textcolor{black}{\bf [OB]}}.
   \end{cases}
\end{equation}
From Proposition~\ref{prop:approx-bis}, $\underset{t\in(0,1]}{\sup}{\P(|X_t|>\log(n))} \leq \log^{-1}(n)\exp(-c\log^{2}(n))$ with $c>0$ a constant. Then, we obtain from Equation~\eqref{eq:pseudo-norm-loss-error-OnePath-bis} that
\begin{equation*}
    \begin{cases}
    \mathbb{E}\left[\|\widehat{\sigma}^{2}_{m,L}-\sigma_{\textcolor{black}{|\mathcal{I}}}^{2}\|^{2}_{n,1}\right]\leq \underset{h\in\mathcal{S}_{m}}{\inf}{\|h-\sigma_{\textcolor{black}{|\mathcal{I}}}^{2}\|^{2}_{n}}+C\sqrt{\frac{m\log^{2}(n)}{n}} ~~ ~~ {\bf [B]} \\ \\
    \mathbb{E}\left[\|\widehat{\sigma}^{2}_{m,L}-\sigma_{\textcolor{black}{|\mathcal{I}}}^{2}\|^{2}_{n,1}\right]\leq \underset{h\in\mathcal{S}_{m}}{\inf}{\|h-\sigma_{\textcolor{black}{|\mathcal{I}}}^{2}\|^{2}_{n}}+C\sqrt{\frac{m^2\log^{2}(n)}{n}} ~~ ~~ {\textcolor{black}{\bf [OB]}}.
   \end{cases}
\end{equation*}
\end{proof}

\subsection{Proof of Section~\ref{sec:Estimation-N.paths}}

The following lemma allows us to obtain a risk bound of $\w{\sigma}^{2}_{m,L}$ defined with the empirical norm $\|.\|_n$ from the risk bound defined from the pseudo norm $\|.\|_{n,N}$.
\begin{lemme}
    \label{lem:Relation-Nn-n}
    Let $\w{\sigma}^{2}_{m,L}$ be the truncated projection estimator on $\R$ of $\sigma^{2}$ over the subspace $\mathcal{S}_{m,L}$. Suppose that $L = \log^{2}(N\textcolor{black}{n}), ~ N>1$. Under Assumption~\ref{ass:Assumption 1}, there exists a constant $C>0$ independent of $m$ and $N$ such that
    $$ \E\left[\left\|\w{\sigma}^{2}_{m,L} - \sigma^{2}\right\|^{2}_{n,N}\right] - 2\E\left[\left\|\w{\sigma}^{2}_{m,L} - \sigma^{2}\right\|^{2}_{n}\right] \leq C \frac{m^2\log^{3}(N\textcolor{black}{n})}{N}.$$
\end{lemme}

The proof of Lemma~\ref{lem:Relation-Nn-n} is provided in \cite{denis2020ridge}, \textit{Theorem 3.3}. The proof uses the independence of the copies $\bar{X}^{1},\ldots,\bar{X}^{N}$ of the process $X$ at discrete times, and the Bernstein inequality.

\subsubsection{Proof of Theorem~\ref{thm:RiskBound-CompactSupport}~}

For fixed $n$ and $N$ in $\mathbb{N}^{*}$, we set for all $m\in\mathcal{M}$,
\begin{equation}
    \label{equivalence-set}
    \Omega_{n,N,m}:=\underset{h\in\mathcal{S}_{m}\setminus\{0\}}{\bigcap}{\left\{\left|\frac{\|h\|^{2}_{n,N}}{\|h\|^{2}_{n}}-1\right|\leq\frac{1}{2}\right\}}.
\end{equation}

As we can see, the empirical norms $\|h\|_{n,N}$ and $\|h\|_n$ of any function $h\in\mathcal{S}_{m}\setminus\{0\}$ are equivalent on $\Omega_{n,N,m}$. More precisely, on the set $\Omega_{n,N,m}$, for all $h\in\mathcal{S}_{m}\setminus\{0\}$, we have : $\frac{1}{2}\|h\|^{2}_{n}\leq\|h\|^{2}_{n,N}\leq\frac{3}{2}\|h\|^{2}_{n}$. We have the following result:

\begin{lemme}
\label{lm:Proba-OmegaComp}
Under Assumption~\ref{ass:Assumption 1}, the following holds:
\begin{itemize}
    \item If $n \geq N$ or $n \propto N$, then $m \in \textcolor{black}{\mathcal{M} = \left\{1,\ldots,\left\lfloor\sqrt{N}/\log(Nn)\right\rfloor\right\}}$ and,
    \begin{align*}
      \mathbb{P}\left(\Omega^{c}_{n,N,m}\right)\leq 2\exp(-C\sqrt{N}).
    \end{align*}
\item If $n \leq N$, then $m \in \textcolor{black}{\mathcal{M}= \left\{1,\ldots,\left\lfloor\sqrt{n}/\log(Nn)\right\rfloor\right\}}$ and
     \begin{align*}
      \mathbb{P}\left(\Omega^{c}_{n,N,m}\right) \leq 2\exp(-C\sqrt{n})
    \end{align*}
\end{itemize}

where $C>0$ is a constant.
\end{lemme}

\textcolor{black}{The proof of Lemma~\ref{lm:Proba-OmegaComp} is provided in appendix.}

\begin{proof}[\textbf{Proof of Theorem}~\ref{thm:RiskBound-CompactSupport}~]
The proof of Theorem~\ref{thm:RiskBound-CompactSupport} extends the proof of Theorem~\ref{thm:RiskBound-OnePath} when $N$ tends to infinity. Then, we deduce from Equation~\eqref{eq:equation3-proof1} that
\begin{equation}
\label{eq:upper-bound-Nn-1}
  \mathbb{E}\left[\left\|\widehat{\sigma}^{2}_{m}-\sigma^{2}_{|I}\right\|^{2}_{n,N}\one_{\Omega_{n,N,m}}\right]\leq 3\underset{h\in\mathcal{S}_{m,L}}{\inf}{\left\|h-\sigma^{2}_{|I}\right\|^{2}_{n}}+C\mathbb{E}\left(\underset{h\in\mathcal{S}_{m}, \|h\|_{n}=1}{\sup}{\nu^{2}(h)}\right)+C\Delta^2
\end{equation}
where $C>0$ is a constant depending on $\sigma_1$, and $\nu = \nu_1+\nu_2+\nu_3$ with
\begin{equation*}
    \nu_i(h) = \frac{1}{Nn}\sum_{j=1}^{N}{\sum_{k=0}^{n-1}{h(X^{j}_{k\Delta})\zeta^{j,i}_{k\Delta}}}, ~~ i=1,2,3
\end{equation*}
and the $\zeta^{j,i}_{k\Delta}$'s are the error terms depending on each path $X^{j},~ j=1,\ldots,N$.

\subsection*{Upper bound of $\mathbb{E}\left(\underset{h\in\mathcal{S}_{m}, \ \|h\|_{n}=1}{\sup}{\nu^{2}(h)}\right)$}

For all $h=\sum_{\ell=0}^{m-1}{a_{\ell}\phi_{\ell}}\in\mathcal{S}_{m}$ such that $\|h\|_n=1$, we have $\|h\|^{2}\leq\frac{1}{\tau_0}$ and the coordinate vector $a=\left(a_{0},\cdots,a_{m-1}\right)$ satisfies:
\begin{itemize}
    \item $\|a\|^{2}_{2}\leq CK \leq Cm$ for the spline basis (see \cite{denis2020ridge}, Lemma 2.6)
    \item $\|a\|^{2}_{2} \leq 1/\tau_0$ for an orthonormal basis since $\|h\|^2 = \|a\|^{2}_{2}$.
\end{itemize}
Furthermore, using the \textcolor{black}{Cauchy$-$Schwarz} inequality, we have:
\begin{align*}
    \nu^{2}(h)=\left(\sum_{\ell=0}^{m-1}{a_{\ell}\nu\left(\phi_{\ell}\right)}\right)^2\leq\|a\|^{2}_{2}\sum_{\ell=0}^{m-1}{\nu^{2}\left(\phi_{\ell}\right)}.
\end{align*}
Thus, for all $\ell\in[\![0,m-1]\!], \ \ \nu=\nu_1+\nu_2+\nu_3$ and for all $i\in\{1,2,3\}$
\begin{align*}
    \mathbb{E}\left[\nu^{2}_{i}\left(\phi_{\ell}\right)\right]=&~\frac{1}{Nn^2}\mathbb{E}\left[\left(\sum_{k=0}^{n-1}{\phi_{\ell}\left(X^{1}_{k\Delta}\right)\zeta^{1,i}_{k\Delta}}\right)^2\right].
\end{align*}
We finally deduce from ~\eqref{eq:upper bound of nu - case i=1},~\eqref{eq:upper bound of nu - case i=2}~and~\eqref{eq:upper bound of nu - case i=3} that there exists a constant $C>0$ depending on $\sigma_{1}$ such that:
\begin{equation}
    \label{eq:upper bound of nu- Nn - all cases}
    \mathbb{E}\left(\underset{h\in\mathcal{S}_{m}, \ \|h\|_{n}=1}{\sup}{\nu^{2}(h)}\right)\leq C\frac{m}{Nn}.
\end{equation}
We deduce from ~\eqref{eq:upper-bound-Nn-1}~and~\eqref{eq:upper bound of nu- Nn - all cases} that there exists a constant $C>0$ such that,
\begin{equation}
    \label{eq:UpperBound-MSE-Omega}
    \mathbb{E}\left[\left\|\widehat{\sigma}^{2}_{m}-\sigma^{2}_{|I}\right\|^{2}_{n,N}\one_{\Omega_{n,N,m}}\right]\leq 3\underset{h\in\mathcal{S}_{m,L}}{\inf}{\left\|\sigma^{2}_{|I}-h\right\|^{2}_{n}}+C\left(\frac{m}{Nn}+\Delta^{2}\right).
\end{equation}
Since we have $\left\|\widehat{\sigma}^{2}_{m}\right\|_{\infty}\leq\sqrt{mL}$, then for $m$ and $L$ large enough, $\left\|\widehat{\sigma}^{2}_{m}-\sigma^{2}_{|I}\right\|^{2}_{\infty}\leq 2mL$.  There exists a constant $C>0$ such that for all $m\in\mathcal{M}$ and for $m$ and $L$ large enough,
\begin{align*}
	\mathbb{E}\left[\left\|\widehat{\sigma}^{2}_{m}-\sigma^{2}_{|I}\right\|^{2}_{n,N}\right] & =\mathbb{E}\left[\left\|\widehat{\sigma}^{2}_{m}-\sigma^{2}_{|I}\right\|^{2}_{n,N}\one_{\Omega_{n,N,m}}\right]+\mathbb{E}\left[\left\|\widehat{\sigma}^{2}_{m}-\sigma^{2}_{|I}\right\|^{2}_{n,N}\one_{\Omega^{c}_{n,N,m}}\right]\\
	&\leq\mathbb{E}\left[\left\|\widehat{\sigma}^{2}_{m}-\sigma^{2}_{|I}\right\|^{2}_{n,N}\one_{\Omega_{n,N,m}}\right]+2mL\P\left(\Omega^{c}_{n,N,m}\right).  
\end{align*}
Then, from Equation~\eqref{eq:UpperBound-MSE-Omega}, Lemma~\ref{lm:Proba-OmegaComp} and for $m \in \mathcal{M} = \left\{1,\ldots,\sqrt{\min(n,N)}/\sqrt{\log(Nn)}\right\}$, we have:
\begin{align*}
       \mathbb{E}\left[\left\|\widehat{\sigma}^{2}_{m}-\sigma^{2}_{|I}\right\|^{2}_{n,N}\right]\leq &~ 3\underset{h\in\mathcal{S}_{m,L}}{\inf}{\left\|h - \sigma^{2}_{|I}\right\|^{2}_{n}}+C\left(\frac{m}{Nn}+mL\exp\left(-C\sqrt{\min(n,N)}\right)+\Delta^{2}\right)
\end{align*}
where $C>0$ is a constant. Recall that the empirical norms $\|.\|_{n,N}$ and $\|.\|_n$ are equivalent on $\Omega_{n,N,m}$, that is  for all $h\in\mathcal{S}_{m}, \ \ \|h\|^{2}_{n}\leq 2\|h\|^{2}_{n,N}$. Thus, we have
\begin{align*}
    \mathbb{E}\left[\left\|\widehat{\sigma}^{2}_{m}-\sigma^{2}_{|I}\right\|^{2}_{n}\right] = &~ \mathbb{E}\left[\left\|\widehat{\sigma}^{2}_{m}-\sigma^{2}_{|I}\right\|^{2}_{n}\one_{\Omega_{n,N,m}}\right] + \mathbb{E}\left[\left\|\widehat{\sigma}^{2}_{m}-\sigma^{2}_{|I}\right\|^{2}_{n}\one_{\Omega^{c}_{n,N,m}}\right]\\
    \leq &~ \mathbb{E}\left[\left\|\widehat{\sigma}^{2}_{m}-\sigma^{2}_{|I}\right\|^{2}_{n}\one_{\Omega_{n,N,m}}\right] + 2mL\P\left(\Omega^{c}_{n,N,m}\right).
\end{align*}
For all $h \in \mathcal{S}_{m,L} \subset \mathcal{S}_{m}$, we have:
\begin{align*}
    \mathbb{E}\left[\left\|\widehat{\sigma}^{2}_{m}-\sigma^{2}_{|I}\right\|^{2}_{n}\one_{\Omega_{n,N,m}}\right] \leq &~ 2\mathbb{E}\left[\left\|\widehat{\sigma}^{2}_{m}-h\right\|^{2}_{n}\one_{\Omega_{n,N,m}}\right] + 2\left\|h-\sigma^{2}_{|I}\right\|^{2}_{n} \\
    \leq &~ 4\mathbb{E}\left[\left\|\widehat{\sigma}^{2}_{m}-h\right\|^{2}_{n,N}\one_{\Omega_{n,N,m}}\right] + 2\left\|h-\sigma^{2}_{|I}\right\|^{2}_{n}\\
    \leq &~ 8\mathbb{E}\left[\left\|\widehat{\sigma}^{2}_{m}-\sigma^{2}_{|I}\right\|^{2}_{n,N}\right] + 10\left\|h-\sigma^{2}_{|I}\right\|^{2}_{n}.
\end{align*}
We finally conclude that
\begin{align*}
    \mathbb{E}\left[\left\|\widehat{\sigma}^{2}_{m}-\sigma^{2}_{|I}\right\|^{2}_{n}\right] \leq & 34\underset{h\in\mathcal{S}_{m,L}}{\inf}{\left\|h-\sigma^{2}_{|I}\right\|^{2}_{n}}+C\left(\frac{m}{Nn}+mL\exp\left(-C\sqrt{\min(n,N)}\right)+\Delta^{2}\right).
\end{align*}
\end{proof}

\subsubsection{Proof of Theorem~\ref{thm:RiskBound-AnySupport}~}

\begin{proof}
\textcolor{black}{For each dimension $m \in \mathcal{M}$,} we have:
\begin{equation*}
    \mathbb{E}\left[\left\|\widehat{\sigma}^{2}_{m,L}-\sigma^{2}\right\|^{2}_{n}\right] = \mathbb{E}\left[\left\|(\widehat{\sigma}^{2}_{m,L}-\sigma^{2})\one_{[-\log(N),\log(N)]}\right\|^{2}_{n}\right] + \mathbb{E}\left[\left\|(\widehat{\sigma}^{2}_{m,L}-\sigma^{2})\one_{[-\log(N),\log(N)]^{c}}\right\|^{2}_{n}\right]
\end{equation*}
and from Proposition~\ref{prop:densityTransition-bis}, Lemma~\ref{lem:controleSortiCompact-bis} and for $N$ large enough, there exists constants $c,C>0$ such that
\begin{align*}
    \mathbb{E}\left[\left\|(\widehat{\sigma}^{2}_{m,L}-\sigma^{2})\one_{[-\log(N),\log(N)]^{c}}\right\|^{2}_{n}\right] \leq & \frac{2L}{n}\sum_{k=0}^{n-1}{\P\left(|X_{k\Delta}| > \log(N)\right)} \leq 2L\underset{t\in[0,1]}{\sup}{\P\left(|X_t|\geq \log(N)\right)} \\
    \leq & \frac{C}{\log(N)}\exp\left(-c\log^{2}(N)\right).
\end{align*}
We deduce that 
\begin{equation}
\label{eq:Semi-Result}
    \mathbb{E}\left[\left\|\widehat{\sigma}^{2}_{m,L}-\sigma^{2}\right\|^{2}_{n}\right] = \mathbb{E}\left[\left\|(\widehat{\sigma}^{2}_{m,L}-\sigma^{2})\one_{[-\log(N),\log(N)]}\right\|^{2}_{n}\right] + \frac{C}{\log(N)}\exp\left(-c\log^{2}(N)\right).
\end{equation}
It remains to upper-bound the first term on the right hand side of Equation~\eqref{eq:Semi-Result}. 

\paragraph{Upper bound of $\mathbb{E}\left[\left\|\widehat{\sigma}^{2}_{m,L}-\sigma^{2}\right\|^{2}_{n}\one_{[-\log(N),\log(N)]}\right]$.}

For all $h\in\mathcal{S}_{m,L}$, we obtain from Equation~\eqref{eq:non adaptive estimator},
\begin{equation}
\label{eq:property-least squares contrast-sigma}
    \gamma_{n,N}(\widehat{\sigma}^{2}_{m,L})-\gamma_{n,N}(\sigma^{2})\leq\gamma_{n,N}(h)-\gamma_{n,N}(\sigma^{2}).
\end{equation}
For all $h\in\mathcal{S}_{m,L}$,
\begin{equation*}
    \gamma_{n,N}(h)-\gamma_{n,N}(\sigma^{2})=\left\|h-\sigma^{2}\right\|^{2}_{n,N}+2\nu_1(\sigma^{2}-h)+2\nu_2(\sigma^{2}-h)+2\nu_3(\sigma^{2}-h)+2\mu(\sigma^{2}-h)
\end{equation*}
where
\begin{equation}
\label{eq:functions nu1 nu2 nu2 and mu}
    \nu_i(h)=\frac{1}{nN}\sum_{j=1}^{N}{\sum_{k=0}^{n-1}{h(X^{j}_{k\Delta})\zeta^{j,i}_{k\Delta}}}, \ \ i\in\{1,2,3\}, \ \ \ \mu(h)=\frac{1}{nN}\sum_{j=1}^{N}{\sum_{k=0}^{n-1}{h(X^{j}_{k\Delta})R^{j}_{k\Delta}}},
\end{equation}
we deduce from Equation~\eqref{eq:property-least squares contrast-sigma} that for all $h\in\mathcal{S}_{m,L}$,
\begin{equation}
    \label{eq:first risk bound}
    \mathbb{E}\left[\left\|\widehat{\sigma}^{2}_{m,L}-\sigma^{2}\right\|^{2}_{n,N}\one_{[-\log(N),\log(N)]}\right] \leq \underset{h\in\mathcal{S}_{m,L}}{\inf}{\|h-\sigma^{2}\|^{2}_{n}}+2\sum_{i=1}^{3}{\mathbb{E}\left[\nu_i(\widehat{\sigma}^{2}_{m,L}-h)\right]}+2\mathbb{E}\left[\mu(\widehat{\sigma}^{2}_{m,L}-h)\right].
\end{equation}
For all $i\in\{1,2,3\}$ and for all $h\in\mathcal{S}_{m,L}$, one has
\begin{equation}
\label{eq:upper-bound Enu_i}
    \mathbb{E}\left[\nu_i\left(\widehat{\sigma}^{2}_{m,L}-h\right)\right]\leq\sqrt{2mL}\sqrt{\sum_{\ell=0}^{m-1}{\mathbb{E}\left[\nu^{2}_{i}(\phi_{\ell})\right]}}.
\end{equation}
\begin{enumerate}
    \item Upper bound of $\sum_{\ell=0}^{m-1}{\mathbb{E}\left[\nu^{2}_{1}(\phi_{\ell})\right]}$
    
    According to Equation~\eqref{eq:functions nu1 nu2 nu2 and mu}, we have
    \begin{equation*}
        \forall \ell\in[\![0,m-1]\!], \ \nu_1(\phi_{\ell})=\frac{1}{Nn}\sum_{j=1}^{N}{\sum_{k=0}^{n-1}{\phi_{\ell}(X^{j}_{k\Delta})\zeta^{j,1}_{k\Delta}}}
    \end{equation*}
    where $\zeta^{j,1}_{k\Delta}=\frac{1}{\Delta}\left[\left(\int_{k\Delta}^{(k+1)\Delta}{\sigma(X^{j}_{s})dW^{j}_s}\right)^2-\int_{k\Delta}^{(k+1)\Delta}{\sigma^{2}(X^{j}_{s})ds}\right]$ is a martingale satisfying
    \begin{equation*}
        \mathbb{E}\left[\zeta^{1,1}_{k\Delta}|\mathcal{F}_{k\Delta}\right]=0 \ \ \mathrm{and} \ \ \mathbb{E}\left[\left(\zeta^{1,1}_{k\Delta}\right)^2|\mathcal{F}_{k\Delta}\right]\leq\frac{1}{\Delta^2}\mathbb{E}\left[\left(\int_{k\Delta}^{(k+1)\Delta}{\sigma^{2}(X^{1}_{s})ds}\right)^2\right]\leq C\sigma^{4}_{1}
    \end{equation*}
    with $C>0$ a constant, $W=W^{1}$ and $(\mathcal{F}_t)_{t\geq 0}$ the natural filtration of the martingale $(M_t)_{t\in[0,1]}$ given for all $t\in[0,1]$ by $M_{t}=\int_{0}^{t}{\sigma(X^{1}_s)dW_s}$. We derive that
    \begin{align*}
        \sum_{\ell=0}^{m-1}{\mathbb{E}\left[\nu^{2}_{1}(\phi_{\ell})\right]}=&\frac{1}{Nn^2}\sum_{\ell=0}^{m-1}{\mathbb{E}\left[\left(\sum_{k=0}^{n-1}{\phi_{\ell}(X^{j}_{k\Delta})\zeta^{1,1}_{k\Delta}}\right)^2\right]}=\frac{1}{Nn^2}\mathbb{E}\left[\sum_{k=0}^{n-1}{\sum_{\ell=0}^{m-1}{\phi^{2}_{\ell}(X^{1}_{k\Delta})\left(\zeta^{1,1}_{k\Delta}\right)^2}}\right].
    \end{align*}
    For each $k\in[\![0,n-1]\!]$, we have 
    $$\begin{cases} 
    \sum_{\ell=0}^{m-1}{\phi^{2}_{\ell}(X^{1}_{k\Delta})} = \sum_{\ell=-M}^{K-1}{B^{2}_{\ell}(X^{1}_{k\Delta})} =1 ~~ \mathrm{for ~ the ~ spline ~ basis} \\ \\
    \sum_{\ell=0}^{m-1}{\phi^{2}_{\ell}(X^{1}_{k\Delta})} \leq Cm ~~ \mathrm{For ~ an ~ orthonormal ~ basis ~ with} ~ C=\underset{0 \leq\ell\leq m-1}{\max}{\|\phi_{\ell}\|^{2}_{\infty}}.
    \end{cases}
    $$ 
    Finally, there exists a constant $C>0$ such that
    \begin{align*}
        \sum_{\ell=0}^{m-1}{\mathbb{E}\left[\nu^{2}_{1}(\phi_{\ell})\right]} \leq \begin{cases}
        \frac{C}{Nn} ~~ \mathrm{for ~ the ~ spline ~ basis} \\ \\
        C\frac{m}{Nn} ~~ \mathrm{for ~ an ~ orthonormal ~ basis}.
        \end{cases}
    \end{align*}
    
    \item Upper bound of $\sum_{\ell=0}^{m-1}{\mathbb{E}\left[\nu^{2}_{2}(\phi_{\ell})\right]}$
    
    For all $k\in[\![0,n-1]\!]$ and for all $s\in[0,1]$, set $\eta(s)=k\Delta$ if $s\in[k\Delta,(k+1)\Delta)$. We have:
    \begin{align*}
        \sum_{\ell=0}^{m-1}{\mathbb{E}\left[\nu^{2}_{2}(\phi_{\ell})\right]}&=\frac{4}{N}\sum_{\ell=0}^{m-1}{\mathbb{E}\left[\left(\sum_{k=0}^{n-1}{\int_{k\Delta}^{(k+1)\Delta}{\phi_{\ell}(X^{1}_{k\Delta})((k+1)\Delta-s)\sigma^{\prime}(X^{1}_{s})\sigma^{2}(X^{1}_{s})dW_s}}\right)^2\right]}\\
        &=\frac{4}{N}\sum_{\ell=0}^{m-1}{\mathbb{E}\left[\left(\int_{0}^{1}{\phi_{\ell}(X^{1}_{\eta(s)})(\eta(s)+\Delta-s)\sigma^{\prime}(X^{1}_{s})\sigma^{2}(X^{1}_{s})dW_s}\right)^2\right]}.
    \end{align*}
    We conclude that
    \begin{align*}
        \sum_{\ell=0}^{m-1}{\mathbb{E}\left[\nu^{2}_{2}(\phi_{\ell})\right]} \leq \begin{cases}
        \frac{C}{Nn^2} ~~ \mathrm{for ~ the ~ spline ~ basis} \\ \\
        C\frac{m}{Nn^2} ~~ \mathrm{for ~ an ~ orthonormal ~ basis}.
        \end{cases}
    \end{align*}
     where the constant $C>0$ depends on the diffusion coefficient and the upper bound of the basis functions.
     
    \item Upper bound of $\sum_{\ell=0}^{m-1}{\mathbb{E}\left[\nu^{3}_{2}(\phi_{\ell})\right]}$
    
    We have:
    \begin{align*}
        \sum_{\ell=0}^{m-1}{\mathbb{E}\left[\nu^{2}_{3}(\phi_{\ell})\right]}&=\frac{4}{Nn^2}\sum_{\ell=0}^{m-1}{\mathbb{E}\left[\left(\sum_{k=0}^{n-1}{\int_{k\Delta}^{(k+1)\Delta}{\phi_{\ell}(X^{1}_{k\Delta})b(X^{1}_{k\Delta})\sigma(X^{1}_{s})dW_s}}\right)^2\right]}\\
        &=\frac{4}{Nn^2}\sum_{\ell=0}^{m-1}{\mathbb{E}\left[\left(\int_{0}^{1}{\phi_{\ell}(X^{1}_{\eta(s)})b(X^{1}_{\eta(s)})\sigma(X^{1}_{s})dW_s}\right)^2\right]}\\
        &\leq\frac{4}{Nn^2}\mathbb{E}\left[\int_{0}^{1}{\sum_{\ell=0}^{m-1}{\phi^{2}_{\ell}(X^{1}_{\eta(s)})b(X^{1}_{\eta(s)})\sigma^{2}(X^{1}_{s})ds}}\right].
    \end{align*}
    Since for all $x\in\mathbb{R}, \ b(x)\leq C_0(1+x^2)$ and $\underset{t\in[0,1]}{\sup}{\mathbb{E}\left(|X_t|^2\right)}<\infty$, there exists a constant $C>0$ depending on the diffusion coefficient such that
    \begin{align*}
        \sum_{\ell=0}^{m-1}{\mathbb{E}\left[\nu^{2}_{3}(\phi_{\ell})\right]} \leq \begin{cases}
        \frac{C}{Nn^2} ~~ \mathrm{for ~ the ~ spline ~ basis} \\ \\
        C\frac{m}{Nn^2} ~~ \mathrm{for ~ an ~ orthonormal ~ basis}.
        \end{cases}
    \end{align*}
\end{enumerate}

We finally deduce that from Equations~\eqref{eq:first risk bound}~and~\eqref{eq:upper-bound Enu_i}~ that for all $h\in\mathcal{S}_{m,L}$,
\begin{equation}
\label{eq:First-UpperBounds}
\begin{cases}
    \mathbb{E}\left[\left\|\widehat{\sigma}^{2}_{m,L}-\sigma^{2}\right\|^{2}_{n,N}\one_{[-\log(N),\log(N)]}\right] \leq \underset{h\in\mathcal{S}_{m,L}}{\inf}{\|h-\sigma^{2}\|^{2}_{n}}+C\sqrt{\frac{mL}{Nn}}+2\mathbb{E}\left[\mu(\widehat{\sigma}^{2}_{m,L}-h)\right] ~~ (1) \\ \\
    \mathbb{E}\left[\left\|\widehat{\sigma}^{2}_{m,L}-\sigma^{2}\right\|^{2}_{n,N}\one_{[-\log(N),\log(N)]}\right] \leq \underset{h\in\mathcal{S}_{m,L}}{\inf}{\|h-\sigma^{2}\|^{2}_{n}}+C\sqrt{\frac{m^{2}L}{Nn}}+2\mathbb{E}\left[\mu(\widehat{\sigma}^{2}_{m,L}-h)\right] ~~ (2)
\end{cases}
\end{equation}
where $C>0$ is a constant, the result $(1)$ corresponds to the spline basis, and the result (2) corresponds to any orthonormal basis. It remains to obtain an upper bound of the term $\mu(\widehat{\sigma}^{2}_{m,L}-h)$. For all $a>0$ and for all $h\in\mathcal{S}_{m,L}$,
\begin{align*}
    2\mu\left(\widehat{\sigma}^{2}_{m,L}-h\right) \leq & \frac{2}{a}\left\|\widehat{\sigma}^{2}_{m,L}-\sigma^{2}\right\|^{2}_{n,N}+\frac{2}{a}\left\|h-\sigma^{2}\right\|^{2}_{n,N}+\frac{a}{Nn}\sum_{j=1}^{N}{\sum_{k=0}^{n-1}{\left(R^{j}_{k\Delta}\right)^2}}\\
    2\mathbb{E}\left[\mu\left(\widehat{\sigma}^{2}_{m,L}-h\right)\right] \leq & \frac{2}{a}\mathbb{E}\left\|\widehat{\sigma}^{2}_{m,L}-\sigma^{2}\right\|^{2}_{n,N}+\frac{2}{a}\underset{h\in\mathcal{S}_{m,L}}{\inf}{\|h-\sigma^{2}\|^{2}_{n}} +\frac{a}{Nn}\sum_{j=1}^{N}{\sum_{k=0}^{n-1}{\mathbb{E}\left[\left(R^{j}_{k\Delta}\right)^2\right]}}.
\end{align*}
Using Equations~\eqref{eq:UpperBound-TimeStep},~\eqref{eq:First-UpperBounds}~and setting $a=4$, we deduce that there exists \textcolor{red}{a} constant $C>0$ depending on $\sigma_{1}$ such that,
\begin{equation}
\label{eq:pseudo-norm-loss-error}
   \begin{cases}
    \mathbb{E}\left[\left\|\widehat{\sigma}^{2}_{m,L}-\sigma^{2}\right\|^{2}_{n,N}\one_{[-\log(N),\log(N)]}\right]\leq \underset{h\in\mathcal{S}_{m,L}}{\inf}{\|h-\sigma^{2}\|^{2}_{n}}+C\left(\sqrt{\frac{mL}{Nn}}+\Delta^2\right) ~~ [\mathbf{B}] \\ \\
    \mathbb{E}\left[\left\|\widehat{\sigma}^{2}_{m,L}-\sigma^{2}\right\|^{2}_{n,N}\one_{[-\log(N),\log(N)]}\right] \leq \underset{h\in\mathcal{S}_{m,L}}{\inf}{\|h-\sigma^{2}\|^{2}_{n}}+C\left(\sqrt{\frac{m^{2}L}{Nn}}+\Delta^2\right) ~~ [\mathbf{H}].
   \end{cases}
\end{equation}
The final result is obtained from Equations \eqref{eq:Semi-Result} and \eqref{eq:pseudo-norm-loss-error}.
\end{proof}

\subsection{Proof of Section~\ref{sec:AdaptiveEstimation-N.paths}}

\subsubsection{Proof of Theorem \ref{thm:adaptive estimator - compact interval}}

Set for all $K, K^{\prime} \in \mathcal{K} = \left\{2^q, ~ q=0,\ldots, q_{\max}, ~ 2^{q_{\max}}\leq\sqrt{N}/\log(N)\right\} \subset \mathcal{M}$,
\begin{equation}
\label{eq:subspace TdN}
    \mathcal{T}_{K,K^{\prime}} = \left\{g\in\mathcal{S}_{K+M}+\mathcal{S}_{K^{\prime}+M}, \ \|g\|_{n}=1, ~ \|g\|_{\infty} \leq \sqrt{L}\right\}.
\end{equation}
Recall that for all $j \in [\![1,N]\!]$ and for all $k \in [\![0,n]\!]$,
$$ \zeta^{j,1}_{k\Delta} = \frac{1}{\Delta}\left[\left(\int_{k\Delta}^{(k+1)\Delta}{\sigma(X^{j}_{s})dW^{j}_{s}}\right)^2-\int_{k\Delta}^{(k+1)\Delta}{\sigma^{2}(X^{j}_{s})ds}\right].$$
The proof of Theorem~\ref{thm:adaptive estimator - compact interval} relies on the following lemma whose proof is in Appendix.
\begin{lemme}
\label{lm:LemmaAdaptation}
Under Assumption~\ref{ass:Assumption 1}, for all $\varepsilon, \ v>0$ and $g\in\mathcal{T}_{K,K^{\prime}}$, there exists a real constant $C>0$ such that,
\begin{align*}
\mathbb{P}\left(\frac{1}{Nn}\sum_{j=1}^{N}{\sum_{k=0}^{n-1}{g(X^{j}_{k\Delta})\zeta^{j,1}_{k\Delta}}}\geq\varepsilon, \|g\|^{2}_{n,N}\leq v^2\right)\leq\exp\left(-C\frac{Nn\varepsilon^{2}}{\sigma^{2}_{1}\left(\varepsilon\|g\|_{\infty}+4\sigma^{2}_{1}v^2\right)}\right)
\end{align*}
and for all $x>0$ such that $x\leq\varepsilon^{2}/\sigma^{2}_{1}\left(\varepsilon\|g\|_{\infty}+4\sigma^{2}_{1}v^{2}\right)$,
\begin{align*}
\mathbb{P}\left(\frac{1}{Nn}\sum_{j=1}^{N}{\sum_{k=0}^{n-1}{g(X^{j}_{k\Delta})\zeta^{j,1}_{k\Delta}}}\geq 2\sigma^{2}_{1}v\sqrt{x}+\sigma^{2}_{1}\|g\|_{\infty}x, \|g\|^{2}_{n,N}\leq v^2\right)\leq\exp\left(-CNnx\right).
\end{align*}
\end{lemme}

\begin{proof}[{\bf Proof of Theorem \ref{thm:adaptive estimator - compact interval}}]

From Equation~\eqref{eq:selection of the dimension}, we have
\begin{align*}
\widehat{K}:=\underset{K\in\mathcal{K}}{\arg\min}{\left\{{\gamma}_{n,N}(\widehat{\sigma}^{2}_{K})+\textcolor{black}{\mathrm{pen}_1(K)}\right\}}.
\end{align*}
For all $K\in\mathcal{K}$ and $h\in\mathcal{S}_{K+M,L}$, 
$$\gamma_{n,N}(\widehat{\sigma}^{2}_{\widehat{K}})+\textcolor{black}{\mathrm{pen}_1(\widehat{K})}\leq\gamma_{n,N}(h)+\textcolor{black}{\mathrm{pen}_1(K)},$$
then, for all $K\in\mathcal{K}$ and for all $h\in\mathcal{S}_{K+M,L}$,
\begin{align*}
\gamma_{n,N}(\widehat{\sigma}^{2}_{\widehat{K}})-\gamma_{n,N}(\sigma^{2}_{|I})\leq &~ \gamma_{n,N}(h)-\gamma_{n,N}(\sigma^{2}_{|I})+\textcolor{black}{\mathrm{pen}_1(K)-\mathrm{pen}_1(\widehat{K})}\\
\left\|\widehat{\sigma}^{2}_{\widehat{K}}-\sigma^{2}_{|I}\right\|^{2}_{n,N} \leq &~\left\|h-\sigma^{2}_{|I}\right\|^{2}_{n,N}+2\nu\left(\widehat{\sigma}^{2}_{\widehat{K}}-h\right)+2\mu\left(\widehat{\sigma}^{2}_{\widehat{K}} - h\right)+\textcolor{black}{\mathrm{pen}_1(K)-\mathrm{pen}_1(\widehat{K})}\\
\leq &~\left\|h-\sigma^{2}_{|I}\right\|^{2}_{n,N}+\frac{1}{d}\left\|\widehat{\sigma}^{2}_{\widehat{K}}-t\right\|^{2}_{n}+d\underset{g\in\mathcal{T}_{K,\widehat{K}}}{\sup}{\nu^{2}(g)}+\frac{1}{d}\left\|\widehat{\sigma}^{2}_{\widehat{K}}-h\right\|^{2}_{n,N}\\
&+\frac{d}{Nn}\sum_{j=1}^{N}{\sum_{k=0}^{n-1}{(R^{j}_{k\Delta})^2}}+\textcolor{black}{\mathrm{pen}_1(K)-\mathrm{pen}_1(\widehat{K})}
\end{align*}
where $d>1$ and the space $\mathcal{T}_{K,\widehat{K}}$ is given in Equation~\eqref{eq:subspace TdN}. On the set $\Omega_{n,N,K_{\max}}$ (given in Equation \eqref{equivalence-set}): $\forall h\in\mathcal{S}_{K+M}, \ \ \frac{1}{2}\|h\|^{2}_{n}\leq\|h\|^{2}_{n,N}\leq\frac{3}{2}\|h\|^{2}_{n}$. Then on $\Omega_{n,N,K_{\max}}$, for all $d>1$ and for all $h\in\mathcal{S}_{K+M}$ with $K\in\mathcal{K}$,
\begin{align*}
\left(1-\frac{10}{d}\right)\left\|\widehat{\sigma}^{2}_{\widehat{K}}-\sigma^{2}_{|I}\right\|^{2}_{n,N}\leq &~\left(1+\frac{10}{d}\right)\left\|h-\sigma^{2}_{|I}\right\|^{2}_{n,N}+d\underset{h\in\mathcal{T}_{K,\widehat{K}}}{\sup}{\nu^{2}(h)}+\frac{d}{Nn}\sum_{j=1}^{N}{\sum_{k=0}^{n-1}{(R^{j}_{k\Delta})^2}}\\
&+\textcolor{black}{\mathrm{pen}_1(K)-\mathrm{pen}_1(\widehat{K})}.
\end{align*}
We set $d=20$. Then, on $\Omega_{n,N,\max}$ and for all $h\in\mathcal{S}_{K+M,L}$,
\begin{equation}
\label{eq:first star-proof of theorem 3.7}
\left\|\widehat{\sigma}^{2}_{\widehat{K}}-\sigma^{2}_{|I}\right\|^{2}_{n,N}\leq 3\left\|h - \sigma^{2}_{|I}\right\|^{2}_{n,N}+20\underset{h\in\mathcal{T}_{K,\widehat{K}}}{\sup}{\nu^{2}(h)}+\frac{20}{Nn}\sum_{j=1}^{N}{\sum_{k=0}^{n-1}{(R^{j}_{k\Delta})^2}}+2\left(\textcolor{black}{\mathrm{pen}_1(K)-\mathrm{pen}_1(\widehat{K})}\right).
\end{equation}
Let \textcolor{black}{$q_1 : \mathcal{K}^{2}\longrightarrow\mathbb{R}_{+}$} such that  \textcolor{black}{$160 q_1(K,K^{\prime})\leq 18 \mathrm{pen}_1(K)+16 \mathrm{pen}_1(K^{\prime})$}. Thus, on the set $\Omega_{n,N,K_{\max}}$, there exists a constant $C>0$ such that for all $h\in\mathcal{S}_{K+M}$
\begin{align*}
\mathbb{E}\left[\left\|\widehat{\sigma}^{2}_{\widehat{K}}-\sigma^{2}_{|I}\right\|^{2}_{n,N}\one_{\Omega_{n,N,K_{\max}}}\right]\leq &~ 34\left(\underset{h\in\mathcal{S}_{K+M,L}}{\inf}{\left\|h-\sigma^{2}_{|I}\right\|^{2}_{n}}+\textcolor{black}{\mathrm{pen}_1(K)}\right) \\
&+160\E\left(\underset{h\in\mathcal{T}_{K,\widehat{K}}}{\sup}{\nu^{2}_{1}(h)}-\textcolor{black}{q_1(K,\widehat{K})}\right)+C\Delta^{2}
\end{align*}
where
\begin{equation}
\label{eq:nu-ErrorTerm}
    \nu_{1}(h):=\frac{1}{Nn}\sum_{j=1}^{N}{\sum_{k=0}^{n-1}{h(X^{j}_{k\Delta})\zeta^{j,1}_{k\Delta}}}
\end{equation}
with $\zeta^{j,1}_{k\Delta}$ the error term. We set for all $K,K^{\prime}\in\mathcal{K}$,
\begin{equation*}
    G_{K}(K^{\prime}):=\underset{h\in\mathcal{T}_{K,K^{\prime}}}{\sup}{\nu^{2}_{1}(h)}
\end{equation*}
and for $N$ and $n$ large enough, $\left\|\widehat{\sigma}^{2}_{\widehat{K}}-\sigma^{2}_{|I}\right\|^{2}_{n,N}\leq 4(K+M)L$. We deduce that,
\begin{align*}
\mathbb{E}\left[\left\|\widehat{\sigma}^{2}_{\widehat{K}}-\sigma^{2}_{|I}\right\|^{2}_{n,N}\right] \leq &~ \mathbb{E}\left[\left\|\widehat{\sigma}^{2}_{\widehat{K}}-\sigma^{2}_{|I}\right\|^{2}_{n,N}\one_{\Omega_{n,N,K_{\max}}}\right]+\mathbb{E}\left[\left\|\widehat{\sigma}^{2}_{\widehat{K}}-\sigma^{2}_{|I}\right\|^{2}_{n,N}\one_{\Omega^{c}_{n,N,K_{\max}}}\right]\\
\leq &~ 34\underset{K\in\mathcal{K}}{\inf}{\left(\underset{h\in\mathcal{S}_{K+M}}{\inf}{\left\|h-\sigma^{2}_{|I}\right\|^{2}_{n}}+\textcolor{black}{\mathrm{pen}_1(K)}\right)} \\
&+C\Delta^2+4(K+M)L\mathbb{P}(\Omega^{c}_{n,N,K_{\max}})\\
&+160\mathbb{E}\left[\left(G_{K}(\widehat{K})-\textcolor{black}{q_1(K,\widehat{K})}\right)_{+}\one_{\Omega_{n,N,K_{\max}}}\right].
\end{align*}
\textcolor{black}{The unit ball $\overline{B}_{\|.\|_n}(0,1)$ of the approximation subspace $\mathcal{S}_{K+M}$ with respect to norm $\|.\|_n$ defined as follows:
\begin{align*}
    \overline{B}_{\|.\|_n}(0,1)=\left\{h\in\mathcal{S}_{K+M} : \|h\|_n\leq 1\right\}=\left\{h\in\mathcal{S}_{K+M} : \|h\|\leq\frac{1}{\tau_0}\right\}=\overline{B}_{2}(0,1/\tau_0),
\end{align*}
admits an $\varepsilon-$net $E_{\varepsilon}$ such that for each $\varepsilon\in(0,1], \ \left|E_{\varepsilon}\right|\leq\left(\frac{3}{\varepsilon\tau_0}\right)^{K+M}$ (see \cite{lorentz1996constructive}, Proposition 1.3, page 487).
}

Recall that $\mathcal{T}_{K,K^{\prime}}=\left\{g\in\mathcal{S}_{K+M}+\mathcal{S}_{K^{\prime}+M}, \|g\|_{n}=1, \|g\|_{\infty} \leq \sqrt{L}\right\}$ and consider the sequence \\ 
$\left(E_{\varepsilon_{k}}\right)_{k\geq 1}$ of $\varepsilon-$net with $\varepsilon_{k}=\varepsilon_{0} 2^{-k}$ and $\varepsilon_{0}\in(0,1]$. Moreover, set $N_k = \log(|E_{\varepsilon_k}|)$ for each $k\geq 0$. Then for each $g\in\mathcal{S}_{K+M}+\mathcal{S}_{K^{\prime}+M}$ such that $\|g\|_{\infty} \leq \sqrt{L}$, there exists a sequence $\left(g_k\right)_{k\geq 0}$ with $g_k\in E_{\varepsilon_{k}}$ such that $g=g_0+\sum_{k=1}^{\infty}{g_k-g_{k-1}}$. Set $\widetilde{\mathbb{P}}:=\mathbb{P}\left(.\cap\Omega_{n,N,K_{\max}}\right)$ and
\begin{equation*} 
\tau:=\sigma_1^2\sqrt{6x^{n,N}_{0}}+\sigma^{2}_{1}\sqrt{L}x^{n,N}_{0}+\sum_{k\geq 1}{\varepsilon_{k-1}\left\{\sigma_1^2\sqrt{6x^{n,N}_{k}}+2\sigma^{2}_{1}\sqrt{L}x^{n,N}_{k}\right\}}=y^{n,N}_0+\sum_{k\geq 0}{y^{n,N}_k}.
\end{equation*}
For all $h\in\mathcal{T}_{K,K^{\prime}}$ and on the event $\Omega_{n,N,K_{\max}}$, one has $\|h\|^{2}_{n,N}\leq\frac{3}{2}\|h\|^{2}_{n}=\frac{3}{2}$. Then, using the chaining technique of \cite{baraud2001model}, we have
\begin{align*}
   \widetilde{\mathbb{P}}\left(\underset{h\in\mathcal{T}_{K,K^{\prime}}}{\sup}{\nu_1(h)}>\tau\right) & =\widetilde{\mathbb{P}}\left(\exists (h_k)_{k\geq 0}\in\prod_{k\geq 0}{E_{\varepsilon_k}}/ \nu_1(h)=\nu_1(h_0)+\sum_{k=1}^{\infty}{\nu_1(h_k-h_{k-1})}>\tau\right)\\
    &\leq\sum_{h_0\in E_0}{\widetilde{\mathbb{P}}\left(\nu_1(h_0)>y^{n,N}_0\right)}+\sum_{k=1}^{\infty}{\sum_{\underset{h_{k-1}\in E_{\varepsilon_{k-1}}}{h_k\in E_{\varepsilon_k}}}{\widetilde{\mathbb{P}}\left(\nu_1(h_k-h_{k-1})>y^{n,N}_k\right)}}.
\end{align*}
According to Equation~\eqref{eq:nu-ErrorTerm} and Lemma~\ref{lm:LemmaAdaptation}, there exists a constant $C>0$ such that
\begin{align*}
    \widetilde{\P}\left(\nu_1(h_0) > y^{n,N}_{0}\right) \leq &~ \widetilde{\P}\left(\nu_1(h_0) > \sigma_1\sqrt{6x^{n,N}_{0}}+\sigma^{2}_{1}\|h_{0}\|_{\infty}x^{n,N}_{0}\right) \\
    \leq &~ \exp(-CNnx^{n,N}_{0}), \\
    \forall k\geq 1, ~ \widetilde{\P}\left(\nu_1(h_k - h_{k-1}) > y^{n,N}_{0}\right) \leq &~ \widetilde{\P}\left(\nu_1(h_k - h_{k-1}) > \sigma_1\sqrt{6x^{n,N}_{k}}+\sigma^{2}_{1}\|h_{k} - h_{k-1}\|_{\infty}x^{n,N}_{k}\right) \\
    \leq &~ \exp(-CNnx^{n,N}_{k}).
\end{align*}
Finally, since $N_k = \log(\left|E_{\varepsilon_k}\right|)$ for all $k\geq 0$, we deduce that 
\begin{align*}
    \widetilde{\mathbb{P}}\left(\underset{h\in\mathcal{T}_{K,K^{\prime}}}{\sup}{\nu_1(h)}>\tau\right) &\leq \left|E_{\varepsilon_0}\right|\exp(-CNnx^{n,N}_{0}) + \sum_{k=1}^{\infty}{\left(\left|E_{\varepsilon_k}\right|+\left|E_{\varepsilon_{k-1}}\right|\right)\exp(-CNnx^{n,N}_{k})} \\
    &\leq\exp\left(N_0-CNnx^{n,N}_0\right)+\sum_{k=1}^{\infty}{\exp\left(N_k+N_{k-1}-CNnx^{n,N}_k\right)}.
\end{align*}
We choose $x^{n,N}_0$ and $x^{n,N}_k, \ \ k\geq 1$ such that,
\begin{align*}
N_0 - CNnx^{n,N}_0 = & -a\left(K+K^{\prime} + 2M\right)-b \\ 
N_k + N_{k-1} - CNnx^{n,N}_k = & -k(K + K^{\prime}+2M) - a\left(K + K^{\prime} + 2M\right) - b
\end{align*}
where $a$ and $b$ are two positive real numbers. We deduce that
\begin{equation}
\label{eq:xk - tau}
x^{n,N}_k\leq C_0(1+k)\frac{K + K^{\prime}+2M}{Nn} \ \ \mathrm{and} \ \ \tau\leq C_1\sigma^{2}_{1}\sqrt{\sqrt{L}\frac{K + K^{\prime}+2M}{Nn}}
\end{equation}
with $C_0>0$ and $C_{1}$ two constants depending on $a$ and $b$. It \textcolor{black}{follows} that
\begin{align*}
\overset{\sim}{\mathbb{P}}\left(\underset{t\in\mathcal{T}_{K,K^{\prime}}}{\sup}{\nu(t)}>\tau\right)&\leq\frac{\mathrm{e}}{\mathrm{e}-1}\mathrm{e}^{-b}\exp\left\{-a\left(K + K^{\prime} + 2M\right)\right\}.
\end{align*}
From Equation~\eqref{eq:xk - tau}, we set 
$$\textcolor{black}{q_1\left(K,K^{\prime}\right)}=\kappa^{*}\sigma^{2}_{1}\sqrt{L}\frac{K + K^{\prime} + 2M}{Nn}$$ 
where $\kappa^{*}>0$ depends on $C_1>0$. 
Thus, for all $K,K^{\prime}\in\mathcal{K}$,
\begin{equation*}
\mathbb{P}\left(\left\{\underset{h\in\mathcal{T}_{K,K^{\prime}}}{\sup}{\nu^{2}(h)}>\textcolor{black}{q_1(K,K^{\prime})}\right\}\cap\Omega_{n,N,K_{\max}}\right)\leq \frac{\mathrm{e}^{-b+1}}{\mathrm{e}+1}\exp\left\{-a\left(K + K^{\prime} + 2M\right)\right\}
\end{equation*}
and there exists constants $c,C>0$ such that
\begin{multline*}
    \mathbb{E}\left[\left(G_{K}\left(K^{\prime}\right)-\textcolor{black}{q_1(K,K^{\prime})}\right)_{+}\one_{\Omega_{n,N,K_{\max}}}\right]\\
    \leq \frac{c(K+K^{\prime})}{Nn}\mathbb{P}\left(\left\{\underset{t\in\mathcal{T}_{K,K^{\prime}}}{\sup}{\nu^{2}(t)} > \textcolor{black}{q_1(K,K^{\prime})}\right\}\cap\Omega_{n,N,K_{\max}}\right)\\
    \leq \frac{C}{Nn}\exp\left\{-\frac{a}{2}\left(K+K^{\prime}\right)\right\}.
\end{multline*}
Finally, there exists a real constant $C>0$ such that, 
\begin{multline*}
\mathbb{E}\left[\left(G_{K}(\widehat{K})-\textcolor{black}{q_1(K,\widehat{K})}\right)_{+}\one_{\Omega_{n,N,K_{\max}}}\right]\leq \sum_{K^{\prime}\in\mathcal{K}}{\mathbb{E}\left[\left(G_{K}\left(K^{\prime}\right)-\textcolor{black}{q_1(K,K^{\prime})}\right)_{+}\one_{\Omega_{n,N,K_{\max}}}\right]}\leq\frac{C}{Nn}.
\end{multline*}
We choose the penalty function $\mathrm{pen}$ such that for each $K\in\mathcal{K}$, 
$$\textcolor{black}{\mathrm{pen}_1(K)}\geq\kappa\sigma^{2}_{1}\sqrt{L}\frac{K+M}{Nn}.$$ 
For $N$ large enough, one has $\sigma^{2}_{1}\leq\sqrt{L}$. Thus, we finally set $\textcolor{black}{\mathrm{pen}_1(K)}=\kappa\frac{(K+M)\log(N)}{Nn}$ with $L = \log(N)$. Then, there exists a constant $C>0$ such that,
\begin{align*}
\mathbb{E}\left[\left\|\widehat{\sigma}^{2}_{\widehat{K}}-\sigma^{2}_{|I}\right\|^{2}_{n,N}\right] & \leq 34\underset{K\in\mathcal{K}}{\inf}{\left\{\underset{h\in\mathcal{S}_{K+M}}{\inf}{\left\|h-\sigma^{2}_{|I}\right\|^{2}_{n}}+\textcolor{black}{\mathrm{pen}_1(K)}\right\}}+\frac{C}{Nn}.
\end{align*}
\end{proof}

\subsubsection{Proof of Theorem~\ref{thm:adaptive estimator - non compact interval}~}

\begin{proof}
     From Equation~\eqref{eq:selection of the dimension}, we have
     $$ \w{K} := \underset{K\in\mathcal{K}}{\arg\min}{\gamma_{n,N}(\w{\sigma}^{2}_{K+M,L}) + \textcolor{black}{\mathrm{pen}_2(K)}}. $$
     Then, for all $K\in\mathcal{K}$ and for all $h \in \mathcal{S}_{K+M,L}$, we have
     \begin{align*}
         \gamma_{n,N}(\w{\sigma}^{2}_{\w{K},L}) + \textcolor{black}{\mathrm{pen}_2(\w{K})} \leq \gamma_{n,N}(h) + \textcolor{black}{\mathrm{pen}_2(K)}.
     \end{align*}
     Then, for all $K\in\mathcal{K}$ and for all $h\in\mathcal{S}_{K+M,L}$,
     \begin{align*}
         \gamma_{n,N}(\w{\sigma}^{2}_{\w{K},L}) - \gamma_{n,N}(\sigma^{2}) \leq &~ \gamma_{n,N}(h) - \gamma_{n,N}(\sigma^{2}) + \textcolor{black}{\mathrm{pen}_2(K) - \mathrm{pen}_2(\w{K})} \\
         \left\|\w{\sigma}^{2}_{\w{K},L} - \sigma^{2}\right\|^{2}_{n,N} \leq &~ \|h - \sigma^{2}\|^{2}_{n,N} + 2\nu(\w{\sigma}^{2}_{\w{K},L} - h) + 2\mu(\w{\sigma}^{2}_{\w{K},L} - h) + \textcolor{black}{\mathrm{pen}_2(K) - \mathrm{pen}_2(\w{K})}.
     \end{align*}
     We have for all $a>0$,
     \begin{align*}
         2\mathbb{E}\left[\mu\left(\widehat{\sigma}^{2}_{\w{K},L}-h\right)\right] \leq & \frac{2}{a}\mathbb{E}\left\|\widehat{\sigma}^{2}_{\w{K},L}-\sigma^{2}\right\|^{2}_{n,N}+\frac{2}{a}\underset{h\in\mathcal{S}_{K+M,L}}{\inf}{\|h-\sigma^{2}\|^{2}_{n}} +\frac{a}{Nn}\sum_{j=1}^{N}{\sum_{k=0}^{n-1}{\mathbb{E}\left[\left(R^{j}_{k\Delta}\right)^2\right]}}
     \end{align*}
     and since $\nu = \nu_1 + \nu_2 + \nu_3$, according to the proof of Theorem~\ref{thm:RiskBound-AnySupport}, there exists a constant $c>0$ such that
     $$ \E\left[\nu(\w{\sigma}^{2}_{\w{K},L} - h)\right] \leq c\E\left[\nu_1(\w{\sigma}^{2}_{\w{K},L} - h)\right] $$
     where the for $i\in\{1,2,3\}$ and for all $h \in \mathcal{S}_{K+M,L}, ~ K\in\mathcal{K}$,
     $$ \nu_i(h) = \frac{1}{Nn}\sum_{j=1}^{N}{\sum_{k=0}^{n-1}{h(X^{j}_{k\Delta})\zeta^{j}_{k\Delta}}},$$
     and the $\zeta^{j}_{k\Delta}$ are given Then,
     \begin{align*}
         \left(1-\frac{2}{a}\right)\E\left[\left\|\w{\sigma}^{2}_{\w{K},L} - \sigma^{2}\right\|^{2}_{n,N}\right] \leq &~ \left(1+\frac{2}{a}\right)\underset{h\in\mathcal{S}_{K+M,L}}{\inf}{\left\|h-\sigma^{2}\right\|^{2}_{n}} + 2c\E\left[\nu_1(\w{\sigma}^{2}_{\w{K},L} - h)\right] \\
         &+ \textcolor{black}{\mathrm{pen}_2(K)-\mathrm{pen}_2(\w{K})} + \frac{a}{Nn}\sum_{j=1}^{N}{\sum_{k=0}^{n-1}{\E\left[\left(R^{j}_{k\Delta}\right)^2\right]}}
     \end{align*}
From Equation~\eqref{eq:UpperBound-TimeStep} and for $a = 4$, there exists a constant $C>0$ such that
\begin{equation}
\label{eq:Adaptive.R-bound0}
     \E\left[\left\|\w{\sigma}^{2}_{\w{K},L} - \sigma^{2}\right\|^{2}_{n,N}\right] \leq 3\underset{h\in\mathcal{S}_{K+M,L}}{\inf}{\left\|h-\sigma^{2}\right\|^{2}_{n}} + 4c\E\left[\nu_1(\w{\sigma}^{2}_{\w{K},L} - h)\right] + 2\left(\textcolor{black}{\mathrm{pen}_2(K)-\mathrm{pen}_2(\w{K})}\right) + C\Delta^{2}.
\end{equation}
Since for all $K\in\mathcal{K}, ~ \textcolor{black}{\mathrm{pen}_2(K) \geq 2\kappa^{*}\sigma^{2}_{1}\sqrt{(K+M)L/(Nn)}}$, define the function \textcolor{black}{$q_2: (K,K^{\prime}) \mapsto q_2(K,K^{\prime})$} such that 
\begin{equation}
\label{eq:Adaptive.penalty.q}
    \textcolor{black}{q_2(K,K^{\prime})  = 2C^{*}\sigma^{2}_{1}\sqrt{\frac{(K+K^{\prime}+2M)L}{Nn}} \geq 2\sigma^{2}_{1}v\sqrt{x^{n,N}} + \sigma^{2}_{1}vx^{n,N}} 
\end{equation}
where
$$ \textcolor{black}{x^{n,N} \propto \frac{K+K^{\prime}+2M}{Nn} ~~ \mathrm{and} ~~ v = \sqrt{2L}}. $$
The constant $C^{*}>0$ depends on constants $\kappa^{*}>0$ and $c>0$ of Equation~\eqref{eq:Adaptive.R-bound0} such that
$$ \textcolor{black}{4cq_2(K,K^{\prime}) \leq \mathrm{pen}_2(K) + 2\mathrm{pen}_2(K^{\prime})}. $$
Then for all $K \in \mathcal{K}$ and for all $h\in\mathcal{S}_{K+M,L}$,
\begin{equation}
\label{eq:Adaptive.R-bound1}
     \E\left[\left\|\w{\sigma}^{2}_{\w{K},L} - \sigma^{2}\right\|^{2}_{n,N}\right] \leq 3\left(\underset{h\in\mathcal{S}_{K+M,L}}{\inf}{\left\|h-\sigma^{2}\right\|^{2}_{n}} + \textcolor{black}{\mathrm{pen}_2(K)}\right) + 4c\E\left[\left(\nu_1(\w{\sigma}^{2}_{\w{m},L} - h) - \textcolor{black}{q_2(K,\w{K})}\right)_{+}\right] + C\Delta^{2}.
\end{equation}
For all $K\in\mathcal{K}$ and for all $h\in\mathcal{S}_{K+M,L}$ such that $\|h\|_{\infty} \leq \sqrt{L}$, we have ,
\begin{equation*}
    \left\|\w{\sigma}^{2}_{\w{K},L} - h\right\|^{2}_{n,N} \leq \left\|\w{\sigma}^{2}_{\w{K},L} - h\right\|^{2}_{\infty} \leq 2L =: v^2.
\end{equation*}
Then, using Equation~\eqref{eq:Adaptive.penalty.q} and Lemma~\ref{lm:LemmaAdaptation}, there exists a constant $C>0$ such that for all $K,K^{\prime}\in\mathcal{K}$ and for all $h\in\mathcal{S}_{K+M,L}$,
\begin{equation}
\label{eq:Adaptive.R-bound2}
   \P\left(\nu_1(\w{\sigma}^{2}_{K^{\prime},L} - h) \geq q(K,K^{\prime}), ~ \|\w{\sigma}^{2}_{\w{K},L} - h\|^{2}_{n,N} \leq v^2\right) \leq \exp\left(-CNnx^{n,N}\right).
\end{equation}
Since $\textcolor{black}{L = \log(Nn)}$, then for $N$ large enough, $\textcolor{black}{\sigma^{2}_{1}\leq \sqrt{\log(Nn)}}$, we finally choose 
$$\textcolor{black}{\mathrm{pen}_2(K) = \kappa \sqrt{\frac{(K+M)\log^{2}(Nn)}{Nn}}}$$
where $\kappa>0$ is a new constant. Since $\E\left[\nu_1(\w{\sigma}^{2}_{\w{K},L} - h)\right] \leq \mathrm{O}\left(\sqrt{\frac{(K_{\max}+M)\log^{2}(N)}{Nn}}\right)$ (see proof of Theorem~\ref{thm:RiskBound-AnySupport}), for all $K \in \mathcal{K}$ and $h \in \mathcal{S}_{K+M,L}$, there exists a constant $c>0$ such that 
\begin{align*}
    \E\left[\left(\nu_1(\w{\sigma}^{2}_{\w{K},L} - h) - \textcolor{black}{q_2(K,\w{K})}\right)_{+}\right] \leq &~ \underset{K^{\prime} \in \mathcal{K}}{\max}{\left\{\E\left[\left(\nu_1(\w{\sigma}^{2}_{K^{\prime},L} - h) - \textcolor{black}{q_2(K,K^{\prime})}\right)_{+}\right]\right\}} \\
    \leq &~ c\textcolor{black}{q_2(K,K_{\max})}\underset{K^{\prime}\in\mathcal{K}}{\max}{\left\{\P\left(\nu_1(\w{\sigma}^{2}_{K^{\prime},L} - h) \geq \textcolor{black}{q_2(K,K^{\prime})}\right)\right\}}.
\end{align*}
From Equation~\eqref{eq:Adaptive.R-bound2}, we obtain that
\begin{equation}
\label{eq:adaptive.Rbound2bis}
    \textcolor{black}{\E\left[\left(\nu_1(\w{\sigma}^{2}_{\w{K},L} - h) - q_2(K,\w{K})\right)_{+}\right] \leq cq_2(K,K_{\max})\exp(-C(K+K_{\max}))\leq \frac{C}{\sqrt{Nn}}}
\end{equation}
since $K$ and $K_{\max}$ increase with the size $N$ of the sample paths $D_{N,n}$, and
$$ \textcolor{black}{c\sqrt{Nn}q_2(K,K_{\max})\exp(-C(K+K_{\max})) \rightarrow 0 ~~ \mathrm{as} ~~ N, n \rightarrow \infty ~~ (\mathrm{or} ~ N = 1 ~ \mathrm{and} ~ n \rightarrow \infty).} $$
Then, from Equations~\eqref{eq:adaptive.Rbound2bis}~and~\eqref{eq:Adaptive.R-bound1}, there exists a constant $C>0$ such that
\begin{equation*}
   \E\left[\left\|\w{\sigma}^{2}_{\w{K},L} - \sigma^{2}\right\|^{2}_{n,N}\right] \leq 3\underset{K\in\mathcal{K}}{\inf}{\left\{\underset{h\in\mathcal{S}_{K+M,L}}{\inf}{\left\|h-\sigma^{2}\right\|^{2}_{n}} + \textcolor{black}{\mathrm{pen}_2(K)}\right\}} + \textcolor{black}{\frac{C}{\sqrt{Nn}}}.
\end{equation*}
\end{proof}

\subsubsection{Proof of Theorem~\ref{thm:AdaptiveEstimator-OnePath}}

\begin{proof}
    The proof of Theorem~\ref{thm:AdaptiveEstimator-OnePath} is similar to the proof of Theorem~\ref{thm:adaptive estimator - compact interval}. Then, from \textcolor{black}{Equation~\eqref{eq:equation3-proof1}}, for all $h\in\mathcal{S}_{K+M}$,
    \begin{align*}
        \textcolor{black}{\mathbb{E}\left[\left\|\widehat{\sigma}^{2}_{\w{K}}-\sigma^{2}_{|I}\right\|^{2}_{n,1}\one_{\Omega_{n,K_{\max}}}\right]\leq} &~  \textcolor{black}{3\underset{h\in\mathcal{S}_{K+M}}{\inf}{\left\|h-\sigma^{2}_{|I}\right\|^{2}_{n}}+20\mathbb{E}\left(\underset{h\in\mathcal{S}_{K+M}, \|h\|_{X}=1}{\sup}{\nu^{2}(h)}\mathds{1}_{\Omega_{n,K_{\max}}}\right)+C\Delta^2}\\
        &+ \textcolor{black}{2\left(\mathrm{pen}(K)-\mathrm{pen}(\widehat{K})\right)},
    \end{align*}
    where $\mathcal{T}_{K,K^{\prime}} = \{h \in \mathcal{S}_{K+M}+\mathcal{S}_{K^{\prime}+M}, ~ \|h\|_{X} = 1, ~ \|h\|_{\infty}\leq\sqrt{L}\}$. Let \textcolor{black}{$q_3: \mathcal{K}^{2} \longrightarrow \R_{+}$} such that \textcolor{black}{$160q_3(K,K^{\prime}) \leq 18\mathrm{pen}_3(K) + 16\mathrm{pen}_3(K^{\prime})$}.  
   
  Recall that the $\mathbb{L}^{2}-$norm $\|.\|$, the norm $\E\left[\|.\|_{X}\right]$ and the empirical norm $\|.\|_{n}$ are equivalent on $\mathbb{L}^{2}(I)$ since the transition density is bounded on the compact interval $I$. Then, for all $K \in \mathcal{K}$ and $h \in \mathcal{S}_{K+M,L},$ we have
\begin{align*}
   \E\left[\left\|\widehat{\sigma}^{2}_{\widehat{K}}-\sigma^{2}_{|I}\right\|^{2}_{n,1}\right] \leq &~ 3\left(\underset{h\in\mathcal{S}_{K+M,L}}{\inf}{\left\|h-\sigma^{2}_{|I}\right\|^{2}_{n}}+\textcolor{black}{\mathrm{pen}_3(K)}\right) + \textcolor{black}{20\E\left[\underset{h\in\mathcal{T}_{K,\w{K}}}{\sup}{\nu^{2}_{1}(h)}-q_3(K,\w{K})\mathds{1}_{\Omega_{n,K_{\max}}}\right]} \\
   &+ C\Delta^{2} + \textcolor{black}{CL\dfrac{K^{2\gamma}_{\max}}{n^{\gamma/2}}},
\end{align*}
 where \textcolor{black}{the numerical constant $\gamma$ can take any value in $(1,+\infty)$, and} 
\begin{equation*}
    \nu_{1}(h):=\frac{1}{n}\sum_{k=0}^{n-1}{h(X^{1}_{k\Delta})\zeta^{1,1}_{k\Delta}}
\end{equation*}
with $\zeta^{1,1}_{k\Delta}$ the error term. We set for all $K,K^{\prime} \in \mathcal{K}$, $G_{K}(K^{\prime}):=\underset{h\in\mathcal{T}_{K,K^{\prime}}}{\sup}{\nu^{2}_{1}(h)}$ \textcolor{black}{and $\gamma = 20$}. Then, \textcolor{black}{for $K_{\max} = n^{1/5}$ and $L = \log(n)$}, there exists $C>0$ such that
\begin{align*}
\E\left[\left\|\w{\sigma}^{2}_{\widehat{K}}-\sigma^{2}_{|I}\right\|^{2}_{n,1}\right] \leq &~ 3\underset{K\in\mathcal{K}}{\inf}{\left\{\underset{h\in\mathcal{S}_{K+M,L}}{\inf}{\left\|h-\sigma^{2}_{|I}\right\|^{2}_{n}}+\textcolor{black}{\mathrm{pen}_3(K)}\right\}} + C\Delta^{2} \textcolor{black}{ +  C\dfrac{\log(n)}{n^{2}}}\\
& + 20\sum_{K^{\prime}\in\mathcal{K}}{\E\left[\left(G_{K}(K^{\prime})-\textcolor{black}{q_3(K,K^{\prime})}\right)_{+}\textcolor{black}{\mathds{1}_{\Omega_{n,K_{\max}}}}\right]} . 
\end{align*}
Considering the unit ball $\overline{B}_{\|.\|_X}(0,1)$ of the approximation subspace given by 
$$ \overline{B}_{\|.\|_X}(0,1) = \left\{h\in\mathcal{S}_{K+M}, ~ \|h\|^{2}_{X} \leq 1\right\} = \left\{h\in\mathcal{S}_{K+M}, ~ \|h\|^{2} \leq \frac{1}{\tau_0}\right\}. $$
\textcolor{black}{Then, using a similar reasoning as in Theorem~\ref{thm:adaptive estimator - compact interval}, we obtain that},
$$ \sum_{K^{\prime}\in\mathcal{K}}{\E\left[\left(G_{K}(K^{\prime}) - \textcolor{black}{q_3(K,K^{\prime})}\right)_{+}\right]} \leq \frac{C}{n}, $$
where $C>0$ is a constant, $\textcolor{black}{q_3(K,K^{\prime})} \propto \sigma^{4}_{1}\frac{(K+K^{\prime}+2M)\sqrt{\log(n)}}{n}$ and $\textcolor{black}{\mathrm{pen}_3(K)} \propto \frac{(K+M)\log(n)}{n}$. Then we obtain
\begin{equation*}
    \mathbb{E}\left[\left\|\widehat{\sigma}^{2}_{\widehat{K},L}-\sigma^{2}_{|I}\right\|^{2}_{n,1}\right] \leq \frac{3}{\tau_0}\underset{K\in\mathcal{K}}{\inf}{\left\{\underset{h\in\mathcal{S}_{K+M,L}}{\inf}{\left\|h-\sigma^{2}_{|I}\right\|^{2}_{n}}+\textcolor{black}{\mathrm{pen}_3(K)}\right\}} + \frac{C}{n}.
\end{equation*}
\end{proof}

\newpage


\bibliographystyle{ScandJStat}
\bibliography{mabiblio.bib}


\newpage

Eddy Ella-Mintsa, LAMA, Université Gustave Eiffel, 16 Bd Descartes, 77420, Champs-sur-Marne, France\\
Email: \href{mailto:eddy-michel.ella-mintsa@univ-eiffel.fr}{\texttt{eddy-michel.ella-mintsa@univ-eiffel.fr}}

\newpage

\appendix{\huge Appendix}
\label{app:appendix}

\subsection*{Calibration}

Fix the drift function $b(x) = 1-x$, the time-horizon $T=1$ and at time $t=0, ~ x_0=0$. Consider the following three models:
\begin{itemize}
    \item[] Model $1$: $\sigma(x)=1$ 
    \item[] Model $2$: $\sigma(x)=0.1+0.9/\sqrt{1+x^2}$
    \item[] Model $3$: $\sigma(x) = 1/3+\sin^{2}(2\pi x)/\pi + 1/(\pi+x^2)$.
\end{itemize}

The three diffusion models satisfy Assumption~\ref{ass:Assumption 1}~ and are used to calibrate the numerical constant $\kappa$ \textcolor{black}{of each penalty function.}

As we already know, the adaptive estimator of $\sigma^{2}$ on the interval $[-\sqrt{\log(N)}, \sqrt{\log(N)}]$ necessitate a data-driven selection of an optimal dimension through the minimization of the penalized least squares contrast given in Equation~\eqref{eq:selection of the dimension}~. Since the penalty function $\mathrm{pen}(d_N)=\kappa (K_N+M)\log^2(N)/N^2$ depends on the unknown numerical constant $\kappa>0$, the goal is to select an optimal value of $\kappa$ in the set $\mathcal{V}=\{0.1,0.5,1,2,4,5,7,10\}$ of its possible values. To this end, we repeat $100$ times the following steps:
\begin{enumerate}
    \item Simulate learning samples $D_N$ and $D_{N^{\prime}}$ with $N\in\{50,100\}$, $N^{\prime}=100$ and $n \in \{100, 250\}$
    \item For each $\kappa\in\mathcal{V}$:
          \begin{enumerate}
              \item For each $K_N\in\mathcal{K}$ and from $D_N$, compute $\widehat{\sigma}^{2}_{d_N,L_N}$ given in Equations~\eqref{eq:non adaptive estimator}~and~\eqref{eq:ridge estimator-non adaptive}.
              \item Select the optimal dimension $\widehat{K}_N\in\mathcal{K}$ using Equation~\eqref{eq:selection of the dimension}~
              \item Using the learning sample $D_{N^{\prime}}$, evaluate $\left\|\widehat{\sigma}^{2}_{\widehat{d}_N,L_N}-\sigma^{2}_{A}\right\|^{2}_{n,N^{\prime}}$ where $\widehat{d}_N=\widehat{K}_N+M$.
          \end{enumerate}
\end{enumerate}

Then, we calculate average values of $\left\|\widehat{\sigma}^{2}_{\widehat{d}_N,L_N}-\sigma^{2}_{A}\right\|^{2}_{n,N^{\prime}}$ for each $\kappa\in\mathcal{V}$ and obtain the following results:

\begin{figure}[hbtp]
    \centering
    \includegraphics[width=0.75\linewidth, height=0.35\textheight]{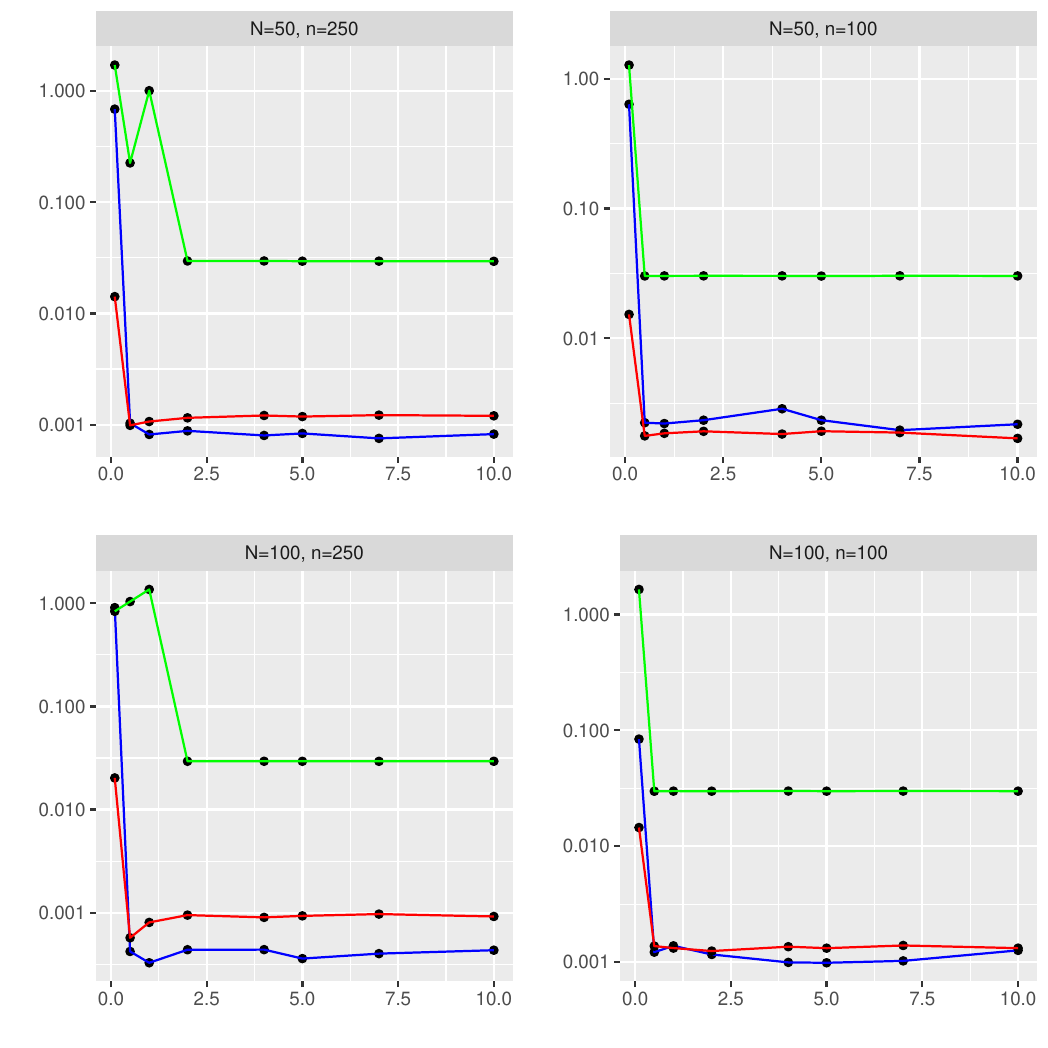}
    \caption{Calibration of the constant $\kappa\in\mathcal{V}$ of the penalty function}
    \label{fig:calibration}
\end{figure}

We finally choose $5\in\mathcal{V}$ as the optimal value of $\kappa$ in reference to the results of Figure~\ref{fig:calibration}.


\begin{proof}[\textbf{Proof of Lemma}~\ref{lm:OccupationFormula}]
The proof is divided into two parts for each of the two results to be proven.

\paragraph{First result.}

Since the function $h$ is continuous on $\R$, let $H$ be a primitive of $h$ on $\R$. We deduce that for all $s \in [0,1]$, 
    \begin{equation*}
        h(X_s) = \underset{\varepsilon \rightarrow 0}{\lim}{\frac{H(X_s + \varepsilon) - H(X_s - \varepsilon)}{2\varepsilon}} = \underset{\varepsilon\rightarrow 0}{\lim}{\frac{1}{2\varepsilon}\int_{X_s - \varepsilon}^{X_s + \varepsilon}{h(x)dx}} = \underset{\varepsilon\rightarrow 0}{\lim}{\frac{1}{2\varepsilon}\int_{-\infty}^{+\infty}{h(x)\one_{(x-\varepsilon,x+\varepsilon)}(X_s)dx}}.
    \end{equation*}
Finally, since $h$ is integrable on $\R$ and using the theorem of dominated convergence, we obtain
\begin{equation*}
        \int_{0}^{1}{h(X_s)ds} = \int_{-\infty}^{+\infty}{h(x)\underset{\varepsilon\rightarrow 0}{\lim}{\frac{1}{2\varepsilon}\int_{0}^{1}{\one_{(x-\varepsilon,x+\varepsilon)}(X_s)ds}}dx} = \int_{-\infty}^{+\infty}{h(x)\mathcal{L}^{x}dx}.
    \end{equation*}

\paragraph{Second result.}

    Fix $t\in(0,1]$ and consider $P_X : (t,x) \mapsto \int_{-\infty}^{x}{p_X(t,y)dy}$ the cumulative density function of the random variable $X_t$ of the density function $x \mapsto p_X(t,x)$. We have:
    \begin{align*}
       \forall~x\in\R, ~~ \E(\mathcal{L}^{x}) = & \underset{\varepsilon\rightarrow 0}{\lim}{\frac{1}{2\varepsilon}\int_{0}^{1}{\E\left[\one_{(x-\varepsilon,x+\varepsilon)}(X_s)\right]ds}} = \underset{\varepsilon\rightarrow 0}{\lim}{\frac{1}{2\varepsilon}\int_{0}^{1}{\P\left(x - \varepsilon \leq X_s\leq x + \varepsilon\right)ds}} \\
       = & \int_{0}^{1}{\underset{\varepsilon\rightarrow 0}{\lim}{\frac{P_X(s,x+\varepsilon) - P_X(s,x-\varepsilon)}{2\varepsilon}}ds} \\
       = & \int_{0}^{1}{p_X(s,x)ds}.
    \end{align*}
\end{proof}


\begin{proof}[\textbf{Proof of Lemma}~\ref{lm:LocalTimeBicontinuous}]
    From \cite{revuz2013continuous}, {\it Theorem 1.7}, we have
    \begin{equation*}
        \forall~x\in\R, ~~ \mathcal{L}^{x} - \mathcal{L}^{x_{-}} = 2\int_{0}^{1}{\one_{X_s = x}dX_s} = 2\int_{0}^{1}{\one_{X_s = x}b(X_s)ds} + 2\int_{0}^{1}{\one_{X_s = x}\sigma(X_s)dW_s}.
    \end{equation*}
  \textcolor{black}{For each $s \in [0,1]$, let $F_s$ be the cumulative density function of the random variable $X_s$.} For all $x\in\R$ and for all $s\in[0,1]$, we have for all $\varepsilon>0$,
  \begin{align*}
      \P(X_s = x) = &~ \underset{\varepsilon\rightarrow 0}{\lim}{~\P(X_s\leq x + \varepsilon)} - \underset{\varepsilon\rightarrow 0}{\lim}{~\P(X_s\leq x - \varepsilon)} = \underset{\varepsilon\rightarrow 0}{\lim}{~F_s(x + \varepsilon)} - \underset{\varepsilon\rightarrow 0}{\lim}{~F_s(x - \varepsilon)} \\
      = &~ F_s(x) - F_s(x^{-}) \\
      = &~ 0
  \end{align*}
  Thus, for all $x\in\R$,
    \begin{align*}
        \E\left[\left|\mathcal{L}^{x} - \mathcal{L}^{x_{-}}\right|\right] \leq &~ 2\int_{0}^{1}{|b(x)|\P(X_s = x)ds} + 2\E\left[\left|\int_{0}^{1}{\one_{X_s=x}\sigma(X_s)dW_s}\right|\right] \\
        = &~ 2\E\left[\left|\int_{0}^{1}{\one_{X_s=x}\sigma(X_s)dW_s}\right|\right].
    \end{align*}

    Using the Cauchy Schwartz inequality, we conclude that
    \begin{align*}
        \E\left[\left|\mathcal{L}^{x} - \mathcal{L}^{x_{-}}\right|\right] \leq &~ 2\sqrt{\E\left(\int_{0}^{1}{\one_{X_s=x}\sigma^{2}(X_s)ds}\right)} = 2\sigma(x)\int_{0}^{1}{\P(X_s = x)ds} = 0.
    \end{align*}

    Using the Markov inequality, we have
    \begin{equation*}
        \forall~\varepsilon>0, ~~ \P(|\mathcal{L}^{x} - \mathcal{L}^{x_{-}}|>\varepsilon) \leq \frac{1}{\varepsilon}\E\left[\left|\mathcal{L}^{x} - \mathcal{L}^{x_{-}}\right|\right] = 0.
    \end{equation*}

    We finally conclude that for all $x \in \R$,
    \begin{equation*}
        \P(\mathcal{L}^{x}\neq \mathcal{L}^{x_{-}}) = \P(|\mathcal{L}^{x} - \mathcal{L}^{x_{-}}|>0) = 0.
    \end{equation*}
\end{proof}


\begin{proof}[\textbf{Proof of Lemma~\ref{lm:Proba-OmegaComp-OnePath}~}]

The proof of this Lemma mainly \textcolor{black}{focuses} on the spline basis and the Fourier basis based on functions $\cos$ and $\sin$ which are Lipschitz functions. Thus, for all $g = \sum_{\ell = 0}^{m-1}{a_{\ell}\phi_{\ell}} \in \mathcal{S}_{m}$,
\begin{equation}
\label{eq:Diff-NormXNorm-n.1}
    \left|\|g\|^{2}_{n,1} - \|g\|^{2}_{X}\right| \leq \int_{0}^{1}{\left|g^{2}(X_{\eta(s)}) - g^{2}(X_s)\right|ds} \leq 2\|g\|_{\infty}\int_{0}^{1}{\left|g(X_{\eta(s)}) - g(X_s)\right|ds}.
\end{equation}
From Equation~\eqref{eq:Equiv-NormX-NormL2}, one has $\E\left[\|g\|^{2}_{X}\right] \geq \tau_0 \|g\|^{2}$. Thus, if $\|g\|^{2}_{X} = 1$, then $\|g\|^{2}\leq 1/\tau_0$, and we deduce for all $g = \sum_{\ell=0}^{m-1}{a_{\ell}\phi_{\ell}}$ that there exists a constant $C>0$ such that
\begin{itemize}
    \item \textcolor{black}{{\bf Collection [B]}}: $\|g\|_{\infty} \leq \|a\|_2 \leq C\sqrt{m}$ ~ (see ~ \cite{denis2020ridge})
    \item \textcolor{black}{{\bf Collection [CS$-$OB]}: $\|g\|_{\infty} \leq Cm^{r/2}$ ~ since $\|g\| = \|a\|_2$ and $\sum_{\ell=0}^{m-1}{\phi^{2}_{\ell}} = \mathrm{O}(m^{r})$.}
\end{itemize}
Moreover, each $g\in\mathcal{S}_{m}$ such that $\|g\|^{2}_{X} = 1$ is \textcolor{black}{a} Lipschitz function with a Lipschitz coefficient $L_g = \mathrm{O}\left(m^{3/2}\right)$ 
 \textcolor{black}{for the collection [\textbf{B}], or $L_g = O\left(m^{r^{\prime}/2}\right)$ for the collection [\textbf{CS$-$OB}}]. For the spline basis, this result is obtained in \cite{denis2020ridge}, \textit{proof of Lemma C.1} combined with \textit{Lemma 2.6}. \textcolor{black}{For the collection [\textbf{CS$-$OB}], for all $x,y\in I$ and using the Cauchy$-$Schwarz inequality together with Equation~\eqref{eq:Condition-OB} and the finite increments theorem, there exists a constant $C > 0$ such that
\begin{align*}
    |g(x) - g(y)| \leq & \sum_{\ell = 0}^{m - 1}{|a_{\ell}|.|\phi_{\ell}(x) - \phi_{\ell}(y)|}\\
    \leq & \|\mathbf{a}\|_2\sqrt{\sum_{\ell = 0}^{m-1}{\left\|\phi^{\prime}_{\ell}\right\|^{2}_{\infty}}}|x-y| \\
    \leq & Cm^{r^{\prime}/2}|x-y|.
\end{align*}
}
Back to Equation~\eqref{eq:Diff-NormXNorm-n.1}, there exists a constant $C>0$ such that
\textcolor{black}{
\begin{equation}
\label{eq:Diff-NormX-Norm-n.1-BIS}
    \left|\|g\|^{2}_{n,1} - \|g\|^{2}_{X}\right| \leq Cm^{\alpha}\int_{0}^{1}{|X_{\eta(s)} - X_s|ds}
\end{equation}
with $\alpha = 2$ for the collection $[\mathbf{B}]$, and $\alpha = (r+r^{\prime})/2$ for the collection $[\mathbf{CS-OB}]$.
}
We have:
\begin{align*}
    \Omega^{c}_{n,m} = \left\{\omega \in \Omega, ~ \exists g\in\mathcal{S}_{m}\setminus\{0\}, ~ \left|\frac{\|g\|^{2}_{n,1}}{\|g\|^{2}_{X}}-1\right| > \frac{1}{2}\right\},
\end{align*}
and, using Equation~\eqref{eq:Diff-NormX-Norm-n.1-BIS}, we obtain
\begin{equation*}
    \underset{g\in\mathcal{S}_{m}\setminus\{0\}}{\sup}{\left|\frac{\|g\|^{2}_{n,1}}{\|g\|^{2}_{X}}-1\right|} = \underset{g\in\mathcal{S}_{m},~\|g\|^{2}_{X} = 1}{\sup}{\left|\|g\|^{2}_{n,1}-\|g\|^{2}_{X}\right|} \leq C\textcolor{black}{m^{\alpha}}\int_{0}^{1}{|X_{\eta(s)} - X_s|ds}.
\end{equation*}
Finally, using the Markov inequality, the H\"older inequality, Lemma~\eqref{lm:ConseqAssumption1}, and Lemma~\ref{lem:discrete-bis}, we conclude that
\begin{align*}
   \P\left(\Omega^{c}_{n,m}\right) \leq &~ \P\left(C\textcolor{black}{m^{\alpha}}\int_{0}^{1}{|X_{\eta(s)} - X_s|ds} \geq \frac{1}{2}\right) \\
   \leq &~ C\textcolor{black}{m^{\alpha\gamma}}\int_{0}^{1}{\E\left[|X_{\eta(s)} - X_s|^{\gamma}\right]ds}\\
   \leq &~ \textcolor{black}{C\frac{m^{\alpha\gamma}}{n^{\gamma/2}}}
\end{align*}
with $\gamma \in (1,+\infty)$.
\end{proof}


\begin{proof}[\textbf{Proof of Lemma}~\ref{lm:Proba-OmegaComp}~]
    We have:
    \begin{align*}
        \Omega^{c}_{n,N,m}=\left\{\omega\in\Omega, \ \exists h_0\in\mathcal{S}_{m}, \ \left|\frac{\|h\|^{2}_{n,N}}{\|h\|^{2}_{n}}-1\right|>\frac{1}{2}\right\},
    \end{align*}
Denote by $\mathcal{H}_{m} = \left\{h\in\mathcal{S}_{m}, ~ \|h\|_{n} = 1\right\}$.
    We have
\begin{align*}
    \underset{h\in\mathcal{H}_{m}}{\sup}{\left|\frac{\|h\|^{2}_{n,N}}{\|h\|^{2}_{n}}-1\right|} = \underset{h\in\mathcal{H}_{m}}{\sup}{\left|\|h\|^{2}_{n,N}-1\right|}.
\end{align*}
Let $\varepsilon > 0$ and let $\mathcal{H}^{\varepsilon}_{m}$ be the $\varepsilon-$net of $\mathcal{H}_{m}$ {\it w.r.t.} the supremum norm $\|.\|_{\infty}$. Then, for each $h\in\mathcal{H}_{m}$, there exists $h_{\varepsilon}\in\mathcal{H}^{\varepsilon}_{m}$ such that $\|h-h_{\varepsilon}\|_{\infty} \leq \varepsilon$. Then
\begin{align*}
    \left|\|h\|^{2}_{n,N} - 1\right| \leq  \left|\|h\|^{2}_{n,N} - \|h_{\varepsilon}\|^{2}_{n,N} \right| + \left|\|h_{\varepsilon}\|^{2}_{n,N} - 1\right|
\end{align*}
and,
\begin{align*}
    \left|\|h\|^{2}_{n,N} - \|h_{\varepsilon}\|^{2}_{n,N} \right| \leq &~ \frac{1}{Nn}\sum_{j=1}^{N}{\sum_{k=0}^{n-1}{\left|h(X^{j}_{k\Delta}) - h_{\varepsilon}(X^{j}_{k\Delta})\right|\left(\|h\|_{\infty} + \|h_{\varepsilon}\|_{\infty}\right)}} \leq \left(\|h\|_{\infty} + \|h_{\varepsilon}\|_{\infty}\right)\varepsilon.
\end{align*}
Moreover, we have $\|h\|^{2}, \|h_{\varepsilon}\|^{2} \leq 1/\tau_0$. Then, there exists a constant $\mathbf{c} > 0$ such that
\begin{align*}
     \begin{cases} 
     \left|\|h\|^{2}_{n,N} - \|h_{\varepsilon}\|^{2}_{n,N} \right| \leq 2\sqrt{\frac{{\bf c}m}{\tau_0}}\varepsilon ~~ \mathrm{for} ~ \mathrm{the} ~ \mathrm{spline} ~ \mathrm{basis} ~ (\mathrm{see ~ \textit{Lemma 2.6} ~ in ~ Denis~{\it et ~ al.}(2021)}) \\ \\
     \left|\|h\|^{2}_{n,N} - \|h_{\varepsilon}\|^{2}_{n,N} \right| \leq 2\sqrt{\frac{{\bf c}m}{\tau_0}}\varepsilon ~~ \mathrm{for ~ an ~ orthonormal ~ basis} ~~ (\|h\|^{2}_{\infty} \leq (\underset{0\leq\ell\leq m-1}{\max}{\|\phi_{\ell}\|^{2}_{\infty}})m\|h\|^{2}).
     \end{cases}
\end{align*}
Therefore, for all $\delta > 0$ and for both the spline basis and any orthonormal basis,
\begin{equation*}
    \P\left(\underset{h\in\mathcal{H}_{m}}{\sup}{\left|\|h\|^{2}_{n,N}-1\right|} \geq \delta\right) \leq \P\left(\underset{h\in\mathcal{H}^{\varepsilon}_{m}}{\sup}{\left|\|h\|^{2}_{n,N}-1\right|} \geq \delta/2\right) + \one_{4\varepsilon\sqrt{\frac{{\bf c}m}{\tau_0}} \geq \delta}.
\end{equation*}
We set $\delta = 1/2$ and we choose $\varepsilon > 0$ such that $4\varepsilon\sqrt{\frac{{\bf c}m}{\tau_0}} < 1/2$. Then, using the Hoeffding inequality, there exists a constant $c>0$ depending on ${\bf c}$ and $\tau_0$ such that
\begin{align}
\label{eq:HoeffInequality}
    \mathbb{P}\left(\Omega^{c}_{n,N,m}\right)\leq 2\mathcal{N}_{\infty}(\varepsilon,\mathcal{H}_{m})\exp\left(-c\frac{N}{m}\right)
\end{align}
where $\mathcal{N}_{\infty}(\varepsilon,\mathcal{H}_{m})$ is the covering number of $\mathcal{H}_{m}$ satisfying:
\begin{equation}
\label{eq:CovNumber}
    \mathcal{N}_{\infty}(\varepsilon,\mathcal{H}_{m}) \leq \left(\kappa\frac{\sqrt{m}}{\varepsilon}\right)^{m}
\end{equation}
where the constant $\kappa>0$ depends on $c>0$ (see \cite{denis2020ridge}, \textit{Proof of Lemma D.1}). We set $\varepsilon = \frac{\kappa \sqrt{m^{*}}}{N}$ with $m^{*} = \max{\mathcal{M}}$ and we derive from Equations~\eqref{eq:HoeffInequality}~and~\eqref{eq:CovNumber}~that
\begin{equation*}
    \P\left(\Omega^{c}_{n,N,m}\right) \leq 2N^{m^{*}}\exp\left(-c\frac{N}{m^{*}}\right) = 2\exp\left(-c\frac{N}{m^{*}}\left(1-\frac{m^{*2}\log(N)}{cN}\right)\right).
\end{equation*}
\begin{itemize}
    \item If $n \geq N$, then $m \in \mathcal{M} = \left\{1,\ldots,\sqrt{N}/\log(Nn)\right\}$. Since $m^{*2}\log(N)/N \rightarrow 0$ as $N\rightarrow +\infty$, there exists a constant $C>0$ such that
    \begin{align*}
        \mathbb{P}\left(\Omega^{c}_{n,N,m}\right)\leq 2\exp\left(-C\sqrt{N}\right).
    \end{align*}
    \item If $n \leq N$, then $m \in \mathcal{M} = \left\{1,\ldots, \sqrt{n}/\log(Nn)\right\}, ~ m^{*2}\log(N)/N \leq \log(N)/\log^{2}(Nn) \rightarrow 0$ as $N,n \rightarrow \infty$, and
    \begin{equation*}
        \mathbb{P}\left(\Omega^{c}_{n,N,m}\right)\leq 2\exp\left(-C\sqrt{n}\right).
    \end{equation*}
    \item If $n \propto N$, then $m \in \mathcal{M} = \left\{1,\ldots,\sqrt{N}/\log(Nn)\right\}$. Since $m^{*2}\log(N)/N \rightarrow 0$ as $N\rightarrow +\infty$, there exists a constant $C>0$ such that
    \begin{align*}
        \mathbb{P}\left(\Omega^{c}_{n,N,m}\right)\leq 2\exp\left(-C\sqrt{N}\right).
    \end{align*}
\end{itemize}
\end{proof}


\subsection*{Proof of Lemma~\ref{lm:LemmaAdaptation}~}

\begin{proof}
We obtain from \textit{Comte,Genon-Catalot,Rozenholc (2007) proof of Lemma 3}
that for each $j\in[\![1,N]\!], \ \ k\in[\![0,n-1]\!]$ and $p\in\mathbb{N}\setminus\{0,1\}$
\begin{align*}
\mathbb{E}\left[\exp\left(ug(X^{j}_{k\Delta})\xi^{j,1}_{k\Delta}-\frac{au^2g^{2}(X^{j}_{k\Delta})}{1-bu}\right)|\mathcal{F}_{k\Delta}\right]\leq 1
\end{align*}

with $a=\mathrm{e}\left(4\sigma^{2}_{1}c^2\right)^2, \ b=4\sigma^{2}_{1}c^2\mathrm{e}\|g\|_{\infty}$, $u\in\mathbb{R}$ such that $bu<1$ and $c>0$ a real constant. Thus,
\begin{multline*}
    \mathbb{P}\left(\frac{1}{Nn}\sum_{j=1}^{N}{\sum_{k=0}^{n-1}{g(X^{j}_{k\Delta})\xi^{j,1}_{k\Delta}}}\geq\varepsilon, \|g\|^{2}_{N,n}\leq v^2\right)=\mathbb{E}\left(\mathbf{1}_{\left\{\sum_{j=1}^{N}{\sum_{k=0}^{n-1}{ug(X^{j}_{k\Delta})\xi^{(j,1)}_{k\Delta}\geq Nnu\varepsilon}}\right\}}\mathbf{1}_{\|g\|^{2}_{n,N}\leq v^2}\right)\\
    =\mathbb{E}\left(\mathbf{1}_{\left\{\exp\left(\sum_{j=1}^{N}{\sum_{k=0}^{n-1}{ug(X^{j}_{k\Delta})\xi^{(j,1)}_{k\Delta}}}\right)\mathrm{e}^{-Nnu\varepsilon}\geq 1\right\}}\one_{\|g\|^{2}_{N,n}\leq v^2}\right)\\
    \leq\mathrm{e}^{-Nnu\varepsilon}\mathbb{E}\left[\one_{\|g\|^{2}_{n,N}\leq v^2}\exp\left\{\sum_{j=1}^{N}{\sum_{k=0}^{n-1}{ug(X^{j}_{k\Delta})\xi^{j,1}_{k\Delta}}}\right\}\right].
\end{multline*}

It follows that,
\begin{eqnarray*}
\mathbb{P}\left(\frac{1}{Nn}\sum_{j=1}^{N}{\sum_{k=0}^{n-1}{g(X^{j}_{k\Delta})\xi^{j,1}_{k\Delta}}}\geq\varepsilon, \|g\|^{2}_{n,N}\leq v^2\right)
\leq&~\exp\left\{-Nnu\varepsilon+\frac{Nnau^2v^2}{1-bu}\right\}.
\end{eqnarray*}

We set $u=\frac{\varepsilon}{\varepsilon b+2av^2}$. Then, we have $-Nnu\varepsilon+Nnav^2u^2/(1-bu)=-Nn\varepsilon^2/2(\varepsilon b+2av^2)$ and,
\begin{align*}
\mathbb{P}\left(\frac{1}{Nn}\sum_{j=1}^{N}{\sum_{k=0}^{n-1}{g(X^{j}_{k\Delta})\xi^{j,1}_{k\Delta}}}\geq\varepsilon, \|g\|^{2}_{n,N}\leq v^2\right)&\leq\exp\left(-\frac{Nn\varepsilon^2}{2(\varepsilon b+av^2)}\right)\\
&\leq\exp\left(-C\frac{Nn\varepsilon^2}{\sigma^{2}_{1}\left(\varepsilon\|g\|_{\infty}+4\sigma^{2}_{1}v^2\right)}\right)
\end{align*}

where $C>0$ is a constant depending on $c>0$.
\end{proof}


\textcolor{black}{
\begin{proof}[\textbf{Proof of Lemma~\ref{lm:ConseqAssumption1}}]
    Let $q \geq 1$. Suppose that $X_0 = 0$. Under Assumption~\ref{ass:Assumption 1}, we deduce from Equation \eqref{eq:model} that for all $t \in [0,1]$,
    \begin{align*}
        X_t = &~ \int_{0}^{t}{b(X_s)ds} + \int_{0}^{t}{\sigma(X_s)dW_s} \\
        |X_t| \leq &~ \int_{0}^{t}{|b(X_s)|ds} + \left|\int_{0}^{t}{\sigma(X_s)dW_s}\right| \\
        |X_t| \leq &~ \int_{0}^{t}{C_0(1+|X_s|)ds} + \left|\int_{0}^{t}{\sigma(X_s)dW_s}\right| 
    \end{align*}
    where the constant $C_0 > 0$ depends on the constant $L_0 > 0$ given in Assumption~\ref{ass:Assumption 1}. Let $p > 1$ be a real number such that $1/p + 1/q = 1$. From the H\"older inequality, there exists a constant $C_q$ depending on $q$ such that
    \begin{align*}
     \forall~ t \in [0,1], ~~ |X_t|^{q} \leq &~ C_q\left(\int_{0}^{t}{(1+|X_s|^{q})ds} + \left|\int_{0}^{t}{\sigma(X_s)dW_s}\right|^{q}\right).
    \end{align*}
    Using the Cauchy Schwarz inequality, one deduces that 
    \begin{align*}
      \mathbb{E}\left[\underset{t \in [0,1]}{\sup}{|X_t|^{q}}\right] \leq &~ C_{q}\left(\int_{0}^{1}{(1+\mathbb{E}\left[|X_s|^{q}\right])ds} + \mathbb{E}\left[\underset{t \in [0,1]}{\sup}{\left|\int_{0}^{t}{\sigma(X_s)dW_s}\right|^{q}}\right]\right)\\
      \leq &~ C_{q}\left(\int_{0}^{1}{(1+\mathbb{E}\left[|X_s|^{q}\right])ds} + \sqrt{\mathbb{E}\left[\underset{t \in [0,1]}{\sup}{\left|\int_{0}^{t}{\sigma(X_s)dW_s}\right|^{2q}}\right]}\right).
    \end{align*}
    From the Doob inequality (see \cite{revuz2013continuous}, corollary 1.6), one has
    \begin{align*}
        \mathbb{E}\left[\underset{t \in [0,1]}{\sup}{\left|\int_{0}^{t}{\sigma(X_s)dW_s}\right|^{q}}\right] \leq \left(\dfrac{2q}{2q-1}\right)^{q}\sqrt{\mathbb{E}\left[\left|\int_{0}^{1}{\sigma(X_s)dW_s}\right|^{2q}\right]}.
    \end{align*}
    moreover, using the Burkholder-Davis-Gundy inequality, there exists a constant $C_{1,q}$ depending on $q \geq 1$ such that 
    $$ \sqrt{\mathbb{E}\left[\left|\int_{0}^{1}{\sigma(X_s)dW_s}\right|^{2q}\right]} \leq C_{1,q}\mathbb{E}\left[\int_{0}^{1}{\sigma^{q}(X_s)ds}\right] < C_{1,q}\sigma^{q}_{1}. $$
    Then, we deduce that
    \begin{align*}
      \mathbb{E}\left[\underset{t \in [0,1]}{\sup}{|X_t|^{q}}\right] \leq &~ C_{q}\left(\int_{0}^{1}{(1+\mathbb{E}\left[|X_s|^{q}\right])ds} + C_{1,q}\sigma^{q}_{1}\right).
    \end{align*}
    Finally, from Proposition~\ref{prop:densityTransition-bis}, we have for all $q \geq 1$ and for all $s \in [0,1]$, $\mathbb{E}\left[|X_s|^{q}\right] < \infty$.
\end{proof}
}


\textcolor{black}{
\begin{lemme}\label{lm:RegModel}
    Under Assumption~\ref{ass:Assumption 1}, we obtain from Equation~\eqref{eq:model} and the sample paths $D_{N,n}$ the following regression model for the estimation of $\sigma^{2}$:
\begin{equation*}
  U^{j}_{k\Delta_n}=\sigma^{2}(X^{j}_{k\Delta_n})+\zeta^{j}_{k\Delta_n}+R^{j}_{k\Delta_n}, ~~ \forall (j,k)\in[\![1,N]\!]\times[\![0,n-1]\!]
\end{equation*}
where for each pair $(j,k)\in[\![1,N]\!]\times[\![0,n-1]\!]$, the response variable $U^{j}_{k\Delta_n}$ is given by
$$ U^{j}_{k\Delta_n} := \frac{\left(X^{j}_{(k+1)\Delta_n} - X^{j}_{k\Delta_n}\right)^2}{\Delta_n},$$
and the error terms are respectively given by $\zeta^{j}_{k\Delta}=\zeta^{j,1}_{k\Delta}+\zeta^{j,2}_{k\Delta}+\zeta^{j,3}_{k\Delta}$, with:
\begin{equation*}
    \zeta^{j,1}_{k\Delta_n}=\frac{1}{\Delta_n}\left[\left(\int_{k\Delta_n}^{(k+1)\Delta_n}{\sigma(X^{j}_{s})dW^{j}_{s}}\right)^2-\int_{k\Delta_n}^{(k+1)\Delta_n}{\sigma^{2}(X^{j}_{s})ds}\right],
\end{equation*}
\begin{equation*}
    \zeta^{j,2}_{k\Delta_n}=\frac{2}{\Delta_n}\int_{k\Delta_n}^{(k+1)\Delta_n}{((k+1)\Delta_n-s)\sigma^{\prime}(X^{j}_{s})\sigma^{2}(X^{j}_{s})dW^{j}_{s}},
\end{equation*}
\begin{equation*}
\zeta^{j,3}_{k\Delta_n}=2b(X^{j}_{k\Delta_n})\int_{k\Delta_n}^{(k+1)\Delta_n}{\sigma\left(X^{j}_{s}\right)dW^{j}_{s}},
\end{equation*}
and $R^{j}_{k\Delta}
=R^{j,1}_{k\Delta}+R^{j,2}_{k\Delta}+R^{j,3}_{k\Delta}$, with:
\begin{equation*}
    R^{j,1}_{k\Delta_n}=\frac{1}{\Delta_n}\left(\int_{k\Delta_n}^{(k+1)\Delta_n}{b(X^{j}_{s})ds}\right)^2, ~~~~ R^{j,3}_{k\Delta_n} = \frac{1}{\Delta_n}\int_{k\Delta_n}^{(k+1)\Delta_n}{((k+1)\Delta_n-s)\Phi(X^{j}_{s})ds}
\end{equation*}
\begin{equation*}
R^{j,2}_{k\Delta_n}=\frac{2}{\Delta_n}\left(\int_{k\Delta_n}^{(k+1)\Delta_n}{\left(b(X^{j}_{s})-b(X^{j}_{k\Delta_n})\right)ds}\right)\left(\int_{k\Delta_n}^{(k+1)\Delta_n}{\sigma(X^{j}_{s})dW^{j}_{s}}\right),
\end{equation*}
where 
\begin{equation*} \Phi:=2b\sigma^{\prime}\sigma+\left[\sigma^{\prime\prime}\sigma+\left(\sigma^{\prime}\right)^2\right]\sigma^{2}.
\end{equation*}
\end{lemme}
}


\textcolor{black}{
\begin{proof}[\textbf{Proof of Lemma~\ref{lm:RegModel}}]
    From Equation~\eqref{eq:model}, we have
    $$dX_t = b(X_t)d_t + \sigma(X_t)dW_t,$$
    and for all $(j,k) \in [\![1,N]\!] \times [\![0,n-1]\!]$,
    \begin{align*}
        \left(X^{j}_{(k+1)\Delta} - X^{j}_{k\Delta}\right)^2 = &~ \left(\int_{k\Delta}^{(k+1)\Delta}{b(X^{j}_s)ds} + \int_{k\Delta}^{(k+1)\Delta}{\sigma(X_s^j)dW_s^j}\right)^2 \\
        = &~ \left(\int_{k\Delta}^{(k+1)\Delta}{b(X_s^j)ds}\right)^2 + 2\left(\int_{k\Delta}^{(k+1)\Delta}{b(X_s^j)ds}\right)\left(\int_{k\Delta}^{(k+1)\Delta}{\sigma(X_s^j)dW_s^j}\right) \\
        &+ \left(\int_{k\Delta}^{(k+1)\Delta}{\sigma(X_s^j)dW_s^j}\right)^2 \\
        = &~ 2\left(\int_{k\Delta}^{(k+1)\Delta}{(b(X_s^j)-b(X^{j}_{k\Delta}))ds}\right)\left(\int_{k\Delta}^{(k+1)\Delta}{\sigma(X_s^j)dW_s^j}\right) \\
        &+ \left(\int_{k\Delta}^{(k+1)\Delta}{b(X_s^j)ds}\right)^2 + \left(\int_{k\Delta}^{(k+1)\Delta}{\sigma(X_s^j)dW_s^j}\right)^2 - \int_{k\Delta}^{(k+1)\Delta}{\sigma^{2}(X_s^j)ds} \\
        & + \int_{k\Delta}^{(k+1)\Delta}{\sigma^{2}(X_s^j)ds} + 2\Delta b(X^{j}_{k\Delta})\int_{k\Delta}^{(k+1)\Delta}{\sigma(X_s^j)dW_s^j}.
    \end{align*}
    Then, we deduce that
    \begin{equation}\label{eq:partial-sigma}
        \left(X^{j}_{(k+1)\Delta} - X^{j}_{k\Delta}\right)^2 = \int_{k\Delta}^{(k+1)\Delta}{\sigma^{2}(X_s^j)ds} + \left(\int_{k\Delta}^{(k+1)\Delta}{b(X_s^j)ds}\right)^2 + \Delta R^{j,2}_{k\Delta} + \Delta\zeta^{j,1}_{k\Delta} + \Delta\zeta^{j,3}_{k\Delta}.
    \end{equation}
    As $\sigma$ is a $\mathcal{C}^{2}$ function on $\mathbb{R}$, from the It\^o formula, one has for all $j \in [\![1,N]\!]$
    \begin{align*}
        d(\sigma^{2}(X_t^j)) = &~ 2\sigma^{\prime}(X_t^j)\sigma(X_t^j)dX_t^j + 2\left[\sigma^{\prime\prime} + (\sigma^{\prime})^2\right](X_t^j)dt \\
        = &~ \Phi(X_t^j)dt + 2\sigma^{\prime}(X_t^j)\sigma^{2}(X_t^j)dW_t^j
    \end{align*}
    with $\Phi = 2b\sigma^{\prime}\sigma+\left[\sigma^{\prime\prime}\sigma+\left(\sigma^{\prime}\right)^2\right]\sigma^{2}$. We deduce that for all $(j,k) \in [\![1,N]\!] \times [\![0,n-1]\!]$ and $s \in [k\Delta, (k+1)\Delta)$ tels que,
    \begin{align*}
        \sigma^{2}(X_s^j) = &~ \sigma^{2}(X^{j}_{k\Delta}) + \int_{k\Delta}^{s}{\Phi(X_u^j)du} + 2\int_{k\Delta}^{s}{\sigma^{\prime}(X_u^j)\sigma^{2}(X_u^j)dW_u^j}.
    \end{align*}
    We then integrate on the interval $[k\Delta, (k+1)\Delta)$ and we obtain 
    \begin{align*}
        \int_{k\Delta}^{(k+1)\Delta}{\sigma^{2}(X_s^j)ds} = &~ \Delta\sigma^{2}(X^{j}_{k\Delta}) + \int_{k\Delta}^{(k+1)\Delta}{\int_{k\Delta}^{s}{\Phi(X_u^j)duds}} \\
        &+ 2\int_{k\Delta}^{(k+1)\Delta}{\int_{k\Delta}^{s}{\sigma^{\prime}(X_u^j)\sigma^{2}(X_u^j)dW_u^jds}} \\
        = &~ \Delta\sigma^{2}(X^{j}_{k\Delta}) + \int_{k\Delta}^{(k+1)\Delta}{\Phi(X_u^j)\int_{u}^{(k+1)\Delta}{dsdu}} \\
        &+ 2\int_{k\Delta}^{(k+1)\Delta}{\sigma^{\prime}(X_u^j)\sigma^{2}(X_u^j)\int_{u}^{(k+1)\Delta}{dsdW_u^j}} \\
        = &~ \Delta\sigma^{2}(X^{j}_{k\Delta}) + \int_{k\Delta}^{(k+1)\Delta}{[(k+1)\Delta - u]\Phi(X_u^j)du} \\
        &+ 2\int_{k\Delta}^{(k+1)\Delta}{[(k+1)\Delta - u]\sigma^{\prime}(X_u^j)\sigma(X_u^j)dW_u^j}.
    \end{align*}
    Thus, we obtain for al $(j,k) \in [\![1,N]\!] \times [\![0,n-1]\!]$,
    \begin{equation}\label{eq:partial2-sigma}
        \int_{k\Delta}^{(k+1)\Delta}{\sigma^{2}(X_s^j)ds} = \Delta\sigma^{2}(X^{j}_{k\Delta}) + \int_{k\Delta}^{(k+1)\Delta}{[(k+1)\Delta - u]\Phi(X_u^j)du} + \zeta^{j,2}_{k\Delta}
    \end{equation}
    We finally deduce from Equations~\eqref{eq:partial-sigma} et \eqref{eq:partial2-sigma} that
    \begin{align*}
        \frac{\left(X^{j}_{(k+1)\Delta} - X^{j}_{k\Delta}\right)^2}{\Delta} = \sigma^{2}(X^{j}_{k\Delta}) + \zeta^{j}_{k\Delta} + R^{j}_{k\Delta}.
    \end{align*}
\end{proof}
}
\end{document}